\documentclass{article}
\usepackage[utf8]{inputenc}
\usepackage{bm,amsfonts,amsthm,amsmath,amssymb}
\usepackage{graphicx,color}
\usepackage{mathrsfs}
\usepackage{indentfirst}
\usepackage{bbm}
\usepackage{appendix}

\DeclareMathAlphabet{\mathpzc}{OT1}{pzc}{m}{it}

\def\eqdefa{\buildrel\hbox{\footnotesize def}\over =}
\newcommand{\ve}{\varepsilon}
\newcommand{\ud}{\mathrm{d}}

\newcommand{\vv}{\mathbf{v}}

\newcommand{\ww}{\mathbf{w}}
\newcommand{\aaa}{\mathbf{a}}
\newcommand{\xx}{\mathbf{x}}
\newcommand{\yy}{\mathbf{y}}
\newcommand{\nn}{\mathbf{n}}

\newcommand{\hh}{\mathbf{h}}

\newcommand{\sss}{\mathbf{s}}

\newcommand{\A}{\mathbf{A}}

\newcommand{\FF}{\mathbf{F}}

\newcommand{\CR}{\mathcal{R}}
\newcommand{\CB}{\mathcal{B}}
\newcommand{\CD}{\mathcal{D}}

\newcommand{\CE}{\mathcal{E}}

\newcommand{\CF}{\mathcal{F}}

\newcommand{\CG}{\mathcal{G}}
\newcommand{\CU}{\mathcal{U}}

\newcommand{\CA}{\mathcal{A}}
\newcommand{\CP}{\mathcal{P}}
\newcommand{\CH}{\mathcal{H}}
\newcommand{\CL}{\mathcal{L}}

\newcommand{\CK}{\mathcal{K}}
\newcommand{\CJ}{\mathcal{J}}

\newcommand{\ML}{\mathscr{L}}

\newcommand{\MF}{\mathscr{F}}

\newcommand{\Fi}{\mathfrak{i}}
\newcommand{\Fp}{\mathfrak{p}}

\newcommand{\BOm}{\mathbf{\Omega}}

\newtheorem{theorem}{Theorem}
\newtheorem{proposition}[theorem]{Proposition}
\newtheorem{lemma}[theorem]{Lemma}

\numberwithin{theorem}{section}
\numberwithin{equation}{section}

\allowdisplaybreaks[4]
\title{Well-posedness of frame hydrodynamics for biaxial nematic liquid crystals}
\author{Sirui Li\footnote{ School of Mathematics and Statistics, Guizhou University, Guiyang 550025, China (srli@gzu.edu.cn) },
Chenchen Wang\footnote{ School of Mathematics and Statistics, Guizhou University, Guiyang 550025, China (ccwangmath@163.com) },
Jie Xu\footnote{LSEC and NCMIS, Institute of Computational Mathematics and Scientific/Engineering Computing (ICMSEC), Academy of Mathematics and Systems Science (AMSS), Chinese Academy of Sciences, Beijing, China (xujie@lsec.cc.ac.cn)}}

\date{}

\begin{document}
\maketitle
\begin{abstract}
We consider the hydrodynamics for the biaxial nematic phase characterized by a field of orthonormal frame, which can be derived from a molecular-theory-based tensor model.
In dimension two and three, we establish the local well-posedness and the blow-up criterion for smooth solutions to the frame hydrodynamic model. Furthermore, we prove the global existence of weak solutions in $\mathbb{R}^2$ which are nonsmooth at finitely many singular times.


\end{abstract}

%\tableofcontents

\section{Introduction}
The focus of this paper is the hydrodynamics for the biaxial nematic phase. The word `biaxial' represents the symmetry of the local anisotropy in liquid crystals, typically formed by nonuniform  orientational distribution of non-spherical rigid molecules.
The local anisotropy that we are more familiar with is the uniaxial nematic phase, for which the word `uniaxial' is often omitted.
For the uniaxial nematic ordering, the local anisotropy is axisymmetric, so that it can be described by a unit vector field $\nn\in\mathbb{S}^2$.
However, there are much more possibilities of local anisotropy that may not have axisymmetry. In this case, it is necessary to utilize an orthonormal frame $\Fp\in SO(3)$ to take the place of the unit vector.

The hydrodynamics for the uniaxial nematic phase is called the Ericksen--Leslie theory \cite{E-61,E-91,Les}, which is a coupled system between the evolution of the unit vector field $\nn(\xx)$ and the Navier--Stokes equation.
The Ericksen--Leslie theory has been studied extensively both analytically and numerically (see \cite{Lin4,WZZ4,Ball} and the references therein).
The connection between the Ericksen--Leslie model and the models at other levels, including the Doi--Onsager model \cite{KD,EZ,WZZ2} and dynamic $Q$-tensor models \cite{HLWZZ,LWZ,LW,WZZ3}, is also discussed.
Let us briefly review the works on the well-posedness of the Ericksen--Leslie model.
In $\mathbb{R}^2$, the existence of global weak solutions
\cite{Lin2, Hong,HX,HLW,WW} and the uniqueness \cite{Lin5,WWZ-zju} are discussed, respectively.
For special initial data, the global existence of weak solutions in $\mathbb{R}^3$ is studied in \cite{Lin3}.
The well-posedness of smooth solutions to the general Ericksen--Leslie model has been established for the whole space \cite{WZZ1,WW} and for bounded domains \cite{HNPS}.
On the other hand, when an inertial term is added, the well-posedness of smooth solutions is also considered \cite{JL,CW}.
The existence of finite time singularity of short time smooth solutions has also been established in $\mathbb{R}^3$ \cite{HLLW} and $\mathbb{R}^2$ \cite{LLWWZ}, respectively.
For more analytic results, we can refer the review papers \cite{Lin4,WZZ4,Ball} and the references therein.

We turn to the biaxial nematic phase. Since the axisymmetry of local anisotropy is broken, the biaxial hydrodynamics calls for an evolution equation of an orthonormal frame field $\Fp(\xx)$.
Based on the symmetry of the biaxial nematic phase, the form of orientational elasticity  \cite{S-A,GV1,SV} and hydrodynamics \cite{S-A,Liu-M,BP,S-W-M,GV2,LLC} have been written down in various variables and forms that should be equivalent.
The model, in its full form, possesses over twenty phenomenological coefficients. These coefficients are derived from physical parameters in a recent work \cite{LX}.
Starting from a molecular-theory-based dynamical tensor model \cite{XZ} with multiple tensors, the biaxial hydrodynamics is derived using the Hilbert expansion.
As a result, the biaxial hydrodynamics is expressed as the evolution equation of coordinates of the frame field $\Fp(\xx)$, and the coefficients in the model are expressed by those in the tensor model that are derived from molecular parameters.
The energy dissipation law is inherited from the tensor model.
Moreover, the Ericksen--Leslie model can be recovered as a special case.

While the biaxial hydrodynamics has been established decades before, the analytical works are rather lacking.
To our knowledge, only an over-simplified model is studied in \cite{LLWa}, and no work for the full form of biaxial hydrodynamics is found.

The main goal of this paper is to prove the local well-posedness and the blow-up criterion of smooth solutions to the new frame hydrodynamic system in $\mathbb{R}^d(d=2,3)$, and to show the global existence of weak solutions in $\mathbb{R}^2$.
Although some ideas of dealing with the Ericksen--Leslie model can be utilized, many essential difficulties occur as we shall explain below.

As we have mentioned above, the biaxial hydrodynamics has been written in various forms where the variables are also different.
To established the well-posedness results, the most important point is to choose the appropriate variables and the best form for deriving the estimates needed.
As it turns out, we find that the form given in \cite{LX}, expressed by the coordinates of the frame field, would be convenient for us to obtain the corresponding estimates.
Specifically, we need to carefully choose among the equivalent forms of the orientational elasticity.
Such a form has good cancellation structures for both the basic energy dissipation law and the estimates on higher order spatial derivatives.
In contrast, the other forms given previously may not be suitable for analytical study, since they either appear the equations of motion for the three Euler angles \cite{GV2}, or contain more complex relations between coefficients in the phenomenological constitutive laws \cite{LLC}.

In order to prove the well-posedness of the biaxial hydrodynamics, we need to construct a suitable approximate system preserving the energy dissipation.
Nevertheless, this is a difficult task if one only keeps equations of three fundamental degrees of freedom of the orthonormal frame, since the constraint that $\Fp(\xx)\in SO(3)$ is an orthonormal frame needs to be  explicitly imposed (see (\ref{frame-SO3})).
To overcome this difficulty, it is necessary to utilize the equations of all the coordinates of $\Fp(\xx)$ (see (\ref{new-frame-equation-n1})--(\ref{imcompressible-v})), which keep the orthonormal constraint implicitly.
We can thus construct the corresponding approximate system (see (\ref{n1-ve-equation})--(\ref{initial-condition})) which maintains the energy dissipation without worrying about the orthonormal constraint.
A similar construction method has been applied in \cite{WZZ1} for the Ericksen--Leslie model.

Of all the estimates needed for the establishment of well-posedness,  the major difficulty comes from the nonlinear terms with the rotational derivatives on $SO(3)$ in the energy dissipative law.
To deal with these terms, we need to rewrite the elastic energy (see \eqref{new-elasitic-density}) to facilitate higher order energy estimates.
To actually acquire higher order energy estimates (see, for instance, (\ref{higher-deriv-estimate}) or Proposition \ref{Fp-v-4-prop}), another key point is to utilize a decomposition according to the tangential space at a point on $SO(3)$ and its orthogonal complement.
The decomposition will be introduced shortly afterwards.

Before introducing the frame hydrodynamics, let us prepare some notations for orthonormal frames and tensors. The orthonormal frame consists of three mutually perpendicular unit vectors, which we denote explicitly as $\Fp=(\nn_1,\nn_2,\nn_3)$.
As for operations of tensors, the symbol $\otimes$ represents the tensor product, and the dot product between two tensors of the same order is defined by the summation of the product of corresponding coordinates,
\begin{align*}
    U\cdot V=U_{i_1\ldots i_n}V_{i_1\ldots i_n},\quad |U|^2=U\cdot U.
\end{align*}
Hereafter, we adopt the Einstein summation convention on repeated indices.

The monomial notation will be utilized to express symmetric tensors: if the symbol $\otimes$ in a product is omitted, it implies that a tensor is symmetrized. For instance, for the frame $\Fp=(\nn_1,\nn_2,\nn_3)\in SO(3)$ and  $\alpha,\beta=1,2,3$, we denote
\begin{align*}
\nn^2_{\alpha}=\nn_{\alpha}\otimes\nn_{\alpha},\quad
\nn_{\alpha}\nn_{\beta}=\frac{1}{2}(\nn_{\alpha}\otimes\nn_{\beta}+\nn_{\beta}\otimes\nn_{\alpha}),\quad \alpha\neq \beta.
\end{align*}
A useful equality is that the second-order identity tensor $\Fi$ satisfies
\begin{align*}
\Fi=\nn^2_1+\nn^2_2+\nn^2_3.
\end{align*}

The hydrodynamics involves the differential operators on $SO(3)$. For any frame $\Fp=(\nn_1,\nn_2,\nn_3)\in SO(3)$, the tangential space of the three-dimensional manifold $SO(3)$ at a point $\Fp$ is denoted by $T_{\Fp}SO(3)$, which can be spanned by the following orthogonal basis:
\begin{align*}
V_1=(0,\nn_3,-\nn_2),\quad V_2=(-\nn_3,0,\nn_1),\quad V_3=(\nn_2,-\nn_1,0).
\end{align*}
Then, the associated orthogonal complement space can be spanned by
\begin{align*}
&W_1=(0,\nn_3,\nn_2),\quad W_2=(\nn_3,0,\nn_1),\quad W_3=(\nn_2,\nn_1,0),\\
&W_4=(\nn_1,0,0),\quad W_5=(0,\nn_2,0),\quad W_6=(0,0,\nn_3).
\end{align*}
As a consequence, we can define three differential operators $\ML_k(k=1,2,3)$ on $T_{\Fp}SO(3)$ by the inner products of the orthogonal basis $V_k$ and $\partial/\partial\Fp=(\partial/\partial\nn_1,\partial/\partial\nn_2,\partial/\partial\nn_3)$, i.e.,
\begin{align}\label{ML-123}
\left\{
\begin{array}{l}
\ML_1\eqdefa V_1\cdot\frac{\partial}{\partial\Fp}=\nn_3\cdot\frac{\partial}{\partial\nn_2}-\nn_2\cdot\frac{\partial}{\partial\nn_3},\vspace{0.5ex}\\
\ML_2\eqdefa V_2\cdot\frac{\partial}{\partial\Fp}=\nn_1\cdot\frac{\partial}{\partial\nn_3}-\nn_3\cdot\frac{\partial}{\partial\nn_1},\vspace{0.5ex}\\
\ML_3\eqdefa V_3\cdot\frac{\partial}{\partial\Fp}=\nn_2\cdot\frac{\partial}{\partial\nn_1}-\nn_1\cdot\frac{\partial}{\partial\nn_2},
\end{array}
  \right.
\end{align}
where the subscript implies that the differential operator acts on the infinitesimal rotation about $\nn_k(k=1,2,3)$, which can be verified by the relation $\ML_k\nn_p=\epsilon^{kpq}\nn_q$ with $\epsilon^{kpq}$ being the Levi-Civita symbol.
We also use the operators $\ML_k$ on functionals, where the partial derivative $\partial/\partial\Fp$ shall be replaced by variational derivative $\delta/\delta\Fp$.

For any two matrices $A, B\in \mathbb{R}^{3\times3}$, the associated inner product $A\cdot B$ has the following orthogonal decomposition:
\begin{align}\label{AB-orthogonal-decomposition}
A\cdot B=\sum^3_{k=1}\frac{1}{|V_k|^2}(A\cdot V_k)(B\cdot V_k)+\sum^6_{k=1}\frac{1}{|W_k|^2}(A\cdot W_k)(B\cdot W_k).
\end{align}
In addition, for any first-order differential operator $\CD$ and $\alpha,\beta=1,2,3$, it follows that
\begin{align}\label{relation-CD-Fp}
\left\{
\begin{array}{l}
\CD\nn_1=(\CD\nn_1\cdot\nn_2)\nn_2+(\CD\nn_1\cdot\nn_3)\nn_3,\vspace{0.5ex}\\
\CD\nn_2=(\CD\nn_2\cdot\nn_1)\nn_1+(\CD\nn_2\cdot\nn_3)\nn_3,\vspace{0.5ex}\\
\CD\nn_3=(\CD\nn_3\cdot\nn_1)\nn_1+(\CD\nn_3\cdot\nn_2)\nn_2,\vspace{0.5ex}\\
\CD\nn_{\alpha}\cdot\nn_{\beta}+\CD\nn_{\beta}\cdot\nn_{\alpha}=\CD(\nn_{\alpha}\cdot\nn_{\beta})=0,\vspace{0.5ex}\\
W_\alpha\cdot\CD \Fp=0.
\end{array}
  \right.
\end{align}

\subsection{Frame hydrodynamics}

We first write down the orientational elasticity for the biaxial nematic phases, which is a result of the symmetry of local anisotropy. Let us denote the elastic energy as
\begin{align}\label{elastic-energy}
\CF_{Bi}[\Fp]=\int_{\mathbb{R}^d}f_{Bi}(\Fp,\nabla\Fp)\ud\xx,
\end{align}
where the deformation free energy density $f_{Bi}$, being a function of the spatial derivatives of the orthonormal frame $\Fp$, has the following form \cite{SV,GV1,Xu2},
\begin{align*}
f_{Bi}(\Fp,\nabla\Fp)=&\,\frac{1}{2}\Big(K_{1}(\nabla\cdot\nn_{1})^{2}+K_{2}(\nabla\cdot\nn_{2})^{2}+K_{3}(\nabla\cdot\nn_{3})^{2}\\
&\,+K_{4}(\nn_{1}\cdot\nabla\times\nn_{1})^{2}+K_{5}(\nn_{2}\cdot\nabla\times\nn_{2})^{2}+K_{6}(\nn_{3}\cdot\nabla\times\nn_{3})^{2}\\
&\,+K_{7}(\nn_{3}\cdot\nabla\times\nn_{1})^{2}+K_{8}(\nn_{1}\cdot\nabla\times\nn_{2})^{2}+K_{9}(\nn_{2}\cdot\nabla\times\nn_{3})^{2}\\
&\,+K_{10}(\nn_{2}\cdot\nabla\times\nn_{1})^2+K_{11}(\nn_{3}\cdot\nabla\times\nn_{2})^2+K_{12}(\nn_{1}\cdot\nabla\times\nn_{3})^2\\
&\,+\gamma_{1}\nabla\cdot[(\nn_{1}\cdot\nabla)\nn_{1}-(\nabla\cdot\nn_{1})\nn_{1}]+\gamma_{2}\nabla\cdot[(\nn_{2}\cdot\nabla)\nn_{2}-(\nabla\cdot\nn_{2})\nn_{2}]\\
&\,+\gamma_{3}\nabla\cdot[(\nn_{3}\cdot\nabla)\nn_{3}-(\nabla\cdot\nn_{3})\nn_{3}]\Big).
\end{align*}
The elasticity above involves twelve bulk terms and three surface terms.
The coefficients $K_i(i=1,\cdots,12)$ of bulk terms are all non-negative. Each surface term is a null Lagrangian when the frame $\Fp$ at infinity is prescribed by some given value of the frame field, and contributes nothing in variational derivatives, so that we could choose their coefficients $\gamma_i>0(i=1,2,3)$ according to our demand, which we specify later.
Note that the orientational elasticity has multiple equivalent forms \cite{SV,GV1,Xu2}.
The form we give above is one of the most convenient for estimates later.

In order to describe the frame hydrodynamics, we introduce a set of local basis formed by nine second-order tensors: the identity tensor $\Fi$, five symmetric traceless tensors,
\begin{align*}%\label{sss-five}
    \sss_1=\nn^2_1-\frac13\Fi,\quad \sss_2=\nn^2_2-\nn^2_3,\quad \sss_3=\nn_1\nn_2,\quad \sss_4=\nn_1\nn_3,\quad \sss_5=\nn_2\nn_3,
\end{align*}
and three asymmetric traceless tensors,
\begin{align*}%\label{aaa-three}
    \aaa_1=\nn_1\otimes\nn_2-\nn_2\otimes\nn_1,\quad \aaa_2=\nn_3\otimes\nn_1-\nn_1\otimes\nn_3,\quad \aaa_3=\nn_2\otimes\nn_3-\nn_3\otimes\nn_2.
\end{align*}

The frame hydrodynamics consists of an evolution equation for the orthonormal frame field $\Fp(\xx)=\big(\nn_1(\xx),\nn_2(\xx),\nn_3(\xx)\big)$, coupled with the Navier--Stokes equations for the fluid velocity field $\vv(\xx)$.
Let us write it down according to the form in \cite{LX}:
\begin{align}
&\,\chi_1\dot{\nn}_2\cdot\nn_3-\frac{1}{2}\chi_1\BOm\cdot\aaa_3-\eta_1\A\cdot\sss_5+\ML_1\CF_{Bi}=0,\label{frame-equation-n1}\\
&\,\chi_2\dot{\nn}_3\cdot\nn_1-\frac{1}{2}\chi_2\BOm\cdot\aaa_2-\eta_2\A\cdot\sss_4+\ML_2\CF_{Bi}=0,\label{frame-equation-n2}\\
&\,\chi_3\dot{\nn}_1\cdot\nn_2-\frac{1}{2}\chi_3\BOm\cdot\aaa_1-\eta_3\A\cdot\sss_3+\ML_3\CF_{Bi}=0,\label{frame-equation-n3}\\
&\,\Fp=(\nn_1,\nn_2,\nn_3)\in SO(3),\label{frame-SO3}\\
&\,\dot{\vv}=-\nabla p+\eta\Delta\vv+\nabla\cdot\sigma+\mathfrak{F},\label{yuan-frame-equation-v}\\
&\,\nabla\cdot\vv=0,\label{yuan-imcompressible-v}
\end{align}
where $\dot{f}=\partial_tf+\vv\cdot\nabla f$ stands for the material derivative, and  $\ML_k\CF_{Bi} (k=1,2,3)$ are the variational derivatives along the infinitesimal rotation round $\nn_k (k=1,2,3)$, respectively.
The constraint condition (\ref{frame-SO3}) is necessary since we only have equations for three scalars in $\Fp$.
In the equation of the velocity $\vv=(v_1,v_2,v_3)^T$, the pressure $p$ ensures the incompressible condition (\ref{yuan-imcompressible-v}), $\eta$ is the viscous coefficient.
The body force $\mathfrak{F}$ is defined by
\begin{align}\label{external-force-F}
\mathfrak{F}_i=\partial_i\nn_1\cdot\nn_2\ML_3\CF_{Bi}+\partial_i\nn_3\cdot\nn_1\ML_2\CF_{Bi}+\partial_i\nn_2\cdot\nn_3\ML_1\CF_{Bi}.
\end{align}
To write down the stress, we denote by $\A$ and $\BOm$ the
symmetric and skew-symmetric parts of the velocity gradient $\kappa_{ij}=\partial_jv_i$, respectively, i.e.,
\begin{align*}
\A=\frac{1}{2}(\kappa+\kappa^T),\quad\BOm=\frac{1}{2}(\kappa-\kappa^T).
\end{align*}
The stress $\sigma=\sigma(\Fp,\vv)$ is given by
\begin{align}\label{sigma-e}
\sigma(\Fp,\vv)=&\,\beta_1(\A\cdot\sss_1)\sss_1+\beta_0(\A\cdot\sss_2)\sss_1+\beta_0(\A\cdot\sss_1)\sss_2+\beta_2(\A\cdot\sss_2)\sss_2\nonumber\\
&\,+\beta_3(\A\cdot\sss_3)\sss_3-\eta_3\Big(\dot{\nn}_1\cdot\nn_2-\frac{1}{2}\BOm\cdot\aaa_1\Big)\sss_3\nonumber\\
&\,+\beta_4(\A\cdot\sss_4)\sss_4-\eta_2\Big(\dot{\nn}_3\cdot\nn_1-\frac{1}{2}\BOm\cdot\aaa_2\Big)\sss_4\nonumber\\
&\,+\beta_5(\A\cdot\sss_5)\sss_5-\eta_1\Big(\dot{\nn}_2\cdot\nn_3-\frac{1}{2}\BOm\cdot\aaa_3\Big)\sss_5\nonumber\\
&\,+\frac{1}{2}\eta_3(\A\cdot\sss_3)\aaa_1-\frac{1}{2}\chi_3\Big(\dot{\nn}_1\cdot\nn_2-\frac{1}{2}\BOm\cdot\aaa_1\Big)\aaa_1\nonumber\\
&\,+\frac{1}{2}\eta_2(\A\cdot\sss_4)\aaa_2-\frac{1}{2}\chi_2\Big(\dot{\nn}_3\cdot\nn_1-\frac{1}{2}\BOm\cdot\aaa_2\Big)\aaa_2\nonumber\\
&\,+\frac{1}{2}\eta_1(\A\cdot\sss_5)\aaa_3-\frac{1}{2}\chi_1\Big(\dot{\nn}_2\cdot\nn_3-\frac{1}{2}\BOm\cdot\aaa_3\Big)\aaa_3,
\end{align}
where the coefficients in (\ref{sigma-e}) can be derived from molecular parameters and satisfy the following nonnegative definiteness conditions (see {\rm \cite{LX}} for details):
\begin{align}\label{coefficient-conditions}
\left\{
\begin{array}{l}
\beta_i\geq0,~i=1,\cdots,5,\quad \chi_j>0,~j=1,2,3,\quad \eta>0,\vspace{1ex}\\
\beta^2_0\leq\beta_1\beta_2,~~\eta^2_1\leq\beta_5\chi_1,~~\eta^2_2\leq\beta_4\chi_2,~~\eta^2_3\leq\beta_3\chi_3.
\end{array}
  \right.
\end{align}
The relations in (\ref{coefficient-conditions}) ensure that the biaxial hydrodynamics has a basic energy dissipation law (see Proposition \ref{energ-diss-prop}).

For convenience on the study of well-posedness, we derive the equations of ${\nn}_i(i=1,2,3)$ from the frame equations (\ref{frame-equation-n1})--(\ref{frame-equation-n3}) by using the relation (\ref{relation-CD-Fp}), for instance,
\begin{align*}
\dot{\nn}_1=&(\dot{\nn}_1\cdot\nn_2)\nn_2+(\dot{\nn}_1\cdot\nn_3)\nn_3\\
=&\Big(\frac{1}{2}\BOm\cdot\aaa_1+\frac{\eta_3}{\chi_3}\A\cdot\sss_3-\frac{1}{\chi_3}\ML_3\CF_{Bi}\Big)\nn_2\\
&-\Big(\frac{1}{2}\BOm\cdot\aaa_2+\frac{\eta_2}{\chi_2}\A\cdot\sss_4-\frac{1}{\chi_2}\ML_2\CF_{Bi}\Big)\nn_3.
\end{align*}
Therefore, the frame hydrodynamics (\ref{frame-equation-n1})--(\ref{yuan-imcompressible-v}) can be reformulated as
\begin{align}
\dot{\nn}_1=&\Big(\frac{1}{2}\BOm\cdot\aaa_1+\frac{\eta_3}{\chi_3}\A\cdot\sss_3-\frac{1}{\chi_3}\ML_3\CF_{Bi}\Big)\nn_2\nonumber\\
&-\Big(\frac{1}{2}\BOm\cdot\aaa_2+\frac{\eta_2}{\chi_2}\A\cdot\sss_4-\frac{1}{\chi_2}\ML_2\CF_{Bi}\Big)\nn_3,\label{new-frame-equation-n1}\\
\dot{\nn}_2=&-\Big(\frac{1}{2}\BOm\cdot\aaa_1+\frac{\eta_3}{\chi_3}\A\cdot\sss_3-\frac{1}{\chi_3}\ML_3\CF_{Bi}\Big)\nn_1\nonumber\\
&+\Big(\frac{1}{2}\BOm\cdot\aaa_3+\frac{\eta_1}{\chi_1}\A\cdot\sss_5-\frac{1}{\chi_1}\ML_1\CF_{Bi}\Big)\nn_3,\label{new-frame-equation-n2}\\
\dot{\nn}_3=&\Big(\frac{1}{2}\BOm\cdot\aaa_2+\frac{\eta_2}{\chi_2}\A\cdot\sss_4-\frac{1}{\chi_2}\ML_2\CF_{Bi}\Big)\nn_1\nonumber\\
&-\Big(\frac{1}{2}\BOm\cdot\aaa_3+\frac{\eta_1}{\chi_1}\A\cdot\sss_5-\frac{1}{\chi_1}\ML_1\CF_{Bi}\Big)\nn_2,\label{new-frame-equation-n3}\\
\dot{\vv}=&-\nabla p+\eta\Delta\vv+\nabla\cdot\sigma+\mathfrak{F},\label{frame-equation-v}\\
\nabla\cdot\vv=&0.\label{imcompressible-v}
\end{align}
It is easy to verify that the above system automatically satisfy $(\nn_1,\nn_2,\nn_3)\in SO(3)$, so that the explicit constraint can be discarded.

Another remark is that compared with the original model in \cite{LX}, we set several quantities to special values, such as letting the concentration, as well as the product of Boltzmann constant and the absolute temperature, be equal to one, which makes no difference in the following analysis.

\subsection{Main results}

For any frame $\Fp=(\nn_1,\nn_2,\nn_3)\in SO(3)$, we denote
\begin{align*}
|\nabla^{j}\Fp|^{k}=\sum^3_{\alpha=1}|\nabla^{j}\nn_{\alpha}|^{k}.
\end{align*}

Firstly, we give the local well-posedness result and the blow-up criterion of smooth solutions to the biaxial frame system \eqref{new-frame-equation-n1}--\eqref{imcompressible-v}. More precisely, we will show the following theorem.
%The first result is stated as follows.
\begin{theorem}\label{local-posedness-theorem}
Let $s\geq2$ be an integer. Assume that $(\nabla\Fp^{(0)},\vv^{(0)})\in H^{2s}(\mathbb{R}^d)\times H^{2s}(\mathbb{R}^d)(d=2~\text{or}~3)$ is the given initial data satisfying  $\nabla\cdot\vv^{(0)}=0$ and $\Fp^{(0)}=\big(\nn^{(0)}_1,\nn^{(0)}_2,\nn^{(0)}_3\big)\in SO(3)$. Then, there exists $T>0$ and a unique solution $(\Fp,\vv)$ to the biaxial frame system \eqref{new-frame-equation-n1}--\eqref{imcompressible-v} such that
\begin{align*}
\nabla\Fp\in C\big([0,T];H^{2s}(\mathbb{R}^d)\big),\quad \vv\in C\big([0,T];H^{2s}(\mathbb{R}^d)\big)\cap L^2\big(0,T;H^{2s+1}(\mathbb{R}^d)\big).
\end{align*}
Let $T^*$ be the maximal existence time of the solution. If $T^*<+\infty$, then it is necessary that
\begin{align*}
\int^{T^*}_0\big(\|\nabla\times\vv(t)\|_{L^{\infty}}+\|\nabla\Fp\|^2_{L^{\infty}}\big)\ud t=+\infty,\quad \|\nabla\Fp\|^2_{L^{\infty}}=\sum^3_{i=1}\|\nabla\nn_i\|^2_{L^{\infty}}.
\end{align*}
\end{theorem}

Armed with the reformulation \eqref{new-frame-equation-n1}--\eqref{new-frame-equation-n3}, it is straightforward to construct an approximate system maintaining the energy dissipative relation.
We thus can prove the local well-posedness following the idea in \cite{WZZ1,WW}.
The constraint $\Fp\in SO(3)$ can be inherently attained in the passage of limit to the original PDE system.
For the blow-up criterion, the key ingredient is to estimate the higher order derivative terms with respect to $\ML_k\CF_{Bi}(k=1,2,3)$. We circumvent the difficulty using the orthogonal decomposition \eqref{AB-orthogonal-decomposition} to give the lower bound of the higher order estimate for the dissipation function, that is,
\begin{align}\label{Higher-order-key-estimate}
\int_{\mathbb{R}^d}\sum^3_{k=1}\frac{1}{\chi_k}|\Delta^s(\ML_k\CF_{Bi})|^2\ud\xx\geq \frac{2\gamma^2}{\chi}\int_{\mathbb{R}^d}|\Delta^{s+1}\Fp|^2\ud\xx+\text{lower order terms},
\end{align}
where $\gamma=\min\{\gamma_1,\gamma_2,\gamma_3\}>0$ and $\chi=\max\{\chi_1,\chi_2,\chi_3\}>0$.

The second result of this paper is concerned with the global existence of weak solutions to the biaxial frame system (\ref{new-frame-equation-n1})--(\ref{imcompressible-v}) in $\mathbb{R}^2$. Here, we assume that the solution $(\Fp,\vv)$ satisfies $v_3=0,\frac{\partial\vv}{\partial x_3}=0$ and $\frac{\partial\Fp}{\partial x_3}=\big(\frac{\partial\nn_1}{\partial x_3},\frac{\partial\nn_2}{\partial x_3},\frac{\partial\nn_3}{\partial x_3}\big)=\mathbf{0}$. Moreover, $\nabla\cdot \sigma$ should be understood as (see \cite{WW,HLW})
\begin{align*}
\left(
  \begin{array}{ccc}
    1 & 0& 0 \\
    0 & 1& 0 \\
    0 & 0& 0
  \end{array}
\right)
\cdot(\nabla\cdot\sigma).
\end{align*}
For any given constant orthonormal frame $\Fp^{*}=(\nn^{*}_1,\nn^{*}_2,\nn^{*}_3)\in SO(3)$, we denote
\begin{align*}
    H^1_{\Fp^{*}}\big(\mathbb{R}^2,SO(3)\big)\eqdefa\big\{\Fp=(\nn_1,\nn_2,\nn_3): \Fp-\Fp^{*}\in H^1(\mathbb{R}^2;\mathbb{R}^3),~|\nn_i|=1~{\rm a.e.~ in}~\mathbb{R}^2, i=1,2,3\big\}.
\end{align*}

The second result is stated as follows.
\begin{theorem}\label{global-posedness-theorem}
Let $(\Fp^{(0)}, \vv^{(0)})\in H^1_{\Fp^{*}}\big(\mathbb{R}^2,SO(3)\big)\times L^2(\mathbb{R}^2,\mathbb{R}^2)$ be given initial data with $\nabla\cdot\vv^{(0)}=0$ and $\Fp^{(0)}=\big(\nn^{(0)}_1,\nn^{(0)}_2,\nn^{(0)}_3\big)\in SO(3)$. Then, there exists a global weak solution $(\Fp,\vv):\mathbb{R}^2\times[0,+\infty)\rightarrow SO(3)\times\mathbb{R}^2$ of the biaxial frame system {\rm (\ref{new-frame-equation-n1})--(\ref{imcompressible-v})} such that the solution $(\Fp,\vv)$ is smooth in $\mathbb{R}^2\times((0,+\infty)\setminus\{T_{l}\}^L_{l=1})$ for a finite number of times $\{T_{l}\}^L_{l=1}$. Furthermore, there exists two constants $\ve_0>0$ and $R_0>0$ such that each singular point $(x^l_i,T_l)$ is characterized by the condition
\begin{align*}
    \limsup_{t\nearrow T_l}\int_{B_R(x^l_i)}\big(|\nabla\Fp|^2+|\vv|^2\big)(\cdot,t)\ud t>\ve_0,\quad |\nabla\Fp|^2=\sum^3_{i=1}|\nabla\nn_i|^2,
\end{align*}
for any $R>0$ with $R\leq R_0$.
\end{theorem}

The main idea to prove Theorem \ref{global-posedness-theorem} is originated from the seminal works in \cite{Struwe,Strume2}, which is further developed in the later works \cite{Hong,HX,WW} for the Ericksen--Leslie model.
In order to prove the global existence of weak solutions in $\mathbb{R}^2$, the main steps are to obtain certain local energy inequalities and interior regularity estimates under the small energy condition (see Proposition \ref{Fp-v-4-prop}--\ref{higher-regularity-l-prop}). To complete these steps, the major task is still to deduce the lower bound of the higher order estimate for the dissipation similar to (\ref{Higher-order-key-estimate}).

The remainder of this paper are organized as follows. The Section \ref{energy-law-section} is devoted to the derivation for the basic energy dissipative law of the frame hydrodynamics (\ref{new-frame-equation-n1})--(\ref{imcompressible-v}), and the presentation for the algebraic structures of variational derivatives on the frame $\Fp\in SO(3)$. In Section \ref{local-well-posed-blowup-section}, the existence and uniqueness, and the blow-up criterion for the smooth solution to the biaxial frame system (\ref{new-frame-equation-n1})--(\ref{imcompressible-v}) are successively established, based on the classical Friedrich's approach and energy estimates.
The aim of the Section \ref{global-weak-solution-section} is to establish the global existence of weak solutions in dimension two. The dissipated energy estimate of the higher order derivative term (see Proposition \ref{Fp-v-4-prop}), which plays an important role in our analysis, will be provided.

\section{Energy dissipative law}\label{energy-law-section}

In this section, we will derive the basic energy dissipation law and clarify some useful algebraic structures of variational derivatives with respect to the frame $\Fp=(\nn_1,\nn_2,\nn_3)\in SO(3)$.

For each pair $(\Fp,\vv)$ and $\Fp=(\nn_1,\nn_2,\nn_3)\in SO(3)$, we define
\begin{align*}
E(\Fp,\vv)\eqdefa\int_{\mathbb{R}^d}e(\Fp,\vv)\ud\xx,\quad e(\Fp,\vv)=f_{Bi}(\Fp,\nabla\Fp)+\frac{1}{2}|\vv|^2.
\end{align*}

\begin{proposition}\label{energ-diss-prop}
Under the condition of \eqref{coefficient-conditions},
assume that $(\Fp,\vv)$ is a smooth solution to the biaxial frame system \eqref{new-frame-equation-n1}--\eqref{imcompressible-v} with initial data $(\Fp^{(0)},\vv^{(0)})$ and $\nabla\cdot\vv^{(0)}=0$. Then it follows that
\begin{align}\label{energy-dissipative-law}
&\int_{\mathbb{R}^d}e(\Fp(\cdot,t),\vv(\cdot,t))\ud\xx+\eta\int^t_0\|\nabla\vv\|^2_{L^2}\ud t+\sum^3_{k=1}\frac{1}{\chi_k}\int^t_0\|\ML_k\CF_{Bi}\|^2_{L^2}\ud t\nonumber\\
&\quad+\int^t_0\bigg(\beta_1\|\A\cdot\sss_1\|^2_{L^2}+2\beta_0\int_{\mathbb{R}^d}(\A\cdot\sss_1)(\A\cdot\sss_2)\ud\xx+\beta_2\|\A\cdot\sss_2\|^2_{L^2}\bigg)\ud t\nonumber\\
&\quad+\int^t_0\Big[\Big(\beta_3-\frac{\eta^2_3}{\chi_3}\Big)\|\A\cdot\sss_3\|^2_{L^2}
+\Big(\beta_4-\frac{\eta^2_2}{\chi_2}\Big)\|\A\cdot\sss_4\|^2_{L^2}
+\Big(\beta_5-\frac{\eta^2_1}{\chi_1}\Big)\|\A\cdot\sss_5\|^2_{L^2}\Big]\ud t\nonumber\\
&=\int_{\mathbb{R}^d}e(\Fp^{(0)},\vv^{(0)})\ud\xx.
\end{align}
\end{proposition}

\begin{proof}
Taking the inner product on the system (\ref{new-frame-equation-n1})--(\ref{new-frame-equation-n3}) with $\frac{\delta\CF_{Bi}}{\delta\nn_i}(i=1,2,3)$, respectively, we can deduce that
\begin{align}\label{energy-p-time}
\frac{\ud\CF_{Bi}}{\ud t}=&\int_{\mathbb{R}^d}\Big(\frac{\delta\CF_{Bi}}{\delta\nn_1}\cdot\frac{\partial\nn_1}{\partial t}+\frac{\delta\CF_{Bi}}{\delta\nn_2}\cdot\frac{\partial\nn_2}{\partial t}+\frac{\delta\CF_{Bi}}{\delta\nn_3}\cdot\frac{\partial\nn_3}{\partial t}\Big)\ud\xx\nonumber\\
=&-\int_{\mathbb{R}^d}\Big((\vv\cdot\nabla\nn_1)\cdot\frac{\delta\CF_{Bi}}{\delta\nn_1}+(\vv\cdot\nabla\nn_2)\cdot\frac{\delta\CF_{Bi}}{\delta\nn_2}
+(\vv\cdot\nabla\nn_3)\cdot\frac{\delta\CF_{Bi}}{\delta\nn_3}\Big)\ud\xx\nonumber\\
&+\int_{\mathbb{R}^d}\Big(\frac{1}{2}\BOm\cdot\aaa_1+\frac{\eta_3}{\chi_3}\A\cdot\sss_3-\frac{1}{\chi_3}\ML_3\CF_{Bi}\Big)\ML_3\CF_{Bi}\ud\xx\nonumber\\
&+\int_{\mathbb{R}^d}\Big(\frac{1}{2}\BOm\cdot\aaa_2+\frac{\eta_2}{\chi_2}\A\cdot\sss_4-\frac{1}{\chi_2}\ML_2\CF_{Bi}\Big)\ML_2\CF_{Bi}\ud\xx\nonumber\\
&+\int_{\mathbb{R}^d}\Big(\frac{1}{2}\BOm\cdot\aaa_3+\frac{\eta_1}{\chi_1}\A\cdot\sss_5-\frac{1}{\chi_1}\ML_1\CF_{Bi}\Big)\ML_1\CF_{Bi}\ud\xx.
\end{align}
On the other hand,
taking the inner product on the equation (\ref{frame-equation-v}) with $\vv$ and using the incompressible condition $\nabla\cdot\vv=0$, we obtain
\begin{align}\label{energy-v-time}
\frac{1}{2}\frac{\ud}{\ud t}\int_{\mathbb{R}^d}|\vv|^2\ud\xx+\eta\|\nabla\vv\|^2_{L^2}=-\langle\sigma,\nabla\vv\rangle+\langle\mathfrak{F},\vv\rangle,
\end{align}
where
\begin{align*}
&\langle\mathfrak{F},\vv\rangle=\int_{\mathbb{R}^d} v_i\big(n_{3k}\partial_in_{2k}\ML_1\CF_{Bi}+n_{1k}\partial_in_{3k}\ML_2\CF_{Bi}
+n_{2k}\partial_in_{1k}\ML_3\CF_{Bi}\big)\ud\xx.
\end{align*}
Using (\ref{sigma-e}) and the equations (\ref{frame-equation-n1})--(\ref{frame-equation-n3}), we deduce that
\begin{align}\label{sigma-e-nabla-v}
\langle\sigma,\nabla\vv\rangle=&\beta_1\|\A\cdot\sss_1\|^2_{L^2}+2\beta_0\int_{\mathbb{R}^d}(\A\cdot\sss_1)(\A\cdot\sss_2)\ud\xx+\beta_2\|\A\cdot\sss_2\|^2_{L^2}\nonumber\\
&+\beta_3\|\A\cdot\sss_3\|^2_{L^2}-\eta_3\int_{\mathbb{R}^d}\Big(\frac{\eta_3}{\chi_3}\A\cdot\sss_3-\frac{1}{\chi_3}\ML_3\CF_{Bi}\Big)\A\cdot\sss_3\ud\xx\nonumber\\
&+\beta_4\|\A\cdot\sss_4\|^2_{L^2}-\eta_2\int_{\mathbb{R}^d}\Big(\frac{\eta_2}{\chi_2}\A\cdot\sss_4-\frac{1}{\chi_2}\ML_2\CF_{Bi}\Big)\A\cdot\sss_4\ud\xx\nonumber\\
&+\beta_5\|\A\cdot\sss_5\|^2_{L^2}-\eta_1\int_{\mathbb{R}^d}\Big(\frac{\eta_1}{\chi_1}\A\cdot\sss_5-\frac{1}{\chi_1}\ML_1\CF_{Bi}\Big)\A\cdot\sss_5\ud\xx\nonumber\\
&+\frac{1}{2}\eta_3\int_{\mathbb{R}^d}(\A\cdot\sss_3)(\BOm\cdot\aaa_1)\ud\xx
-\frac{1}{2}\chi_3\int_{\mathbb{R}^d}\Big(\frac{\eta_3}{\chi_3}\A\cdot\sss_3-\frac{1}{\chi_3}\ML_3\CF_{Bi}\Big)\BOm\cdot\aaa_1\ud\xx\nonumber\\
&+\frac{1}{2}\eta_2\int_{\mathbb{R}^d}(\A\cdot\sss_4)(\BOm\cdot\aaa_2)\ud\xx
-\frac{1}{2}\chi_2\int_{\mathbb{R}^d}\Big(\frac{\eta_2}{\chi_2}\A\cdot\sss_4-\frac{1}{\chi_2}\ML_2\CF_{Bi}\Big)\BOm\cdot\aaa_2\ud\xx\nonumber\\
&+\frac{1}{2}\eta_1\int_{\mathbb{R}^d}(\A\cdot\sss_5)(\BOm\cdot\aaa_3)\ud\xx
-\frac{1}{2}\chi_1\int_{\mathbb{R}^d}\Big(\frac{\eta_1}{\chi_1}\A\cdot\sss_5-\frac{1}{\chi_1}\ML_1\CF_{Bi}\Big)\BOm\cdot\aaa_3\ud\xx\nonumber\\
=&\beta_1\|\A\cdot\sss_1\|^2_{L^2}+2\beta_0\int_{\mathbb{R}^d}(\A\cdot\sss_1)(\A\cdot\sss_2)\ud\xx+\beta_2\|\A\cdot\sss_2\|^2_{L^2}\nonumber\\
&+\Big(\beta_3-\frac{\eta^2_3}{\chi_3}\Big)\|\A\cdot\sss_3\|^2_{L^2}+\frac{\eta_3}{\chi_3}\int_{\mathbb{R}^d}(\A\cdot\sss_3)\ML_3\CF_{Bi}\ud\xx\nonumber\\
&+\Big(\beta_4-\frac{\eta^2_2}{\chi_2}\Big)\|\A\cdot\sss_4\|^2_{L^2}+\frac{\eta_2}{\chi_2}\int_{\mathbb{R}^d}(\A\cdot\sss_4)\ML_2\CF_{Bi}\ud\xx\nonumber\\
&+\Big(\beta_5-\frac{\eta^2_1}{\chi_1}\Big)\|\A\cdot\sss_5\|^2_{L^2}+\frac{\eta_1}{\chi_1}\int_{\mathbb{R}^d}(\A\cdot\sss_5)\ML_1\CF_{Bi}\ud\xx\nonumber\\
&+\frac{1}{2}\int_{\mathbb{R}^d}\Big((\BOm\cdot\aaa_1)\ML_3\CF_{Bi}+(\BOm\cdot\aaa_2)\ML_2\CF_{Bi}+(\BOm\cdot\aaa_3)\ML_1\CF_{Bi}\Big)\ud\xx.
\end{align}
It can be seen from the relations (\ref{relation-CD-Fp}) that
\begin{align}\label{F-v-inner}
&\int_{\mathbb{R}^d}\Big((v_i\partial_i\nn_1)\cdot\frac{\delta\CF_{Bi}}{\delta\nn_1}+(v_i\partial_i\nn_2)\cdot\frac{\delta\CF_{Bi}}{\delta\nn_2}
+(v_i\partial_i\nn_3)\cdot\frac{\delta\CF_{Bi}}{\delta\nn_3}\Big)\ud\xx\nonumber\\
&=\int_{\mathbb{R}^d}\Big(v_i\big((\partial_i\nn_1\cdot\nn_2)\nn_2+(\partial_i\nn_1\cdot\nn_3)\nn_3\big)\cdot\frac{\delta\CF_{Bi}}{\delta\nn_1}\nonumber\\
&\quad+v_i\big((\partial_i\nn_2\cdot\nn_1)\nn_1+(\partial_i\nn_2\cdot\nn_3)\nn_3\big)\cdot\frac{\delta\CF_{Bi}}{\delta\nn_2}\nonumber\\
&\quad+v_i\big((\partial_i\nn_3\cdot\nn_1)\nn_1+(\partial_i\nn_3\cdot\nn_2)\nn_2\big)\cdot\frac{\delta\CF_{Bi}}{\delta\nn_3}
\Big)\ud\xx\nonumber\\
&=\int_{\mathbb{R}^d}\Big(v_i(\partial_i\nn_1\cdot\nn_2)\ML_3\CF_{Bi}+v_i(\partial_i\nn_3\cdot\nn_1)\ML_2\CF_{Bi}
+v_i(\partial_i\nn_2\cdot\nn_3)\ML_1\CF_{Bi}\Big)\ud\xx\nonumber\\
&=\langle\mathfrak{F},\vv\rangle.
\end{align}
Hence, combining (\ref{energy-p-time}) with (\ref{energy-v-time}), and using (\ref{sigma-e-nabla-v})--(\ref{F-v-inner}), we obtain the energy dissipative law:
\begin{align}\label{energy-law}
&\frac{\ud}{\ud t}\Big(\frac{1}{2}\int_{\mathbb{R}^d}|\vv|^2\ud\xx+\CF_{Bi}[\Fp]\Big)
=-\eta\|\nabla\vv\|^2_{L^2}-\sum^3_{k=1}\frac{1}{\chi_k}\|\ML_k\CF_{Bi}\|^2_{L^2}\nonumber\\
&\qquad-\bigg(\beta_1\|\A\cdot\sss_1\|^2_{L^2}+2\beta_0\int_{\mathbb{R}^d}(\A\cdot\sss_1)(\A\cdot\sss_2)\ud\xx+\beta_2\|\A\cdot\sss_2\|^2_{L^2}\bigg)\nonumber\\
&\qquad-\Big(\beta_3-\frac{\eta^2_3}{\chi_3}\Big)\|\A\cdot\sss_3\|^2_{L^2}
-\Big(\beta_4-\frac{\eta^2_2}{\chi_2}\Big)\|\A\cdot\sss_4\|^2_{L^2}
-\Big(\beta_5-\frac{\eta^2_1}{\chi_1}\Big)\|\A\cdot\sss_5\|^2_{L^2},
\end{align}
which implies this proposition by integrating on the time.

\end{proof}

To make our analysis more convenient, we need to rewrite the orientational elasticity (\ref{elastic-energy}). For any frame $\Fp=(\nn_1,\nn_2,\nn_3)\in SO(3)$, we have the following simple identity relations:
\begin{align*}
&|\nn_1\times(\nabla\times\nn_1)|^2=(\nn_2\cdot\nabla\times\nn_1)^2+(\nn_3\cdot\nabla\times\nn_1)^2,\\
&|\nn_2\times(\nabla\times\nn_2)|^2=(\nn_1\cdot\nabla\times\nn_2)^2+(\nn_3\cdot\nabla\times\nn_2)^2,\\
&|\nn_3\times(\nabla\times\nn_3)|^2=(\nn_1\cdot\nabla\times\nn_3)^2+(\nn_2\cdot\nabla\times\nn_3)^2,\\
&|\nabla\nn_i|^2=(\nabla\cdot\nn_i)^2+(\nn_i\cdot\nabla\times\nn_i)^2+|\nn_i\times(\nabla\times\nn_i)|^2\\
&\qquad\qquad+\nabla\cdot[(\nn_i\cdot\nabla)\nn_i-(\nabla\cdot\nn_i)\nn_i],\quad i=1,2,3.
\end{align*}

Without loss of generality, as in \cite{MGJ,HX}, in terms of the above simple relations, we rewrite the density $f_{Bi}(\Fp,\nabla\Fp)$ as
\begin{align}\label{new-elasitic-density}
    f_{Bi}(\Fp,\nabla\Fp)=\frac{1}{2}\sum^3_{i=1}\gamma_i|\nabla \nn_{i}|^2+W(\Fp,\nabla\Fp).
\end{align}
Here, the coefficients $\gamma_i(i=1,2,3)$ are taken as, respectively,
\begin{align*}
&\gamma_1=\min\{K_1,K_4,K_7,K_{10}\}>0,~ \gamma_2=\min\{K_2,K_5,K_8,K_{11}\}>0,\\
&\gamma_3=\min\{K_3,K_6,K_9,K_{12}\}>0,
\end{align*}
and $W(\Fp,\nabla\Fp)$ is expressed by
\begin{align*}
W(\Fp,\nabla\Fp)=&\frac{1}{2}\Big(\sum^3_{i=1}k_i(\nabla\cdot\nn_i)^2+\sum^3_{i,j=1}k_{ij}(\nn_i\cdot\nabla\times\nn_j)^2\Big),
\end{align*}
where the coefficients $k_i\geq 0,k_{ij}\geq0(i,j=1,2,3)$ are given by
\begin{align*}
&k_1=K_1-\gamma_1,\quad k_2=K_2-\gamma_2,\quad
k_3=K_3-\gamma_3,\\
&k_{11}=K_4-\gamma_1,\quad
k_{22}=K_5-\gamma_2,\quad
k_{33}=K_6-\gamma_3,\\
&k_{31}=K_7-\gamma_1,\quad
k_{12}=K_8-\gamma_2,\quad
k_{23}=K_9-\gamma_3,\\
&k_{21}=K_{10}-\gamma_1,\quad
k_{32}=K_{11}-\gamma_2,\quad
k_{13}=K_{12}-\gamma_3.
\end{align*}

For simplicity, we define
\begin{align*}
&(\hh_1,\hh_2,\hh_3)\eqdefa-\Big(\frac{\delta\CF_{Bi}}{\delta\nn_1},\frac{\delta\CF_{Bi}}{\delta\nn_2},\frac{\delta\CF_{Bi}}{\delta\nn_3}\Big)=-\frac{\delta\CF_{Bi}}{\delta\Fp}=\nabla\cdot\frac{\partial f_{Bi}}{\partial(\nabla\Fp)}-\frac{\partial f_{Bi}}{\partial\Fp}.
\end{align*}

Similar to Lemma 2.3 in \cite{WW}, we also introduce the algebraic structures of the  variational derivative with regards to the frame field $\Fp=(\nn_1,\nn_2,\nn_3)\in SO(3)$.

\begin{lemma}\label{h-decomposition}
For the terms $\hh_i(i=1,2,3)$, we have the following representation:
\begin{align*}
\hh_i=&\gamma_i\Delta\nn_i+k_i\nabla{\rm div}\nn_i
-\sum^3_{j=1}k_{ji}\nabla\times(\nabla\times\nn_i\cdot\nn^2_j)
-\sum^3_{j=1}k_{ij}(\nn_i\cdot\nabla\times\nn_j)(\nabla\times\nn_j),
\end{align*}
where $\nn^2_j=\nn_j\otimes\nn_j$.
\end{lemma}

\section{Local well-posedness and blow-up criterion}\label{local-well-posed-blowup-section}

This section is devoted to the local well-posedness and blow-up criterion of the system (\ref{new-frame-equation-n1})--(\ref{imcompressible-v}). In what follows, we denote by $C$ a positive constant depending on the coefficients of the system (\ref{new-frame-equation-n1})--(\ref{imcompressible-v}) but independent of the solution $(\Fp,\vv)$. The symbol $\langle\cdot,\cdot\rangle$ stands for the $L^2$-inner product in $\mathbb{R}^d$ with $d=2,3$. In addition, we use $\CP$ to represent a finite degree polynomial of the variables in the bracket.

The following product estimate and commutator estimate are well-known, see \cite{Triebel, BCD, WZZ1} for example, which will be frequently used in this paper.

\begin{lemma}\label{product-estimate-lemma}
For any multi-indices $\alpha, \beta, \gamma\in\mathbb{N}^3$ and any differential operator $\CD$, it follows that
\begin{align*}
\|\CD^{\alpha}(fg)\|_{L^2}\leq& C\sum_{|\gamma|=|\alpha|}(\|f\|_{L^{\infty}}\|\CD^{\gamma}g\|_{L^2}+\|g\|_{L^{\infty}}\|\CD^{\gamma}f\|_{L^2}),\\
\|[\CD^{\alpha},f]\CD^{\beta}g\|_{L^2}\leq&C\bigg(\sum_{|\gamma|=|\alpha|+|\beta|}\|\CD^{\gamma}f\|_{L^2}\|g\|_{L^{\infty}}
+\sum_{|\gamma|=|\alpha|+|\beta|-1}\|\nabla f\|_{L^{\infty}}\|\CD^{\gamma}g\|_{L^2}\bigg).
\end{align*}
\end{lemma}
Besides, we present the commutator estimate with the mollification operator (see \cite{WW}).

\begin{lemma}\label{molifi-lemma}
 For the mollification operator $\CJ_{\ve}$ defined by {\rm (\ref{molifi-operator})} and $1\leq p\leq\infty$, it follows that
 \begin{align*}
 \|[\CJ_{\ve},f]\partial_ig\|_{L^p}\leq C(1+\|\nabla f\|_{L^{\infty}})\|g\|_{L^p}.
 \end{align*}
\end{lemma}

\subsection{Local well-posedness}\label{local-well-subsect}

The main framework of our proof follows the strategy in \cite{WZZ1,WW}. We divide the proof of Theorem \ref{local-posedness-theorem} into three steps.

{\bf Step 1}. {\it Construction of approximate solutions}. We construct an approximate system to apply the classical Friedrich's method.
We define the mollification operator
\begin{align}\label{molifi-operator}
\CJ_{\ve}f(\xi)\eqdefa \MF^{-1}\big(\phi(\xi/\ve)\MF f\big),
\end{align}
where $\MF(f)(\xi)=\int_{\mathbb{R}^d}f(\xx)e^{-i\xx\cdot\xi}\ud\xx$ is the Fourier transform and $\phi(\xi)$ is a smooth cut-off function with $\phi=1$ in the unit ball $B(0,1)$ and $\phi=0$ outside $B(0,2)$. Assume that $\mathbb{P}$ is the Leray's projection operator from a vector field $\ww$ into the corresponding divergence-free field, namely, $\mathbb{P}\ww=\ww-\nabla[\Delta^{-1}(\nabla\cdot\ww)]$.

Then, the approximate system associated with (\ref{new-frame-equation-n1})--(\ref{imcompressible-v}) is given by
\begin{align}
&\dot{\nn}^{\ve}_1=\CJ_{\ve}\bigg(\Big(\frac{1}{2}\CJ_{\ve}\BOm^{\ve}\cdot\CJ_{\ve}\aaa^{\ve}_1
+\frac{\eta_3}{\chi_3}\CJ_{\ve}\A^{\ve}\cdot\CJ_{\ve}\sss^{\ve}_3-\frac{1}{\chi_3}\CJ_{\ve}(\ML^{\ve}_3\CF_{Bi})\Big)\CJ_{\ve}\nn^{\ve}_2\bigg)\nonumber\\
&\qquad\quad-\CJ_{\ve}\bigg(\Big(\frac{1}{2}\CJ_{\ve}\BOm^{\ve}\cdot\CJ_{\ve}\aaa^{\ve}_2+\frac{\eta_2}{\chi_2}\CJ_{\ve}\A^{\ve}\cdot\CJ_{\ve}\sss^{\ve}_4
-\frac{1}{\chi_2}\CJ_{\ve}(\ML^{\ve}_2\CF_{Bi})\Big)\CJ_{\ve}\nn^{\ve}_3\bigg),\label{n1-ve-equation}\\
&\dot{\nn}^{\ve}_2=-\CJ_{\ve}\bigg(\Big(\frac{1}{2}\CJ_{\ve}\BOm^{\ve}\cdot\CJ_{\ve}\aaa^{\ve}_1
+\frac{\eta_3}{\chi_3}\CJ_{\ve}\A^{\ve}\cdot\CJ_{\ve}\sss^{\ve}_3-\frac{1}{\chi_3}\CJ_{\ve}(\ML^{\ve}_3\CF_{Bi})\Big)\CJ_{\ve}\nn^{\ve}_1\bigg)\nonumber\\
&\qquad\quad+\CJ_{\ve}\bigg(\Big(\frac{1}{2}\CJ_{\ve}\BOm^{\ve}\cdot\CJ_{\ve}\aaa^{\ve}_3+\frac{\eta_1}{\chi_1}\CJ_{\ve}\A^{\ve}\cdot\CJ_{\ve}\sss^{\ve}_5
-\frac{1}{\chi_1}\CJ_{\ve}(\ML^{\ve}_1\CF_{Bi})\Big)\CJ_{\ve}\nn^{\ve}_3\bigg),\label{n2-ve-equation}\\
&\dot{\nn}^{\ve}_3=\CJ_{\ve}\bigg(\Big(\frac{1}{2}\CJ_{\ve}\BOm^{\ve}\cdot\CJ_{\ve}\aaa^{\ve}_2
+\frac{\eta_2}{\chi_2}\CJ_{\ve}\A^{\ve}\cdot\CJ_{\ve}\sss^{\ve}_4-\frac{1}{\chi_2}\CJ_{\ve}(\ML^{\ve}_2\CF_{Bi})\Big)\CJ_{\ve}\nn^{\ve}_1\bigg)\nonumber\\
&\qquad\quad-\CJ_{\ve}\bigg(\Big(\frac{1}{2}\CJ_{\ve}\BOm^{\ve}\cdot\CJ_{\ve}\aaa^{\ve}_3+\frac{\eta_1}{\chi_1}\CJ_{\ve}\A^{\ve}\cdot\CJ_{\ve}\sss^{\ve}_5
-\frac{1}{\chi_1}\CJ_{\ve}(\ML^{\ve}_1\CF_{Bi})\Big)\CJ_{\ve}\nn^{\ve}_2\bigg),\label{n3-ve-equation}\\
&\frac{\partial\vv_{\ve}}{\partial t}+\CJ_{\ve}\mathbb{P}\big(\CJ_{\ve}\vv_{\ve}\cdot\nabla\CJ_{\ve}\vv_{\ve}\big)=\eta\CJ_{\ve}\Delta\CJ_{\ve}\vv_{\ve}
+\CJ_{\ve}\mathbb{P}\big(\nabla\cdot\CJ_{\ve}\sigma^{\ve}+\CJ_{\ve}\mathfrak{F}^{\ve}\big),\label{v-ve-equation}\\
&\nabla\cdot\vv_{\ve}=0,\label{v-imcomp}\\
&\big(\vv_{\ve},\nn^{\ve}_1,\nn^{\ve}_2,\nn^{\ve}_3\big)\big|_{t=0}=\big(\CJ_{\ve}\vv_0,\CJ_{\ve}\nn^{(0)}_1,\CJ_{\ve}\nn^{(0)}_2,\CJ_{\ve}\nn^{(0)}_3\big).\label{initial-condition}
\end{align}
Here, $\dot{\nn}^{\ve}_k(k=1,2,3)$ are expressed by
\begin{align*}
\dot{\nn}^{\ve}_k=\partial_t\nn^{\ve}_k+\CJ_{\ve}\big(\CJ_{\ve}\vv_{\ve}\cdot\nabla\CJ_{\ve}\nn^{\ve}_k\big),
\end{align*}
and $\CJ_{\ve}\sss^{\ve}_i(i=1,\cdots,5)$ and $\CJ_{\ve}\aaa^{\ve}_j(j=1,2,3)$ are defined as, respectively,
\begin{align*}
\CJ_{\ve}\sss^{\ve}_1=&\CJ_{\ve}\nn^{\ve}_1\otimes\CJ_{\ve}\nn^{\ve}_1-\frac{\Fi}{3},\quad
\CJ_{\ve}\sss^{\ve}_2=\CJ_{\ve}\nn^{\ve}_2\otimes\CJ_{\ve}\nn^{\ve}_2-\CJ_{\ve}\nn^{\ve}_3\otimes\CJ_{\ve}\nn^{\ve}_3,\\
\CJ_{\ve}\sss^{\ve}_3=&\frac{1}{2}\big(\CJ_{\ve}\nn^{\ve}_1\otimes\CJ_{\ve}\nn^{\ve}_2+\CJ_{\ve}\nn^{\ve}_2\otimes\CJ_{\ve}\nn^{\ve}_1\big),\quad
\CJ_{\ve}\sss^{\ve}_4=\frac{1}{2}\big(\CJ_{\ve}\nn^{\ve}_1\otimes\CJ_{\ve}\nn^{\ve}_3+\CJ_{\ve}\nn^{\ve}_3\otimes\CJ_{\ve}\nn^{\ve}_1\big),\\
\CJ_{\ve}\sss^{\ve}_5=&\frac{1}{2}\big(\CJ_{\ve}\nn^{\ve}_2\otimes\CJ_{\ve}\nn^{\ve}_3+\CJ_{\ve}\nn^{\ve}_3\otimes\CJ_{\ve}\nn^{\ve}_2\big),\quad
\CJ_{\ve}\aaa^{\ve}_1=\CJ_{\ve}\nn^{\ve}_1\otimes\CJ_{\ve}\nn^{\ve}_2-\CJ_{\ve}\nn^{\ve}_2\otimes\CJ_{\ve}\nn^{\ve}_1,\\
\CJ_{\ve}\aaa^{\ve}_2=&\CJ_{\ve}\nn^{\ve}_3\otimes\CJ_{\ve}\nn^{\ve}_1-\CJ_{\ve}\nn^{\ve}_1\otimes\CJ_{\ve}\nn^{\ve}_3,\quad
\CJ_{\ve}\aaa^{\ve}_3=\CJ_{\ve}\nn^{\ve}_2\otimes\CJ_{\ve}\nn^{\ve}_3-\CJ_{\ve}\nn^{\ve}_3\otimes\CJ_{\ve}\nn^{\ve}_2.
\end{align*}
The variational derivatives in the approximate system read
\begin{align*}
\CJ_{\ve}\big(\ML^{\ve}_1\CF_{Bi}\big)=&\CJ_{\ve}\nn^{\ve}_2\cdot\CJ_{\ve}\hh^{\ve}_3-\CJ_{\ve}\nn^{\ve}_3\cdot\CJ_{\ve}\hh^{\ve}_2,\\
\CJ_{\ve}\big(\ML^{\ve}_2\CF_{Bi}\big)=&\CJ_{\ve}\nn^{\ve}_3\cdot\CJ_{\ve}\hh^{\ve}_1-\CJ_{\ve}\nn^{\ve}_1\cdot\CJ_{\ve}\hh^{\ve}_3,\\
\CJ_{\ve}\big(\ML^{\ve}_3\CF_{Bi}\big)=&\CJ_{\ve}\nn^{\ve}_1\cdot\CJ_{\ve}\hh^{\ve}_2-\CJ_{\ve}\nn^{\ve}_2\cdot\CJ_{\ve}\hh^{\ve}_1,
\end{align*}
where $\CJ_{\ve}\hh^{\ve}_i(i=1,2,3)$ are respectively defined as
\begin{align}\label{CJ-hi}
\CJ_{\ve}\hh^{\ve}_i=&\gamma_i\Delta\CJ_{\ve}\nn^{\ve}_i+k_i\nabla{\rm div}\CJ_{\ve}\nn^{\ve}_i
-\sum^3_{j=1}k_{ji}\nabla\times\big(\nabla\times\CJ_{\ve}\nn^{\ve}_i\cdot(\CJ_{\ve}\nn^{\ve}_j)^2\big)\nonumber\\
&-\sum^3_{j=1}k_{ij}(\CJ_{\ve}\nn^{\ve}_i\cdot\nabla\times\CJ_{\ve}\nn^{\ve}_j)(\nabla\times\CJ_{\ve}\nn^{\ve}_j).
\end{align}
The stress $\CJ_{\ve}\sigma^{\ve}$ is given by
\begin{align*}
\CJ_{\ve}\sigma^{\ve}=&\beta_1\big(\CJ_{\ve}\A^{\ve}\cdot\CJ_{\ve}\sss^{\ve}_1\big)\CJ_{\ve}\sss^{\ve}_1
+\beta_0\big(\CJ_{\ve}\A^{\ve}\cdot\CJ_{\ve}\sss^{\ve}_2\big)\CJ_{\ve}\sss^{\ve}_1
+\beta_0\big(\CJ_{\ve}\A^{\ve}\cdot\CJ_{\ve}\sss^{\ve}_1\big)\CJ_{\ve}\sss^{\ve}_2\\
&+\beta_2\big(\CJ_{\ve}\A^{\ve}\cdot\CJ_{\ve}\sss^{\ve}_2\big)\CJ_{\ve}\sss^{\ve}_2
+\Big(\beta_3-\frac{\eta^2_3}{\chi_3}\Big)\big(\CJ_{\ve}\A^{\ve}\cdot\CJ_{\ve}\sss^{\ve}_3\big)\CJ_{\ve}\sss^{\ve}_3\\
&+\frac{\eta_3}{\chi_3}\CJ_{\ve}\big(\ML^{\ve}_3\CF_{Bi}\big)\CJ_{\ve}\sss^{\ve}_3
+\Big(\beta_4-\frac{\eta^2_2}{\chi_2}\Big)\big(\CJ_{\ve}\A^{\ve}\cdot\CJ_{\ve}\sss^{\ve}_4\big)\CJ_{\ve}\sss^{\ve}_4\\
&+\frac{\eta_2}{\chi_2}\CJ_{\ve}\big(\ML^{\ve}_2\CF_{Bi}\big)\CJ_{\ve}\sss^{\ve}_4
+\Big(\beta_5-\frac{\eta^2_1}{\chi_1}\Big)\big(\CJ_{\ve}\A^{\ve}\cdot\CJ_{\ve}\sss^{\ve}_5\big)\CJ_{\ve}\sss^{\ve}_5\\
&+\frac{\eta_1}{\chi_1}\CJ_{\ve}\big(\ML^{\ve}_1\CF_{Bi}\big)\CJ_{\ve}\sss^{\ve}_5
+\frac{1}{2}\CJ_{\ve}\big(\ML^{\ve}_3\CF_{Bi}\big)\CJ_{\ve}\aaa^{\ve}_1\\
&+\frac{1}{2}\CJ_{\ve}\big(\ML^{\ve}_2\CF_{Bi}\big)\CJ_{\ve}\aaa^{\ve}_2+\frac{1}{2}\CJ_{\ve}\big(\ML^{\ve}_1\CF_{Bi}\big)\CJ_{\ve}\aaa^{\ve}_3,
\end{align*}
where we have used the equations (\ref{frame-equation-n1})--(\ref{frame-equation-n3}) to eliminate time derivatives.
In addition, the body force $\CJ_{\ve}\mathfrak{F}^{\ve}$ is defined by
\begin{align*}
\CJ_{\ve}\mathfrak{F}^{\ve}_i=&\partial_i\CJ_{\ve}\nn^{\ve}_1\cdot\CJ_{\ve}\nn^{\ve}_2\CJ_{\ve}\big(\ML^{\ve}_3\CF_{Bi}\big)
+\partial_i\CJ_{\ve}\nn^{\ve}_3\cdot\CJ_{\ve}\nn^{\ve}_1\CJ_{\ve}\big(\ML^{\ve}_2\CF_{Bi}\big)\\
&+\partial_i\CJ_{\ve}\nn^{\ve}_2\cdot\CJ_{\ve}\nn^{\ve}_3\CJ_{\ve}\big(\ML^{\ve}_1\CF_{Bi}\big).
\end{align*}

The above approximate system (\ref{n1-ve-equation})--(\ref{initial-condition}) can be regarded as an ODE system on $L^2(\mathbb{R}^d)$. Thus, the Cauchy--Lipschitz theorem implies that there exists a strictly maximal time $T_{\ve}$ and a unique solution $(\Fp^{\ve},\vv^{\ve})\in C\big([0,T_{\ve});H^k(\mathbb{R}^d)\big)$ for any $k\geq0$, where $\Fp^{\ve}=(\nn^{\ve}_1,\nn^{\ve}_2,\nn^{\ve}_3)$.

{\bf Step 2}. {\it Uniform energy estimates}. We define the following energy functional
\begin{align}\label{energy-functional-e}
\CE_s(\Fp^{\ve},\vv^{\ve})\eqdefa&\frac{1}{2}\|\Fp^{\ve}-\Fp^{(0)}\|^2_{L^2}+\CF_{Bi}[\Fp^{\ve}]+\frac{1}{2}\|\vv^{\ve}\|^2_{L^2}+\sum^3_{i=1}\CE^{s}_i(\Fp^{\ve})+\frac{1}{2}\|\Delta^s\vv^{\ve}\|^2_{L^2},
\end{align}
where $\|\Fp^{\ve}-\Fp^{(0)}\|^2_{L^2}=\sum\limits^3_{i=1}\|\nn^{\ve}_i-\nn^{(0)}_i\|^2_{L^2}$,  $\Fp^{(0)}=\big(\nn^{(0)}_1,\nn^{(0)}_2,\nn^{(0)}_3\big)\in SO(3)$ is the initial orthonormal frame, and $\CE^{s}_i(\Fp^{\ve})(i=1,2,3)$ are respectively defined as
\begin{align*}
\CE^s_i(\Fp^{\ve})=&\frac{1}{2}\Big(\gamma_i\|\Delta^s\nabla\nn^{\ve}_i\|^2_{L^2}+k_i\|\Delta^s{\rm div}\nn^{\ve}_i\|^2_{L^2}
+\sum^3_{j=1}k_{ji}\|\Delta^s(\nabla\times\nn^{\ve}_i)\cdot\CJ_{\ve}\nn^{\ve}_j\|^2_{L^2}\Big).
\end{align*}

We first deal with the estimates of lower order terms in (\ref{energy-functional-e}). The approximate system (\ref{n1-ve-equation})--(\ref{v-imcomp}) still satisfies the following energy dissipative law:
\begin{align}\label{ve-energy-law}
&\frac{\ud}{\ud t}\Big(\frac{1}{2}\int_{\mathbb{R}^d}|\vv^{\ve}|^2\ud\xx+\CF_{Bi}[\Fp^{\ve}]\Big)
=-\eta\|\nabla\CJ_{\ve}\vv^{\ve}\|^2_{L^2}-\sum^3_{k=1}\frac{1}{\chi_k}\|\CJ_{\ve}\big(\ML^{\ve}_k\CF_{Bi}\big)\|^2_{L^2}\nonumber\\
&\quad-\bigg(\beta_1\|\CJ_{\ve}\A^{\ve}\cdot\CJ_{\ve}\sss^{\ve}_1\|^2_{L^2}
+2\beta_0\int_{\mathbb{R}^d}(\CJ_{\ve}\A^{\ve}\cdot\CJ_{\ve}\sss^{\ve}_1)(\CJ_{\ve}\A^{\ve}\cdot\CJ_{\ve}\sss^{\ve}_2)\ud\xx
+\beta_2\|\CJ_{\ve}\A^{\ve}\cdot\CJ_{\ve}\sss^{\ve}_2\|^2_{L^2}\bigg)\nonumber\\
&\quad-\Big(\beta_3-\frac{\eta^2_3}{\chi_3}\Big)\|\CJ_{\ve}\A^{\ve}\cdot\CJ_{\ve}\sss^{\ve}_3\|^2_{L^2}
-\Big(\beta_4-\frac{\eta^2_2}{\chi_2}\Big)\|\CJ_{\ve}\A^{\ve}\cdot\CJ_{\ve}\sss^{\ve}_4\|^2_{L^2}\nonumber\\
&\quad-\Big(\beta_5-\frac{\eta^2_1}{\chi_1}\Big)\|\CJ_{\ve}\A^{\ve}\cdot\CJ_{\ve}\sss^{\ve}_5\|^2_{L^2}.
\end{align}
By the Calder$\acute{\rm o}$n--Zygmund inequality,
\begin{align*}
\|\nabla^2\nn_i\|_{L^2}\leq C(\|\Delta\nn_i\|_{L^2}+\|\nn_i-\nn^{(0)}_i\|_{L^2}),\quad i=1,2,3,
\end{align*}
we derive from (\ref{CJ-hi}) that
\begin{align*}
\|\CJ_{\ve}\hh^{\ve}_i\|_{L^2}\leq&C\Big(\|\Delta\nn^{\ve}_i\|_{L^2}+\|\nabla{\rm div}\nn^{\ve}_i\|_{L^2}+\|\nabla^2\nn^{\ve}_i\|_{L^2}\|\Fp^{\ve}\|^2_{L^{\infty}}\\
&+\|\nabla\nn^{\ve}_i\|_{L^2}\|\Fp^{\ve}\|_{L^{\infty}}\|\nabla\Fp^{\ve}\|_{L^{\infty}}\Big)\\
\leq&C\CP\big(\|\Fp^{\ve}\|_{L^{\infty}},\|\nabla\Fp^{\ve}\|_{L^{\infty}}\big)\CE^{\frac{1}{2}}_s(\vv^{\ve},\Fp^{\ve}).
\end{align*}

By means of the equations (\ref{n1-ve-equation})--(\ref{n3-ve-equation}), $\nabla\cdot\vv^{\ve}=0$ and the estimates of $\CJ_{\ve}\hh^{\ve}_i(i=1,2,3)$, we deduce that
\begin{align}\label{L-2-estimate-n1}
&\frac{1}{2}\frac{\ud}{\ud t}\|\nn^{\ve}_1-\nn^{(0)}_1\|^2_{L^2}=\big\langle\partial_t\nn^{\ve}_1,\nn^{\ve}_1-\nn^{(0)}_1\big\rangle\nonumber\\
&\quad=-\big\langle\CJ_{\ve}\vv^{\ve}\cdot\nabla\CJ_{\ve}\nn^{(0)}_1,\CJ_{\ve}(\nn^{\ve}_1-\nn^{(0)}_1)\big\rangle\nonumber\\
&\qquad+\Big\langle\Big(\frac{1}{2}\CJ_{\ve}\BOm^{\ve}\cdot\CJ_{\ve}\aaa^{\ve}_1
+\frac{\eta_3}{\chi_3}\CJ_{\ve}\A^{\ve}\cdot\CJ_{\ve}\sss^{\ve}_3-\frac{1}{\chi_3}\CJ_{\ve}(\ML^{\ve}_3\CF_{Bi})\Big)
\CJ_{\ve}\nn^{\ve}_2,\CJ_{\ve}(\nn^{\ve}_1-\nn^{(0)}_1)\Big\rangle\nonumber\\
&\qquad-\Big\langle\Big(\frac{1}{2}\CJ_{\ve}\BOm^{\ve}\cdot\CJ_{\ve}\aaa^{\ve}_2
+\frac{\eta_2}{\chi_2}\CJ_{\ve}\A^{\ve}\cdot\CJ_{\ve}\sss^{\ve}_4-\frac{1}{\chi_2}\CJ_{\ve}(\ML^{\ve}_2\CF_{Bi})\Big)
\CJ_{\ve}\nn^{\ve}_3,\CJ_{\ve}(\nn^{\ve}_1-\nn^{(0)}_1)\Big\rangle\nonumber\\
&\quad\leq C\Big(\|\nabla\nn^{(0)}_1\|_{L^{\infty}}\|\vv_{\ve}\|_{L^2}+(\|\aaa^{\ve}_1\|_{L^{\infty}}+\|\sss^{\ve}_3\|_{L^{\infty}})\|\nn^{\ve}_2\|_{L^{\infty}}\|\nabla\vv^{\ve}\|_{L^2}\Big)\|\nn^{\ve}_1-\nn^{(0)}_1\|_{L^2}\nonumber\\
&\qquad+C(\|\nn^{\ve}_1\|_{L^{\infty}}+\|\nn^{\ve}_2\|_{L^{\infty}})\|\nn^{\ve}_2\|_{L^{\infty}}\big(\|\CJ_{\ve}\hh^{\ve}_1\|_{L^2}+\|\CJ_{\ve}\hh^{\ve}_2\|_{L^2}\big)\|\nn^{\ve}_1-\nn^{(0)}_1\|_{L^2}\nonumber\\
&\qquad+C(\|\aaa^{\ve}_2\|_{L^{\infty}}+\|\sss^{\ve}_4\|_{L^{\infty}})\|\nn^{\ve}_3\|_{L^{\infty}}\|\nabla\vv^{\ve}\|_{L^2}\|\nn^{\ve}_1-\nn^{(0)}_1\|_{L^2}\nonumber\\
&\qquad+C(\|\nn^{\ve}_1\|_{L^{\infty}}+\|\nn^{\ve}_3\|_{L^{\infty}})\|\nn^{\ve}_3\|_{L^{\infty}}\big(\|\CJ_{\ve}\hh^{\ve}_1\|_{L^2}+\|\CJ_{\ve}\hh^{\ve}_3\|_{L^2}\big)\|\nn^{\ve}_1-\nn^{(0)}_1\|_{L^2}\nonumber\\
&\quad\leq C\CP\big(\|\nabla\nn^{(0)}_1\|_{L^{\infty}},\|\Fp^{\ve}\|_{L^{\infty}},\|\nabla\Fp^{\ve}\|_{L^{\infty}}\big)\CE_s(\Fp^{\ve},\vv^{\ve}).
\end{align}
Similar derivations yield
\begin{align}
\frac{1}{2}\frac{\ud}{\ud t}\|\nn^{\ve}_2-\nn^{(0)}_2\|^2_{L^2}
\leq& C\CP\big(\|\nabla\nn^{(0)}_2\|_{L^{\infty}},\|\Fp^{\ve}\|_{L^{\infty}},\|\nabla\Fp^{\ve}\|_{L^{\infty}}\big)\CE_s(\Fp^{\ve},\vv^{\ve}),\label{L-2-estimate-n2}\\
\frac{1}{2}\frac{\ud}{\ud t}\|\nn^{\ve}_3-\nn^{(0)}_3\|^2_{L^2}
\leq& C\CP\big(\|\nabla\nn^{(0)}_3\|_{L^{\infty}},\|\Fp^{\ve}\|_{L^{\infty}},\|\nabla\Fp^{\ve}\|_{L^{\infty}}\big)\CE_s(\Fp^{\ve},\vv^{\ve}).\label{L-2-estimate-n3}
\end{align}

We now consider the estimate of the higher order derivatives for the frame $\Fp^{\ve}=(\nn^{\ve}_1,\nn^{\ve}_2,\nn^{\ve}_3)$. Using the equation (\ref{n1-ve-equation}) and integrating by parts, we can derive that
\begin{align}\label{n1-higher-grid-L2}
&\frac{1}{2}\frac{\ud}{\ud t}\big\langle\nabla\Delta^s\nn^{\ve}_1,\nabla\Delta^s\nn^{\ve}_1\big\rangle\nonumber\\
&\quad=-\Big\langle\Delta^s\Big[\Big(\frac{1}{2}\CJ_{\ve}\BOm^{\ve}\cdot\CJ_{\ve}\aaa^{\ve}_1+\frac{\eta_3}{\chi_3}\CJ_{\ve}\A^{\ve}\cdot\CJ_{\ve}\sss^{\ve}_3\Big)\CJ_{\ve}\nn^{\ve}_2\Big], \Delta^{s+1}\CJ_{\ve}\nn^{\ve}_1\Big\rangle\nonumber\\
&\qquad+\frac{1}{\chi_3}\big\langle\Delta^s\big(\CJ_{\ve}(\ML^{\ve}_3\CF_{Bi})\CJ_{\ve}\nn^{\ve}_2\big), \Delta^{s+1}\CJ_{\ve}\nn^{\ve}_1\big\rangle\nonumber\\
&\qquad+\Big\langle\Delta^s\Big[\Big(\frac{1}{2}\CJ_{\ve}\BOm^{\ve}\cdot\CJ_{\ve}\aaa^{\ve}_2+\frac{\eta_2}{\chi_2}\CJ_{\ve}\A^{\ve}\cdot\CJ_{\ve}\sss^{\ve}_4\Big)\CJ_{\ve}\nn^{\ve}_3\Big], \Delta^{s+1}\CJ_{\ve}\nn^{\ve}_1\Big\rangle\nonumber\\
&\qquad-\frac{1}{\chi_2}\big\langle\Delta^s\big(\CJ_{\ve}(\ML^{\ve}_2\CF_{Bi})\CJ_{\ve}\nn^{\ve}_3\big), \Delta^{s+1}\CJ_{\ve}\nn^{\ve}_1\big\rangle\nonumber\\
&\qquad-\big\langle\nabla\Delta^s(\CJ_{\ve}\vv^{\ve}\cdot\nabla\CJ_{\ve}\nn^{\ve}_1), \nabla\Delta^s\CJ_{\ve}\nn^{\ve}_1\big\rangle\nonumber\\
&\quad\eqdefa I_1+I_2+I_3+I_4+I_5.
\end{align}
Similarly, we have
\begin{align}\label{n1-higher-div-L2}
&\frac{1}{2}\frac{\ud}{\ud t}\big\langle\Delta^s{\rm div}\nn^{\ve}_1,\Delta^s{\rm div}\nn^{\ve}_1\big\rangle\nonumber\\
&\quad=-\Big\langle\Delta^s\Big[\Big(\frac{1}{2}\CJ_{\ve}\BOm^{\ve}\cdot\CJ_{\ve}\aaa^{\ve}_1+\frac{\eta_3}{\chi_3}\CJ_{\ve}\A^{\ve}\cdot\CJ_{\ve}\sss^{\ve}_3\Big)\CJ_{\ve}\nn^{\ve}_2\Big], \Delta^{s}\nabla{\rm div}\CJ_{\ve}\nn^{\ve}_1\Big\rangle\nonumber\\
&\qquad+\frac{1}{\chi_3}\big\langle\Delta^s\big(\CJ_{\ve}(\ML^{\ve}_3\CF_{Bi})\CJ_{\ve}\nn^{\ve}_2\big), \Delta^{s}\nabla{\rm div}\CJ_{\ve}\nn^{\ve}_1\big\rangle\nonumber\\
&\qquad+\Big\langle\Delta^s\Big[\Big(\frac{1}{2}\CJ_{\ve}\BOm^{\ve}\cdot\CJ_{\ve}\aaa^{\ve}_2+\frac{\eta_2}{\chi_2}\CJ_{\ve}\A^{\ve}\cdot\CJ_{\ve}\sss^{\ve}_4\Big)\CJ_{\ve}\nn^{\ve}_3\Big], \Delta^{s}\nabla{\rm div}\CJ_{\ve}\nn^{\ve}_1\Big\rangle\nonumber\\
&\qquad-\frac{1}{\chi_2}\big\langle\Delta^s\big(\CJ_{\ve}(\ML^{\ve}_2\CF_{Bi})\CJ_{\ve}\nn^{\ve}_3\big), \Delta^{s}\nabla{\rm div}\CJ_{\ve}\nn^{\ve}_1\big\rangle\nonumber\\
&\qquad-\big\langle\Delta^s{\rm div}(\CJ_{\ve}\vv^{\ve}\cdot\nabla\CJ_{\ve}\nn^{\ve}_1), \Delta^s{\rm div}\CJ_{\ve}\nn^{\ve}_1\big\rangle\nonumber\\
&\quad\eqdefa I'_1+I'_2+I'_3+I'_4+I'_5.
\end{align}
Using the incompressibility condition $\nabla\cdot\vv^{\ve}=0$ and Lemma \ref{product-estimate-lemma}, we deduce that
\begin{align}\label{CI+CI1-5}
I_5+I'_5=&-\big\langle[\nabla\Delta^s,\CJ_{\ve}\vv^{\ve}\cdot]\nabla\CJ_{\ve}\nn^{\ve}_1,\Delta^s\nabla\CJ_{\ve}\nn^{\ve}_1\big\rangle\nonumber\\
&-\big\langle[\Delta^s{\rm div},\CJ_{\ve}\vv^{\ve}\cdot]\nabla\CJ_{\ve}\nn^{\ve}_1,\Delta^s{\rm div}\CJ_{\ve}\nn^{\ve}_1\big\rangle\nonumber\\
\leq&\|[\nabla\Delta^s,\CJ_{\ve}\vv^{\ve}\cdot]\nabla\vv^{\ve}\|_{L^2}\|\Delta^s\nabla\CJ_{\ve}\nn^{\ve}_1\|_{L^2}\nonumber\\
&+\|[\Delta^s{\rm div},\CJ_{\ve}\vv^{\ve}\cdot]\nabla\CJ_{\ve}\nn^{\ve}_1\|_{L^2}\|\Delta^s{\rm div}\CJ_{\ve}\nn^{\ve}_1\|_{L^2}\nonumber\\
\leq& C\big(\|\nabla\nn^{\ve}_1\|_{H^{2s}}\|\nabla\vv^{\ve}\|_{L^{\infty}}+\|\nabla\CJ_{\ve}\vv^{\ve}\|_{H^{2s}}\|\nabla\nn^{\ve}_1\|_{L^{\infty}}\big)\|\nabla\nn^{\ve}_1\|_{H^{2s}}\nonumber\\
\leq& C_{\delta}(\|\nabla\vv^{\ve}\|_{L^{\infty}}+\|\nabla\nn^{\ve}_1\|^2_{L^{\infty}})\|\nabla\nn^{\ve}_1\|^2_{H^{2s}}+\delta\|\nabla\CJ_{\ve}\vv^{\ve}\|^2_{H^{2s}},
\end{align}
where $\delta$ represents a small positive constant to be determined later, whose value may vary in different places.

For simplicity, we denote
\begin{align*}
\CJ_{\ve}\CH^{\Delta^s}_{1}=&\CJ_{\ve}\nn^{\ve}_2\cdot\CJ_{\ve}\Delta^s\hh^{\ve}_3-\CJ_{\ve}\nn^{\ve}_3\cdot\CJ_{\ve}\Delta^s\hh^{\ve}_2,\\
\CJ_{\ve}\CH^{\Delta^s}_{2}=&\CJ_{\ve}\nn^{\ve}_3\cdot\CJ_{\ve}\Delta^s\hh^{\ve}_1-\CJ_{\ve}\nn^{\ve}_1\cdot\CJ_{\ve}\Delta^s\hh^{\ve}_3,\\
\CJ_{\ve}\CH^{\Delta^s}_{3}=&\CJ_{\ve}\nn^{\ve}_1\cdot\CJ_{\ve}\Delta^s\hh^{\ve}_2-\CJ_{\ve}\nn^{\ve}_2\cdot\CJ_{\ve}\Delta^s\hh^{\ve}_1.
\end{align*}
Moreover, from (\ref{CJ-hi}), we also define
\begin{align}\label{Deltas-CJ-hh-ve-comp}
\Delta^s\CJ_{\ve}\hh^{\ve}_i\eqdefa\CL^{\ve}_i(\Fp^{\ve})+\CG^{\ve}_i(\Fp^{\ve}),\quad i=1,2,3,
\end{align}
where $\CL^{\ve}_i(\Fp^{\ve})$ and $\CG^{\ve}_i(\Fp^{\ve})$ are higher-order and lower-order derivative terms, respectively,
\begin{align*}
\CL^{\ve}_i(\Fp^{\ve})=&\gamma_i\Delta^{s+1}\CJ_{\ve}\nn^{\ve}_i+k_i\Delta^s\nabla{\rm div}\CJ_{\ve}\nn^{\ve}_i
-\sum^3_{j=1}k_{ji}\Delta^s\nabla\times\big(\nabla\times\CJ_{\ve}\nn^{\ve}_i\cdot(\CJ_{\ve}\nn^{\ve}_j)^2\big),\\
\CG^{\ve}_i(\Fp^{\ve})=&-\sum^3_{j=1}k_{ij}\Delta^s\big[(\CJ_{\ve}\nn^{\ve}_i\cdot\nabla\times\CJ_{\ve}\nn^{\ve}_j)(\nabla\times\CJ_{\ve}\nn^{\ve}_j)\big].
\end{align*}

Using the approximate system (\ref{n1-ve-equation})--(\ref{n3-ve-equation}),
and the expressions of $\CJ_{\ve}\hh^{\ve}_i$, it follows that
\begin{align}\label{CJ-h-Linfty}
\|\partial_t\nn^{\ve}_i\|_{L^{\infty}}\leq& C\CP\big(\|\vv^{\ve}\|_{L^{\infty}},\|\nabla\vv^{\ve}\|_{L^{\infty}},\|\Fp^{\ve}\|_{L^{\infty}},\|\nabla\Fp^{\ve}\|_{L^{\infty}},\|\nabla^2\Fp^{\ve}\|_{L^{\infty}}\big),\quad i=1,2,3.
\end{align}
By a direct calculation,
we obtain from (\ref{CJ-h-Linfty}) that
\begin{align}\label{n1-K2-energy-estimate}
&\frac{1}{2}\frac{\ud}{\ud t}\big\|\CJ_{\ve}\nn^{\ve}_j\cdot\Delta^{s}(\nabla\times\nn^{\ve}_1)\big\|^2_{L^2}\nonumber\\
&\quad=\big\langle\CJ_{\ve}\partial_t\nn^{\ve}_j\cdot\Delta^{s}(\nabla\times\nn^{\ve}_1)+\CJ_{\ve}\nn^{\ve}_j\cdot(\Delta^{s}\nabla\times\dot{\nn}^{\ve}_1), \CJ_{\ve}\nn^{\ve}_j\cdot\Delta^{s}(\nabla\times\nn^{\ve}_1)\big\rangle\nonumber\\
&\qquad-\big\langle\CJ_{\ve}\nn^{\ve}_j\cdot\big(\Delta^{s}\nabla\times\CJ_{\ve}(\CJ_{\ve}\vv^{\ve}\cdot\nabla\CJ_{\ve}\nn^{\ve}_1)\big),\CJ_{\ve}\nn^{\ve}_j\cdot\Delta^{s}(\nabla\times\nn^{\ve}_1)\big\rangle\nonumber\\
&\quad\leq\|\partial_t\nn^{\ve}_j\|_{L^{\infty}}\|\nn^{\ve}_j\|_{L^{\infty}}\|\nabla\nn^{\ve}_1\|^2_{H^{2s}}\nonumber\\
&\qquad+\big\langle\CJ_{\ve}\nn^{\ve}_j\cdot(\Delta^{s}\nabla\times\dot{\nn}^{\ve}_1),\CJ_{\ve}\nn^{\ve}_j\cdot\Delta^{s}(\nabla\times\nn^{\ve}_1)\big\rangle\nonumber\\
&\qquad-\big\langle\CJ_{\ve}\nn^{\ve}_j\cdot\big(\Delta^{s}\nabla\times\CJ_{\ve}(\CJ_{\ve}\vv^{\ve}\cdot\nabla\CJ_{\ve}\nn^{\ve}_1)\big),\CJ_{\ve}\nn^{\ve}_j\cdot\Delta^{s}(\nabla\times\nn^{\ve}_1)\big\rangle\nonumber\\
&\quad\leq C\CP\big(\|\vv^{\ve}\|_{L^{\infty}},\|\nabla\vv^{\ve}\|_{L^{\infty}},\|\Fp^{\ve}\|_{L^{\infty}},\|\nabla\Fp^{\ve}\|_{L^{\infty}},\|\nabla^2\Fp^{\ve}\|_{L^{\infty}}\big)\|\nabla\nn^{\ve}_1\|^2_{H^{2s}}\nonumber\\
&\qquad\underbrace{+\big\langle\CJ_{\ve}\nn^{\ve}_j\cdot(\Delta^{s}\nabla\times\dot{\nn}^{\ve}_1),\CJ_{\ve}\nn^{\ve}_j\cdot\Delta^{s}(\nabla\times\nn^{\ve}_1)\big\rangle}_{J_1}\nonumber\\
&\qquad\underbrace{-\big\langle\Delta^{s}\nabla\times(\CJ_{\ve}\vv^{\ve}\cdot\nabla\CJ_{\ve}\nn^{\ve}_1),[\CJ_{\ve},(\CJ_{\ve}\nn^{\ve}_j)^2\cdot]\Delta^{s}(\nabla\times\nn^{\ve}_1)\big\rangle}_{J_2}\nonumber\\
&\qquad\underbrace{-\big\langle[\Delta^{s}\nabla\times,\CJ_{\ve}\vv^{\ve}\cdot]\nabla\CJ_{\ve}\nn^{\ve}_1,(\CJ_{\ve}\nn^{\ve}_j)^2\cdot(\Delta^{s}\nabla\times\CJ_{\ve}\nn^{\ve}_1)\big\rangle}_{J_3}\nonumber\\
&\qquad\underbrace{-\big\langle\CJ_{\ve}\vv^{\ve}\cdot\nabla(\Delta^{s}\nabla\times\CJ_{\ve}\nn^{\ve}_1),(\CJ_{\ve}\nn^{\ve}_j)^2\cdot(\Delta^{s}\nabla\times\CJ_{\ve}\nn^{\ve}_1)\big\rangle}_{J_4},
\end{align}
where $\dot{\nn}^{\ve}_1=\partial_t\nn^{\ve}_1+\CJ_{\ve}\big(\CJ_{\ve}\vv_{\ve}\cdot\nabla\CJ_{\ve}\nn^{\ve}_1\big)$.

For the term $J_2$, applying Lemma \ref{product-estimate-lemma} and Lemma \ref{molifi-lemma}, we have
\begin{align*}
J_2&\leq\big|\big\langle\Delta^{s-1}\nabla_{k}\nabla\times(\CJ_{\ve}\vv_{\ve}\cdot\nabla\CJ_{\ve}\nn^{\ve}_1),[\CJ_{\ve},(\CJ_{\ve}\nn^{\ve}_j)^2\cdot]\Delta^{s}(\nabla\times\nabla_{k}\nn^{\ve}_1)\big\rangle\big|\\
&\quad+C\|\nabla^{2s}(\CJ_{\ve}\vv_{\ve}\cdot\nabla\CJ_{\ve}\nn^{\ve}_1)\|_{L^2}\|\nn^{\ve}_j\|_{L^{\infty}}\|\nabla\nn^{\ve}_j\|_{L^{\infty}}\|\nabla\nn^{\ve}_1\|_{H^{2s}}\\
&\leq C\|\nabla^{2s}(\CJ_{\ve}\vv_{\ve}\cdot\nabla\CJ_{\ve}\nn^{\ve}_1)\|_{L^2}(1+\|\nabla\nn^{\ve}_j\|_{L^{\infty}}\|\nn^{\ve}_j\|_{L^{\infty}})\|\Delta^{s}\nabla\times\CJ_{\ve}\nn^{\ve}_1\|_{L^2}\\
&\quad+C\big(\|\vv^{\ve}\|_{L^{\infty}}\|\nabla\nn^{\ve}_1\|_{H^{2s}}+\|\nabla\nn^{\ve}_1\|_{L^{\infty}}\|\vv^{\ve}\|_{H^{2s}}\big)\|\nn^{\ve}_j\|_{L^{\infty}}\|\nabla\nn^{\ve}_j\|_{L^{\infty}}\|\nabla\nn^{\ve}_1\|_{H^{2s}}\\
&\leq C\CP\big(\|\vv^{\ve}\|_{L^{\infty}},\|\nabla\vv^{\ve}\|_{L^{\infty}},\|\Fp^{\ve}\|_{L^{\infty}},\|\nabla\Fp^{\ve}\|_{L^{\infty}}\big)(\|\nabla\nn^{\ve}_1\|_{H^{2s}}^2+\|\vv^{\ve}\|^2_{H^{2s}}).
\end{align*}
By Lemma \ref{product-estimate-lemma}, the incompressibility condition $\nabla\cdot\vv^{\ve}=0$ and integration by parts, the terms $J_3$ and $J_4$ can be estimated as follows:
\begin{align*}
J_3&\leq\|[\Delta^{s}\nabla\times,\CJ_{\ve}\vv^{\ve}\cdot]\nabla\CJ_{\ve}\nn^{\ve}_1\|_{L^2}\|(\CJ_{\ve}\nn^{\ve}_j)^2\cdot(\Delta^{s}\nabla\times\CJ_{\ve}\nn^{\ve}_1)\|_{L^2}\\
&\leq C(\|\nabla\vv^{\ve}\|_{L^{\infty}}\|\nabla\nn^{\ve}_1\|_{H^{2s}}+\|\nabla\CJ_{\ve}\vv^{\ve}\|_{H^{2s}}\|\nabla\nn^{\ve}_1\|_{L^{\infty}})(1+\|\nn^{\ve}_j\|_{L^{\infty}}^2)\|\nabla\nn^{\ve}_1\|_{H^{2s}}\\
&\leq C_{\delta}\CP\big(\|\Fp^{\ve}\|_{L^{\infty}},\|\nabla\Fp^{\ve}\|_{L^{\infty}},\|\nabla\vv^{\ve}\|_{L^{\infty}}\big)\|\nabla\nn^{\ve}_1\|_{H^{2s}}^2+\delta\|\nabla\CJ_{\ve}\vv^{\ve}\|_{H^{2s}}^2,\\
J_4&=-\frac{1}{2}\int_{\mathbb{R}^d}\CJ_{\ve}\vv^{\ve}\cdot\nabla\big[\CJ_{\ve}\nn^{\ve}_j\cdot(\Delta^s\nabla\times\CJ_{\ve}\nn^{\ve}_1)\big]^2\ud\xx\\
&\quad+\int_{\mathbb{R}^d}\CJ_{\ve}v^{\ve}_{\alpha}(\Delta^{s}\nabla\times\CJ_{\ve}\nn^{\ve}_1)_{\beta}(\Delta^{s}\nabla\times\CJ_{\ve}\nn^{\ve}_1)_{\gamma}\CJ_{\ve}n^{\ve}_{j\gamma}\partial_{\alpha}(\CJ_{\ve}n^{\ve}_{j\beta})\ud\xx\\
&\leq\|\vv^{\ve}\|_{L^{\infty}}\|\nn^{\ve}_j\|_{L^{\infty}}\|\nabla\nn^{\ve}_j\|_{L^{\infty}}\|\nabla\nn^{\ve}_1\|^2_{H^{2s}}.
\end{align*}

We are now ready to control the term $J_1$. Define $\CU^{\ve}$ by
\begin{align*}
\CU^{\ve}\eqdefa&\Big(\frac{1}{2}\CJ_{\ve}\BOm^{\ve}\cdot\CJ_{\ve}\aaa^{\ve}_1
+\frac{\eta_3}{\chi_3}\CJ_{\ve}\A^{\ve}\cdot\CJ_{\ve}\sss^{\ve}_3-\frac{1}{\chi_3}\CJ_{\ve}(\ML^{\ve}_3\CF_{Bi})\Big)\CJ_{\ve}\nn^{\ve}_2\\
&-\Big(\frac{1}{2}\CJ_{\ve}\BOm^{\ve}\cdot\CJ_{\ve}\aaa^{\ve}_2+\frac{\eta_2}{\chi_2}\CJ_{\ve}\A^{\ve}\cdot\CJ_{\ve}\sss^{\ve}_4
-\frac{1}{\chi_2}\CJ_{\ve}(\ML^{\ve}_2\CF_{Bi})\Big)\CJ_{\ve}\nn^{\ve}_3.
\end{align*}
Then, with the help of (\ref{n1-ve-equation}), Lemma \ref{product-estimate-lemma}, Lemma \ref{molifi-lemma}, and the analogous argument for the term $J_2$, the term $J_1$ can be handled as
\begin{align*}
J_1=&\big\langle\Delta^s\nabla\times\dot{\nn}^{\ve}_1,(\CJ_{\ve}\nn^{\ve}_j)^2\cdot\Delta^s(\nabla\times\nn^{\ve}_1)\big\rangle\\
=&\big\langle\Delta^s\CU^{\ve},\Delta^s\nabla\times\big(\nabla\times\CJ_{\ve}\nn^{\ve}_1\cdot(\CJ_{\ve}\nn^{\ve}_j)^2\big)\big\rangle\\
&+\Big\langle\Delta^s\nabla\times\CU^{\ve},[\CJ_{\ve},(\CJ_{\ve}\nn^{\ve}_j)^2\cdot]\Delta^s(\nabla\times\nn^{\ve}_1)-[\Delta^s,(\CJ_{\ve}\nn^{\ve}_j)^2\cdot](\nabla\times\CJ_{\ve}\nn^{\ve}_1)\Big\rangle\\
\leq&\underbrace{\Big\langle\Delta^s\Big[\Big(\frac{1}{2}\CJ_{\ve}\BOm^{\ve}\cdot\CJ_{\ve}\aaa^{\ve}_1+\frac{\eta_3}{\chi_3}\CJ_{\ve}\A^{\ve}\cdot\CJ_{\ve}\sss^{\ve}_3\Big)\CJ_{\ve}\nn^{\ve}_2\Big], \Delta^s\nabla\times\big(\nabla\times\CJ_{\ve}\nn^{\ve}_1\cdot(\CJ_{\ve}\nn^{\ve}_j)^2\big)\Big\rangle}_{J_{11}}\\
&\underbrace{-\frac{1}{\chi_3}\big\langle\Delta^s\big(\CJ_{\ve}(\ML^{\ve}_3\CF_{Bi})\CJ_{\ve}\nn^{\ve}_2\big), \Delta^s\nabla\times\big(\nabla\times\CJ_{\ve}\nn^{\ve}_1\cdot(\CJ_{\ve}\nn^{\ve}_j)^2\big)\big\rangle}_{J_{12}}\\
&\underbrace{-\Big\langle\Delta^s\Big[\Big(\frac{1}{2}\CJ_{\ve}\BOm^{\ve}\cdot\CJ_{\ve}\aaa^{\ve}_2+\frac{\eta_2}{\chi_2}\CJ_{\ve}\A^{\ve}\cdot\CJ_{\ve}\sss^{\ve}_4\Big)\CJ_{\ve}\nn^{\ve}_3\Big], \Delta^s\nabla\times\big(\nabla\times\CJ_{\ve}\nn^{\ve}_1\cdot(\CJ_{\ve}\nn^{\ve}_j)^2\big)\Big\rangle}_{J_{13}}\\
&+\underbrace{\frac{1}{\chi_2}\big\langle\Delta^s\big(\CJ_{\ve}(\ML^{\ve}_2\CF_{Bi})\CJ_{\ve}\nn^{\ve}_3\big), \Delta^s\nabla\times\big(\nabla\times\CJ_{\ve}\nn^{\ve}_1\cdot(\CJ_{\ve}\nn^{\ve}_j)^2\big)\big\rangle}_{J_{14}}\\
&+C_{\delta}\CP\big(\|\nabla\vv^{\ve}\|_{L^{\infty}},\|\Fp^{\ve}\|_{L^{\infty}},\|\nabla\Fp^{\ve}\|_{L^{\infty}},\|\nabla^2\Fp^{\ve}\|_{L^{\infty}}\big)\|\nabla\Fp^{\ve}\|^2_{H^{2s}}\\
&+\delta\big(\|\nabla\CJ_{\ve}\vv^{\ve}\|^2_{H^{2s}}+\|\CJ_{\ve}\CH^{\Delta^s}_{2}\|^2_{L^2}+\|\CJ_{\ve}\CH^{\Delta^s}_{3}\|^2_{L^2}\big),
\end{align*}
where $\|\nabla\Fp^{\ve}\|^2_{H^{2s}}=\sum^3_{i=1}\|\nabla\nn^{\ve}_i\|^2_{H^{2s}}$, and we have used the following estimates:
\begin{align*}
&\|\Delta^s\CU^{\ve}\|_{L^2}
\leq C\CP\big(\|\nabla\vv^{\ve}\|_{L^{\infty}},\|\Fp^{\ve}\|_{L^{\infty}},\|\nabla\Fp^{\ve}\|_{L^{\infty}},\|\nabla^2\Fp^{\ve}\|_{L^{\infty}}\big)\Big(\|\nabla\Fp^{\ve}\|_{H^{2s}}\\
&\qquad\qquad\qquad+\|\nabla\CJ_{\ve}\vv^{\ve}\|_{H^{2s}}+\|\CJ_{\ve}\CH^{\Delta^s}_{2}\|_{L^2}+\|\CJ_{\ve}\CH^{\Delta^s}_{3}\|_{L^2}\Big),\\
&\Big|\Big\langle\Delta^s\CU^{\ve},\nabla\times\Big([\CJ_{\ve},(\CJ_{\ve}\nn^{\ve}_j)^2\cdot]\Delta^s(\nabla\times\nn^{\ve}_1)-[\Delta^s,(\CJ_{\ve}\nn^{\ve}_j)^2\cdot](\nabla\times\CJ_{\ve}\nn^{\ve}_1)\Big)\Big\rangle\Big|\\
&\quad\leq C\|\Delta^s\CU^{\ve}\|_{L^2}\Big(\|[\CJ_{\ve},(\CJ_{\ve}\nn^{\ve}_j)^2\cdot]\Delta^s\nabla\times(\nabla\times\nn^{\ve}_1)\|_{L^2}\\
&\qquad+\|[\Delta^s,(\CJ_{\ve}\nn^{\ve}_j)^2\cdot]\nabla\times(\nabla\times\CJ_{\ve}\nn^{\ve}_1)\|_{L^2}+\|\nn^{\ve}_j\|_{L^{\infty}}\|\nabla\nn^{\ve}_j\|_{L^{\infty}}\|\nabla\nn_1\|_{H^{2s}}\Big)\\
&\quad\leq C_{\delta}\CP\big(\|\nabla\vv^{\ve}\|_{L^{\infty}},\|\Fp^{\ve}\|_{L^{\infty}},\|\nabla\Fp^{\ve}\|_{L^{\infty}},\|\nabla^2\Fp^{\ve}\|_{L^{\infty}}\big)\|\nabla\Fp^{\ve}\|^2_{H^{2s}}\\
&\qquad+\delta\big(\|\nabla\CJ_{\ve}\vv^{\ve}\|^2_{H^{2s}}+\|\CJ_{\ve}\CH^{\Delta^s}_{2}\|^2_{L^2}+\|\CJ_{\ve}\CH^{\Delta^s}_{3}\|^2_{L^2}\big).
\end{align*}
Consequently, from (\ref{n1-K2-energy-estimate}) and the estimates of the terms $J_i(i=1,\cdots,4)$, one can obtain
\begin{align}\label{nj-Deltas-nn1-L2}
&\frac{1}{2}\frac{\ud}{\ud t}\|\CJ_{\ve}\nn_j\cdot\Delta^s(\nabla\times\nn^{\ve}_1)\|^2_{L^2}\nonumber\\
&\quad\leq J_{11}+J_{12}+J_{13}+J_{14}\nonumber\\
&\qquad+C_{\delta}\CP\big(\|\vv^{\ve}\|_{L^{\infty}},\|\nabla\vv^{\ve}\|_{L^{\infty}},\|\Fp^{\ve}\|_{L^{\infty}},\|\nabla\Fp^{\ve}\|_{L^{\infty}},\|\nabla^2\Fp^{\ve}\|_{L^{\infty}}\big)\CE_s\big(\Fp^{\ve},\vv^{\ve}\big)\nonumber\\
&\qquad+\delta\big(\|\nabla\CJ_{\ve}\vv^{\ve}\|^2_{H^{2s}}+\|\CJ_{\ve}\CH^{\Delta^s}_{2}\|^2_{L^2}+\|\CJ_{\ve}\CH^{\Delta^s}_{3}\|^2_{L^2}\big).
\end{align}
Then, using the definition of $\CJ_{\ve}\hh_1$ and (\ref{Deltas-CJ-hh-ve-comp}), we deduce that
\begin{align}\label{sum-4+beta+CI+CJ1-beta}
&\sum^4_{\beta=1}\Big(\gamma_1I_{\beta}+k_1I'_{\beta}+\sum^3_{j=1}k_{j1}J_{1\beta}\Big)\nonumber\\
&\quad=-\big\langle\Delta^s\CU^{\ve},\Delta^s\CJ_{\ve}\hh^{\ve}_1-\CG^{\ve}_1(\Fp^{\ve})\big\rangle\nonumber\\
&\quad\leq-\Big\langle\Big(\frac{1}{2}\CJ_{\ve}\Delta^s\BOm^{\ve}\cdot\CJ_{\ve}\aaa^{\ve}_1+\frac{\eta_3}{\chi_3}\CJ_{\ve}\Delta^s\A^{\ve}\cdot\CJ_{\ve}\sss^{\ve}_3\Big)\CJ_{\ve}\nn^{\ve}_2, \Delta^s\CJ_{\ve}\hh^{\ve}_1\Big\rangle\nonumber\\
&\qquad+\Big\langle\Big(\frac{1}{2}\CJ_{\ve}\Delta^s\BOm^{\ve}\cdot\CJ_{\ve}\aaa^{\ve}_2+\frac{\eta_2}{\chi_2}\CJ_{\ve}\Delta^s\A^{\ve}\cdot\CJ_{\ve}\sss^{\ve}_4\Big)\CJ_{\ve}\nn^{\ve}_3, \Delta^s\CJ_{\ve}\hh^{\ve}_1\Big\rangle\nonumber\\
&\qquad+\Big\langle-\frac{1}{\chi_2}(\CJ_{\ve}\CH^{\Delta^s}_2)\CJ_{\ve}\nn^{\ve}_3+\frac{1}{\chi_3}(\CJ_{\ve}\CH^{\Delta^s}_3)\CJ_{\ve}\nn^{\ve}_2,\Delta^s\CJ_{\ve}\hh^{\ve}_1\Big\rangle\nonumber\\
&\qquad+C\Big|\Big\langle\nabla_l[\Delta^s,\CJ_{\ve}\aaa^{\ve}_1\otimes\CJ_{\ve}\nn^{\ve}_2\cdot]\BOm^{\ve}+\nabla_l[\Delta^s,\CJ_{\ve}\aaa^{\ve}_2\otimes\CJ_{\ve}\nn^{\ve}_3\cdot]\BOm^{\ve},\Delta^{s-1}\nabla^l\CJ_{\ve}\hh^{\ve}_1\Big\rangle\Big|\nonumber\\
&\qquad+C\Big|\Big\langle\nabla_l[\Delta^s,\CJ_{\ve}\sss^{\ve}_3\otimes\CJ_{\ve}\nn^{\ve}_2\cdot]\A^{\ve}+\nabla_l[\Delta^s,\CJ_{\ve}\sss^{\ve}_4\otimes\CJ_{\ve}\nn^{\ve}_3\cdot]\A^{\ve},\Delta^{s-1}\nabla^l\CJ_{\ve}\hh^{\ve}_1\Big\rangle\Big|\nonumber\\
&\qquad+C\Big|\Big\langle\nabla_l[\Delta^s,\CJ_{\ve}\nn^{\ve}_2]\CJ_{\ve}(\ML^{\ve}_3\CF_{Bi})-\nabla_l[\Delta^s,\CJ_{\ve}\nn^{\ve}_3]\CJ_{\ve}(\ML^{\ve}_2\CF_{Bi}),\Delta^{s-1}\nabla^l\CJ_{\ve}\hh^{\ve}_1\Big\rangle\Big|\nonumber\\
&\qquad+\underbrace{\frac{1}{\chi_3}\Big\langle\big([\Delta^s,\CJ_{\ve}\nn^{\ve}_1\cdot]\hh^{\ve}_2-[\Delta^s,\CJ_{\ve}\nn^{\ve}_2\cdot]\hh^{\ve}_1\big)\CJ_{\ve}\nn^{\ve}_2,\Delta^s\CJ_{\ve}\hh^{\ve}_1\Big\rangle}_{\CB_1}\nonumber\\
&\qquad\underbrace{-\frac{1}{\chi_2}\Big\langle\big([\Delta^s,\CJ_{\ve}\nn^{\ve}_3\cdot]\hh^{\ve}_1-[\Delta^s,\CJ_{\ve}\nn^{\ve}_1\cdot]\hh^{\ve}_3\big)\CJ_{\ve}\nn^{\ve}_3,\Delta^s\CJ_{\ve}\hh^{\ve}_1\Big\rangle}_{\CB_2}\nonumber\\
&\qquad+\|\Delta^s\CU^{\ve}\|_{L^2}\|\CG^{\ve}_{1}(\Fp^{\ve})\|_{L^2}\nonumber\\
&\quad\leq-\Big\langle\Big(\frac{1}{2}\CJ_{\ve}\Delta^s\BOm^{\ve}\cdot\CJ_{\ve}\aaa^{\ve}_1+\frac{\eta_3}{\chi_3}\CJ_{\ve}\Delta^s\A^{\ve}\cdot\CJ_{\ve}\sss^{\ve}_3\Big)\CJ_{\ve}\nn^{\ve}_2, \Delta^s\CJ_{\ve}\hh^{\ve}_1\Big\rangle\nonumber\\
&\qquad+\Big\langle\Big(\frac{1}{2}\CJ_{\ve}\Delta^s\BOm^{\ve}\cdot\CJ_{\ve}\aaa^{\ve}_2+\frac{\eta_2}{\chi_2}\CJ_{\ve}\Delta^s\A^{\ve}\cdot\CJ_{\ve}\sss^{\ve}_4\Big)\CJ_{\ve}\nn^{\ve}_3, \Delta^s\CJ_{\ve}\hh^{\ve}_1\Big\rangle\nonumber\\
&\qquad+\Big\langle-\frac{1}{\chi_2}(\CJ_{\ve}\CH^{\Delta^s}_2)\CJ_{\ve}\nn^{\ve}_3+\frac{1}{\chi_3}(\CJ_{\ve}\CH^{\Delta^s}_3)\CJ_{\ve}\nn^{\ve}_2,\Delta^s\CJ_{\ve}\hh^{\ve}_1\Big\rangle\nonumber\\
&\qquad+\CB_1+\CB_2+C_{\delta}\CP\big(\|\nabla\vv^{\ve}\|_{L^{\infty}},\|\Fp^{\ve}\|_{L^{\infty}},\|\nabla\Fp^{\ve}\|_{L^{\infty}},\|\nabla^2\Fp^{\ve}\|_{L^{\infty}}\big)\|\nabla\Fp^{\ve}\|^2_{H^{2s}}\nonumber\\
&\qquad+\delta\big(\|\nabla\CJ_{\ve}\vv^{\ve}\|^2_{H^{2s}}+\|\CJ_{\ve}\CH^{\Delta^s}_{2}\|^2_{L^2}+\|\CJ_{\ve}\CH^{\Delta^s}_{3}\|^2_{L^2}\big),
\end{align}
where we have employed the estimate
\begin{align*}
\|\nabla^{s-1}\nabla_l\CJ_{\ve}\hh^{\ve}_1\|_{L^2}\leq C\CP\big(\|\Fp^{\ve}\|_{L^{\infty}},\|\nabla\Fp^{\ve}\|_{L^{\infty}},\|\nabla^2\Fp^{\ve}\|_{L^{\infty}}\big)\|\nabla\Fp^{\ve}\|^2_{H^{2s}}.
\end{align*}
Hence, combining (\ref{n1-higher-grid-L2})--(\ref{CI+CI1-5}) with (\ref{nj-Deltas-nn1-L2})--(\ref{sum-4+beta+CI+CJ1-beta}), we arrive at
\begin{align}\label{CE-n1}
&\frac{\ud}{\ud t}\CE^{s}_1(\Fp^{\ve})-\delta\big(\|\nabla\CJ_{\ve}\vv^{\ve}\|^2_{H^{2s}}+\|\CJ_{\ve}\CH^{\Delta^s}_{2}\|^2_{L^2}+\|\CJ_{\ve}\CH^{\Delta^s}_{3}\|^2_{L^2}\big)\nonumber\\
&\quad\leq-\Big\langle\Big(\frac{1}{2}\CJ_{\ve}\Delta^s\BOm^{\ve}\cdot\CJ_{\ve}\aaa^{\ve}_1+\frac{\eta_3}{\chi_3}\CJ_{\ve}\Delta^s\A^{\ve}\cdot\CJ_{\ve}\sss^{\ve}_3\Big)\CJ_{\ve}\nn^{\ve}_2, \Delta^s\CJ_{\ve}\hh^{\ve}_1\Big\rangle\nonumber\\
&\qquad+\Big\langle\Big(\frac{1}{2}\CJ_{\ve}\Delta^s\BOm^{\ve}\cdot\CJ_{\ve}\aaa^{\ve}_2+\frac{\eta_2}{\chi_2}\CJ_{\ve}\Delta^s\A^{\ve}\cdot\CJ_{\ve}\sss^{\ve}_4\Big)\CJ_{\ve}\nn^{\ve}_3, \Delta^s\CJ_{\ve}\hh^{\ve}_1\Big\rangle\nonumber\\
&\qquad+\Big\langle-\frac{1}{\chi_2}(\CJ_{\ve}\CH^{\Delta^s}_2)\CJ_{\ve}\nn^{\ve}_3+\frac{1}{\chi_3}(\CJ_{\ve}\CH^{\Delta^s}_3)\CJ_{\ve}\nn^{\ve}_2,\Delta^s\CJ_{\ve}\hh^{\ve}_1\Big\rangle+\CB_1+\CB_2\nonumber\\
&\qquad+C_{\delta}\CP\big(\|\vv^{\ve}\|_{L^{\infty}},\|\nabla\vv^{\ve}\|_{L^{\infty}},\|\Fp^{\ve}\|_{L^{\infty}},\|\nabla\Fp^{\ve}\|_{L^{\infty}},\|\nabla^2\Fp^{\ve}\|_{L^{\infty}}\big)\CE_s\big(\Fp^{\ve},\vv^{\ve}\big).
\end{align}
By an analysis similar to (\ref{CE-n1}), it holds that
\begin{align}
&\frac{\ud}{\ud t}\CE^{s}_2(\Fp^{\ve})-\delta\big(\|\nabla\CJ_{\ve}\vv^{\ve}\|^2_{H^{2s}}+\|\CJ_{\ve}\CH^{\Delta^s}_{1}\|^2_{L^2}+\|\CJ_{\ve}\CH^{\Delta^s}_{3}\|^2_{L^2}\big)\nonumber\\
&\quad\leq\Big\langle\Big(\frac{1}{2}\CJ_{\ve}\Delta^s\BOm^{\ve}\cdot\CJ_{\ve}\aaa^{\ve}_1+\frac{\eta_3}{\chi_3}\CJ_{\ve}\Delta^s\A^{\ve}\cdot\CJ_{\ve}\sss^{\ve}_3\Big)\CJ_{\ve}\nn^{\ve}_1, \Delta^s\CJ_{\ve}\hh^{\ve}_2\Big\rangle\nonumber\\
&\qquad-\Big\langle\Big(\frac{1}{2}\CJ_{\ve}\Delta^s\BOm^{\ve}\cdot\CJ_{\ve}\aaa^{\ve}_3+\frac{\eta_1}{\chi_1}\CJ_{\ve}\Delta^s\A^{\ve}\cdot\CJ_{\ve}\sss^{\ve}_5\Big)\CJ_{\ve}\nn^{\ve}_3, \Delta^s\CJ_{\ve}\hh^{\ve}_2\Big\rangle\nonumber\\
&\qquad+\Big\langle-\frac{1}{\chi_3}(\CJ_{\ve}\CH^{\Delta^s}_3)\CJ_{\ve}\nn^{\ve}_1+\frac{1}{\chi_1}(\CJ_{\ve}\CH^{\Delta^s}_1)\CJ_{\ve}\nn^{\ve}_3,\Delta^s\CJ_{\ve}\hh^{\ve}_2\Big\rangle+\CB'_1+\CB'_2\nonumber\\
&\qquad+C_{\delta}\CP\big(\|\vv^{\ve}\|_{L^{\infty}},\|\nabla\vv^{\ve}\|_{L^{\infty}},\|\Fp^{\ve}\|_{L^{\infty}},\|\nabla\Fp^{\ve}\|_{L^{\infty}},\|\nabla^2\Fp^{\ve}\|_{L^{\infty}}\big)\CE_s\big(\Fp^{\ve},\vv^{\ve}\big),\label{CE-n2}\\
&\frac{\ud}{\ud t}\CE^{s}_3(\Fp^{\ve})-\delta\big(\|\nabla\CJ_{\ve}\vv^{\ve}\|^2_{H^{2s}}+\|\CJ_{\ve}\CH^{\Delta^s}_{1}\|^2_{L^2}+\|\CJ_{\ve}\CH^{\Delta^s}_{2}\|^2_{L^2}\big)\nonumber\\
&\quad\leq-\Big\langle\Big(\frac{1}{2}\CJ_{\ve}\Delta^s\BOm^{\ve}\cdot\CJ_{\ve}\aaa^{\ve}_2+\frac{\eta_2}{\chi_2}\CJ_{\ve}\Delta^s\A^{\ve}\cdot\CJ_{\ve}\sss^{\ve}_4\Big)\CJ_{\ve}\nn^{\ve}_1, \Delta^s\CJ_{\ve}\hh^{\ve}_3\Big\rangle\nonumber\\
&\qquad+\Big\langle\Big(\frac{1}{2}\CJ_{\ve}\Delta^s\BOm^{\ve}\cdot\CJ_{\ve}\aaa^{\ve}_3+\frac{\eta_1}{\chi_1}\CJ_{\ve}\Delta^s\A^{\ve}\cdot\CJ_{\ve}\sss^{\ve}_5\Big)\CJ_{\ve}\nn^{\ve}_2, \Delta^s\CJ_{\ve}\hh^{\ve}_3\Big\rangle\nonumber\\
&\qquad+\Big\langle-\frac{1}{\chi_1}(\CJ_{\ve}\CH^{\Delta^s}_1)\CJ_{\ve}\nn^{\ve}_2+\frac{1}{\chi_2}(\CJ_{\ve}\CH^{\Delta^s}_2)\CJ_{\ve}\nn^{\ve}_1,\Delta^s\CJ_{\ve}\hh^{\ve}_3\Big\rangle+\CB''_1+\CB''_2\nonumber\\
&\qquad+C_{\delta}\CP\big(\|\vv^{\ve}\|_{L^{\infty}},\|\nabla\vv^{\ve}\|_{L^{\infty}},\|\Fp^{\ve}\|_{L^{\infty}},\|\nabla\Fp^{\ve}\|_{L^{\infty}},\|\nabla^2\Fp^{\ve}\|_{L^{\infty}}\big)\CE_s\big(\Fp^{\ve},\vv^{\ve}\big),\label{CE-n3}
\end{align}
where $\CB'_i,\CB''_i(i=1,2)$ are written as, respectively,
\begin{align*}
\CB'_1=&-\frac{1}{\chi_3}\Big\langle\big([\Delta^s,\CJ_{\ve}\nn^{\ve}_1\cdot]\hh^{\ve}_2-[\Delta^s,\CJ_{\ve}\nn^{\ve}_2\cdot]\hh^{\ve}_1\big)\CJ_{\ve}\nn^{\ve}_1,\Delta^s\CJ_{\ve}\hh^{\ve}_2\Big\rangle,\\
\CB'_2=&\frac{1}{\chi_1}\Big\langle\big([\Delta^s,\CJ_{\ve}\nn^{\ve}_2\cdot]\hh^{\ve}_3-[\Delta^s,\CJ_{\ve}\nn^{\ve}_3\cdot]\hh^{\ve}_2\big)\CJ_{\ve}\nn^{\ve}_3,\Delta^s\CJ_{\ve}\hh^{\ve}_2\Big\rangle,\\
\CB''_1=&-\frac{1}{\chi_1}\Big\langle\big([\Delta^s,\CJ_{\ve}\nn^{\ve}_2\cdot]\hh^{\ve}_3-[\Delta^s,\CJ_{\ve}\nn^{\ve}_3\cdot]\hh^{\ve}_2\big)\CJ_{\ve}\nn^{\ve}_2,\Delta^s\CJ_{\ve}\hh^{\ve}_3\Big\rangle,\\
\CB''_2=&\frac{1}{\chi_2}\Big\langle\big([\Delta^s,\CJ_{\ve}\nn^{\ve}_3\cdot]\hh^{\ve}_1-[\Delta^s,\CJ_{\ve}\nn^{\ve}_1\cdot]\hh^{\ve}_3\big)\CJ_{\ve}\nn^{\ve}_1,\Delta^s\CJ_{\ve}\hh^{\ve}_3\Big\rangle.
\end{align*}

From the definitions of $\CJ_{\ve}\CH^{\Delta^s}_i(i=1,2,3)$, direct calculations yield
\begin{align}\label{CB1-6}
&\sum^2_{\alpha=1}(\CB_{\alpha}+\CB'_{\alpha}+\CB''_{\alpha})\nonumber\\
&\quad=-\frac{1}{\chi_3}\big\langle[\Delta^s,\CJ_{\ve}\nn^{\ve}_1\cdot]\hh^{\ve}_2-[\Delta^s,\CJ_{\ve}\nn^{\ve}_2\cdot]\hh^{\ve}_1,\CJ_{\ve}\CH^{\Delta^s}_3\big\rangle\nonumber\\
&\qquad-\frac{1}{\chi_2}\big\langle[\Delta^s,\CJ_{\ve}\nn^{\ve}_3\cdot]\hh^{\ve}_1-[\Delta^s,\CJ_{\ve}\nn^{\ve}_1\cdot]\hh^{\ve}_3,\CJ_{\ve}\CH^{\Delta^s}_2\big\rangle\nonumber\\
&\qquad-\frac{1}{\chi_1}\big\langle[\Delta^s,\CJ_{\ve}\nn^{\ve}_2\cdot]\hh^{\ve}_3-[\Delta^s,\CJ_{\ve}\nn^{\ve}_3\cdot]\hh^{\ve}_2,\CJ_{\ve}\CH^{\Delta^s}_1\big\rangle\nonumber\\
&\quad\leq C_{\delta}\CP\big(\|\Fp^{\ve}\|_{L^{\infty}},\|\nabla\Fp^{\ve}\|_{L^{\infty}},\|\nabla^2\Fp^{\ve}\|_{L^{\infty}}\big)\|\nabla\Fp^{\ve}\|^2_{H^{2s}}+\delta\sum^3_{i=1}\|\CJ_{\ve}\CH^{\Delta^s}_i\|^2_{L^2}.
\end{align}

Combining (\ref{ve-energy-law})--(\ref{L-2-estimate-n1}) with (\ref{CE-n1})--(\ref{CB1-6}), and choosing $\delta>0$ small enough, we conclude that
\begin{align}\label{Fp-highderivative-L2+end}
&\frac{\ud}{\ud t}\Big(\frac{1}{2}\|\Fp^{\ve}-\Fp^{(0)}\|^2_{L^2}+\CF_{Bi}[\Fp^{\ve}]+\frac{1}{2}\|\vv^{\ve}\|^2_{L^2}+\sum^3_{i=1}\CE^s_i(\Fp^{\ve})\Big)\nonumber\\
&\qquad+\sum^3_{i=1}\frac{1}{2\chi_i}\|\CJ_{\ve}\CH^{\Delta^s}_i\|^2_{L^2}-\delta\|\nabla\CJ_{\ve}\vv_{\ve}\|_{H^{2s}}^2\nonumber\\
&\quad\leq\Big\langle\frac{\eta_3}{\chi_3}(\CJ_{\ve}\CH^{\Delta^s}_3)\CJ_{\ve}\sss^{\ve}_3+\frac{\eta_2}{\chi_2}(\CJ_{\ve}\CH^{\Delta^s}_2)\CJ_{\ve}\sss^{\ve}_4+\frac{\eta_1}{\chi_1}(\CJ_{\ve}\CH^{\Delta^s}_1)\CJ_{\ve}\sss^{\ve}_5,\Delta^s\CJ_{\ve}\A^{\ve}\Big\rangle\nonumber\\
&\qquad+\frac{1}{2}\Big\langle(\CJ_{\ve}\CH^{\Delta^s}_3)\CJ_{\ve}\aaa^{\ve}_1+(\CJ_{\ve}\CH^{\Delta^s}_2)\CJ_{\ve}\aaa^{\ve}_2+(\CJ_{\ve}\CH^{\Delta^s}_1)\CJ_{\ve}\aaa^{\ve}_3,\Delta^s\CJ_{\ve}\BOm^{\ve}\Big\rangle\nonumber\\
&\qquad+C_{\delta}\CP\Big(\|\vv^{\ve}\|_{L^{\infty}},\|\nabla\vv^{\ve}\|_{L^{\infty}},\|\Fp^{\ve}\|_{L^{\infty}},\|\nabla\Fp^{\ve}\|_{L^{\infty}},\|\nabla^2\Fp^{\ve}\|_{L^{\infty}}\Big)\CE_s(\Fp^{\ve},\vv^{\ve}).
\end{align}

Next, we turn to the estimate of the higher order derivatives for $\vv^{\ve}$. Acting the operator $\Delta^s$ on the equation (\ref{v-ve-equation}) and taking the inner product with $\Delta^s\vv^{\ve}$, we deduce that
\begin{align}\label{v-highderivative-L2}
&\frac{1}{2}\frac{\ud}{\ud t}\|\Delta^s\vv^{\ve}\|^2_{L^2}+\eta\|\nabla\Delta^s\CJ_{\ve}\vv^{\ve}\|^2_{L^2}\nonumber\\
&\quad=-\big\langle\Delta^s(\CJ_{\ve}\vv^{\ve}\cdot\nabla\CJ_{\ve}\vv^{\ve}),\Delta^s\nabla\CJ_{\ve}\vv^{\ve}\big\rangle-\big\langle\Delta^s(\CJ_{\ve}\sigma^{\ve}),\Delta^s\nabla\CJ_{\ve}\vv^{\ve}\big\rangle\nonumber\\
&\qquad+\langle\Delta^s\mathfrak{F}^{\ve},\Delta^s\CJ_{\ve}\vv^{\ve}\rangle,\nonumber\\
&\quad\eqdefa I+II+III.
\end{align}
We deal with (\ref{v-highderivative-L2}) term by term. It follows from Lemma \ref{product-estimate-lemma} that
\begin{align*}
I=\langle[\Delta^s,\CJ_{\ve}\vv^{\ve}\cdot]\nabla\CJ_{\ve}\vv^{\ve},\Delta^s\nabla\CJ_{\ve}\vv^{\ve}\rangle
\leq C\|\nabla\vv^{\ve}\|_{L^{\infty}}\|\vv^{\ve}\|^2_{H^{2s}}.
\end{align*}
Using the definition of the stress $\CJ_{\ve}\sigma^{\ve}$, the second term in (\ref{v-highderivative-L2}) can be rewritten as
\begin{align}\label{III-1234}
II\eqdefa&II_1+II_2+II_3+II_4,
\end{align}
where
\begin{align*}
II_1=&-\Big\langle\beta_1\Delta^s\big((\CJ_{\ve}\A^{\ve}\cdot\CJ_{\ve}\sss^{\ve}_1)\CJ_{\ve}\sss^{\ve}_1\big)
+\beta_0\Delta^s\big((\CJ_{\ve}\A^{\ve}\cdot\CJ_{\ve}\sss^{\ve}_2)\CJ_{\ve}\sss^{\ve}_1\big)\\
&\qquad+\beta_0\Delta^s\big((\CJ_{\ve}\A^{\ve}\cdot\CJ_{\ve}\sss^{\ve}_1)\CJ_{\ve}\sss^{\ve}_2\big)+\beta_2\Delta^s\big((\CJ_{\ve}\A^{\ve}\cdot\CJ_{\ve}\sss^{\ve}_2)\CJ_{\ve}\sss^{\ve}_2\big),\Delta^s\CJ_{\ve}\A^{\ve}\Big\rangle,\\
II_2=&-\Big\langle\Big(\beta_3-\frac{\eta^2_3}{\chi_3}\Big)\Delta^s\big((\CJ_{\ve}\A^{\ve}\cdot\CJ_{\ve}\sss^{\ve}_3)\CJ_{\ve}\sss^{\ve}_3\big)+\Big(\beta_4-\frac{\eta^2_2}{\chi_2}\Big)\Delta^s\big((\CJ_{\ve}\A^{\ve}\cdot\CJ_{\ve}\sss^{\ve}_4)\CJ_{\ve}\sss^{\ve}_4\big)\\
&\qquad+\Big(\beta_5-\frac{\eta^2_1}{\chi_1}\Big)\Delta^s\big((\CJ_{\ve}\A^{\ve}\cdot\CJ_{\ve}\sss^{\ve}_5)\CJ_{\ve}\sss^{\ve}_5\big), \Delta^s\CJ_{\ve}\A^{\ve}\Big\rangle,\\
II_3=&-\Big\langle\frac{\eta_3}{\chi_3}\Delta^s\big(\CJ_{\ve}(\ML^{\ve}_3\CF_{Bi})\CJ_{\ve}\sss^{\ve}_3\big)+\frac{\eta_2}{\chi_2}\Delta^s\big(\CJ_{\ve}(\ML^{\ve}_2\CF_{Bi})\CJ_{\ve}\sss^{\ve}_4\big)\\
&\qquad+\frac{\eta_1}{\chi_1}\Delta^s\big(\CJ_{\ve}(\ML^{\ve}_1\CF_{Bi})\CJ_{\ve}\sss^{\ve}_5\big),\Delta^s\CJ_{\ve}\A^{\ve}\Big\rangle,\\
II_4=&-\frac{1}{2}\Big\langle\Delta^s\big(\CJ_{\ve}(\ML^{\ve}_3\CF_{Bi})\CJ_{\ve}\aaa^{\ve}_1+\CJ_{\ve}(\ML^{\ve}_2\CF_{Bi})\CJ_{\ve}\aaa^{\ve}_2+\CJ_{\ve}(\ML^{\ve}_1\CF_{Bi})\CJ_{\ve}\aaa^{\ve}_3\big),\Delta^s\CJ_{\ve}\BOm^{\ve}\Big\rangle.
\end{align*}
By the symmetry of $II_1+II_2$, the relation $\beta^2_0\leq\beta_1\beta_2$ in (\ref{coefficient-conditions}) and Lemma \ref{product-estimate-lemma}, we have
\begin{align*}
&II_1+II_2\nonumber\\
&\quad\leq\underbrace{-\Big(\sum^2_{j=1}\beta_j\|\Delta^s\CJ_{\ve}\A^{\ve}\cdot\CJ_{\ve}\sss^{\ve}_j\|^2_{L^2}+2\beta_0\int_{\mathbb{R}^d}(\Delta^s\CJ_{\ve}\A^{\ve}\cdot\CJ_{\ve}\sss^{\ve}_1)(\Delta^s\CJ_{\ve}\A^{\ve}\cdot\CJ_{\ve}\sss^{\ve}_2)\ud\xx\Big)}_{\leq 0}\\
&\qquad-\Big(\beta_3-\frac{\eta^2_3}{\chi_3}\Big)\|\Delta^s\CJ_{\ve}\A^{\ve}\cdot\CJ_{\ve}\sss^{\ve}_3\|^2_{L^2}
-\Big(\beta_4-\frac{\eta^2_2}{\chi_2}\Big)\|\Delta^s\CJ_{\ve}\A^{\ve}\cdot\CJ_{\ve}\sss^{\ve}_4\|^2_{L^2}\\
&\qquad
-\Big(\beta_5-\frac{\eta^2_1}{\chi_1}\Big)\|\Delta^s\CJ_{\ve}\A^{\ve}\cdot\CJ_{\ve}\sss^{\ve}_5\|^2_{L^2}\\
&\qquad+C\sum^5_{\substack{i,j=1;\\i=j~\text{if}~ i,j\geq3}}\Big(\|[\Delta^s,\CJ_{\ve}\sss^{\ve}_i](\CJ_{\ve}\A^{\ve}\cdot\CJ_{\ve}\sss^{\ve}_j)\|_{L^2}+\|\sss^{\ve}_i\|_{L^{\infty}}\|[\Delta^s,\CJ_{\ve}\sss^{\ve}_j\cdot](\CJ_{\ve}\A^{\ve})\|_{L^2}\Big)\|\nabla\CJ_{\ve}\vv^{\ve}\|_{H^{2s}}\\
&\quad\leq C_{\delta}\CP\big(\|\nabla\vv^{\ve}\|_{L^{\infty}},\|\Fp^{\ve}\|_{L^{\infty}},\|\nabla\Fp^{\ve}\|_{L^{\infty}}\big)\big(\|\vv^{\ve}\|^2_{H^{2s}}+\|\nabla\Fp^{\ve}\|^2_{H^{2s}}\big)+\delta\|\nabla\CJ_{\ve}\vv^{\ve}\|^2_{H^{2s}}.
\end{align*}
The same techniques can be applied to deal with the latter two term in (\ref{III-1234}). Retaining the higher derivative terms and using Lemma \ref{product-estimate-lemma} leads to
\begin{align*}
&II_3+II_4\\
&\quad\leq -\Big\langle\frac{\eta_3}{\chi_3}(\CJ_{\ve}\CH^{\Delta^s}_3)\CJ_{\ve}\sss^{\ve}_3+\frac{\eta_2}{\chi_2}(\CJ_{\ve}\CH^{\Delta^s}_2)\CJ_{\ve}\sss^{\ve}_4+\frac{\eta_1}{\chi_1}(\CJ_{\ve}\CH^{\Delta^s}_1)\CJ_{\ve}\sss^{\ve}_5,\Delta^s\CJ_{\ve}\A^{\ve}\Big\rangle\\
&\qquad-\frac{1}{2}\Big\langle(\CJ_{\ve}\CH^{\Delta^s}_3)\CJ_{\ve}\aaa^{\ve}_1+(\CJ_{\ve}\CH^{\Delta^s}_2)\CJ_{\ve}\aaa^{\ve}_2+(\CJ_{\ve}\CH^{\Delta^s}_1)\CJ_{\ve}\aaa^{\ve}_3,\Delta^s\CJ_{\ve}\BOm^{\ve}\Big\rangle\\
&\qquad+C_{\delta}\CP\big(\|\Fp^{\ve}\|_{L^{\infty}},\|\nabla\Fp^{\ve}\|_{L^{\infty}},\|\nabla^2\Fp^{\ve}\|_{L^{\infty}}\big)\|\nabla\Fp^{\ve}\|^2_{H^{2s}}+\delta\|\nabla\CJ_{\ve}\vv^{\ve}\|^2_{H^{2s}}.
\end{align*}
For the term $III$, integrating by parts and using Lemma \ref{product-estimate-lemma} gives
\begin{align*}
III\leq C_{\delta}\CP\big(\|\Fp^{\ve}\|_{L^{\infty}},\|\nabla\Fp^{\ve}\|_{L^{\infty}}\big)\|\nabla\Fp^{\ve}\|^2_{H^{2s}}+\delta\|\nabla\CJ_{\ve}\vv^{\ve}\|^2_{H^{2s}}.
\end{align*}
Summing up the above estimates from $I$ to $III$, we immediately obtain from (\ref{v-highderivative-L2}) that
\begin{align}\label{v-highderivative-L2+end}
&\frac{1}{2}\frac{\ud}{\ud t}\|\Delta^s\vv^{\ve}\|^2_{L^2}+(\eta-\delta)\|\nabla\Delta^s\CJ_{\ve}\vv^{\ve}\|^2_{L^2}\nonumber\\
&\quad\leq -\Big\langle\frac{\eta_3}{\chi_3}(\CJ_{\ve}\CH^{\Delta^s}_3)\CJ_{\ve}\sss^{\ve}_3+\frac{\eta_2}{\chi_2}(\CJ_{\ve}\CH^{\Delta^s}_2)\CJ_{\ve}\sss^{\ve}_4+\frac{\eta_1}{\chi_1}(\CJ_{\ve}\CH^{\Delta^s}_1)\CJ_{\ve}\sss^{\ve}_5,\Delta^s\CJ_{\ve}\A^{\ve}\Big\rangle\nonumber\\
&\qquad-\frac{1}{2}\Big\langle(\CJ_{\ve}\CH^{\Delta^s}_3)\CJ_{\ve}\aaa^{\ve}_1+(\CJ_{\ve}\CH^{\Delta^s}_2)\CJ_{\ve}\aaa^{\ve}_2+(\CJ_{\ve}\CH^{\Delta^s}_1)\CJ_{\ve}\aaa^{\ve}_3,\Delta^s\CJ_{\ve}\BOm^{\ve}\Big\rangle\nonumber\\
&\qquad+C_{\delta}\CP\big(\|\vv^{\ve}\|_{L^{\infty}},\|\nabla\vv^{\ve}\|_{L^{\infty}},\|\Fp^{\ve}\|_{L^{\infty}},\|\nabla\Fp^{\ve}\|_{L^{\infty}},\|\nabla^2\Fp^{\ve}\|_{L^{\infty}}\big)\CE_s(\Fp^{\ve},\vv^{\ve}).
\end{align}
Therefore, using (\ref{Fp-highderivative-L2+end}) and (\ref{v-highderivative-L2+end}), discarding the cancellation terms and choosing $\delta>0$ small enough,  we may obtain
\begin{align}\label{energy-equality-total}
\frac{\ud}{\ud t}\CE_s(\Fp^{\ve},\vv^{\ve})+\eta\|\nabla\vv^{\ve}\|^2_{L^2}+\eta\|\nabla\Delta^s\vv^{\ve}\|^2_{L^2}\leq F\big(\CE_s(\Fp^{\ve},\vv^{\ve})\big),
\end{align}
where $F$ is an increasing function with $F(0)=0$. For $s\geq 2$, by virtue of (\ref{energy-equality-total}), then there exists $T>0$ depending only on $\CE_s\big(\Fp^{(0)},\vv^{(0)}\big)$ such that, for any $t\in [0,\min(T,T_{\ve})]$,
\begin{align*}
\CE_s(\Fp^{\ve},\vv^{\ve})+\eta\|\nabla\vv^{\ve}\|^2_{L^2}+\eta\|\nabla\Delta^s\vv^{\ve}\|^2_{L^2}\leq 2\CE_s\big(\Fp^{(0)},\vv^{(0)}\big).
\end{align*}
By a continuous argument, we deduce that $T_{\ve}\geq T$. Hence, we get a uniform estimate for the approximate solution on $[0,T]$. Furthermore, the existence of the solution can be obtained by the standard compactness argument.

{\bf Step 3}. {\it Uniqueness of the solution}.
Assume that $\big(\Fp^{(1)},\vv^{(1)}\big)$ and $\big(\Fp^{(2)},\vv^{(2)}\big)$ are two solutions to the system (\ref{new-frame-equation-n1})--(\ref{imcompressible-v}) with the same initial data, where the orthonormal frames $\Fp^{(i)}=\big(\nn^{(i)}_1,\nn^{(i)}_2,\nn^{(i)}_3\big) (i=1,2)$.

We set
\begin{align*}
\delta_{\nn_1}=&\nn^{(1)}_1-\nn^{(2)}_1,\quad\delta_{\nn_2}=\nn^{(1)}_2-\nn^{(2)}_2,\quad\delta_{\nn_3}=\nn^{(1)}_3-\nn^{(2)}_3,\\
\delta_{\vv}=&\vv^{(1)}-\vv^{(2)},\quad\delta_{\A}=\A^{(1)}-\A^{(2)},\quad\delta_{\BOm}=\BOm^{(1)}-\BOm^{(2)},\\
\delta_{\hh_1}=&\hh^{(1)}_1-\hh^{(2)}_1,\quad\delta_{\hh_2}=\hh^{(1)}_2-\hh^{(2)}_2,\quad\delta_{\hh_3}=\hh^{(1)}_3-\hh^{(2)}_3,\\
\delta_{\sss_1}=&\sss^{(1)}_1-\sss^{(2)}_1,\quad\delta_{\sss_2}=\sss^{(1)}_2-\sss^{(2)}_2,\quad\delta_{\sss_3}=\sss^{(1)}_3-\sss^{(2)}_3,\\
\delta_{\sss_4}=&\sss^{(1)}_4-\sss^{(2)}_4,\quad\delta_{\sss_5}=\sss^{(1)}_5-\sss^{(2)}_5,\quad\delta_{\aaa_1}=\aaa^{(1)}_1-\aaa^{(2)}_1,\\
\delta_{\aaa_2}=&\aaa^{(1)}_2-\aaa^{(2)}_2,\quad\delta_{\aaa_3}=\aaa^{(1)}_3-\aaa^{(2)}_3,
\quad\delta_{p}=p^{(1)}-p^{(2)},\\
\delta_{\ML_1}=&\nn^{(1)}_2\cdot\delta_{\hh_3}-\nn^{(1)}_3\cdot\delta_{\hh_2},\quad\delta_{\ML_2}=\nn^{(1)}_3\cdot\delta_{\hh_1}-\nn^{(1)}_1\cdot\delta_{\hh_3},\\
\delta_{\ML_3}=&\nn^{(1)}_1\cdot\delta_{\hh_2}-\nn^{(1)}_2\cdot\delta_{\hh_1}.
\end{align*}
To control the higher order derivative terms, we introduce the following functional:
\begin{align}
\widetilde{\CF}_{Bi}[\Fp^{(1)},\nabla\delta_{\Fp}]=&\int_{\mathbb{R}^d}\Big(\frac{1}{2}\sum^3_{i=1}\gamma_i|\nabla\delta_{\nn_i}|^2+W\big(\Fp^{(1)},\nabla\delta_{\Fp}\big)\Big)\ud\xx,
\end{align}
where $W\big(\Fp^{(1)},\nabla\delta_{\Fp}\big)$ are given by
\begin{align*}
W\big(\Fp^{(1)},\nabla\delta_{\Fp}\big)=\frac{1}{2}\Big(\sum^3_{i=1}k_i(\nabla\cdot\delta_{\nn_i})^2+\sum^3_{i,j=1}k_{ij}\big(\nn^{(1)}_i\cdot\nabla\times\delta_{\nn_j}\big)^2\Big).
\end{align*}
Then we have
\begin{align*}
\delta_{\hh_j}=\hh^{(1)}_j-\hh^{(2)}_j=-\frac{\delta\widetilde{\CF}_{Bi}}{\delta(\delta_{\nn_j})}+\tilde{\delta}_{\hh_j},\quad j=1,2,3,
\end{align*}
where
\begin{align*}
\|(\tilde{\delta}_{\hh_1},\tilde{\delta}_{\hh_2},\tilde{\delta}_{\hh_3})\|_{L^2}\leq C \Big(\|\delta_{\vv}\|_{L^2}+\sum^3_{i=1}\|\delta_{\nn_i}\|_{L^2}\Big).
\end{align*}

By taking the difference between the equations for $\big(\Fp^{(1)},\vv^{(1)}\big)$ and $\big(\Fp^{(2)},\vv^{(2)}\big)$, we find that
\begin{align}
\frac{\partial\delta_{\nn_1}}{\partial t}+\vv^{(1)}\cdot\nabla\delta_{\nn_1}=&\Big(\frac{1}{2}\delta_{\BOm}\cdot\aaa^{(1)}_1+\frac{\eta_3}{\chi_3}\delta_{\A}\cdot\sss^{(1)}_3-\frac{1}{\chi_3}\delta_{\ML_3}\Big)\nn^{(1)}_2\nonumber\\
&-\Big(\frac{1}{2}\delta_{\BOm}\cdot\aaa^{(1)}_2+\frac{\eta_2}{\chi_2}\delta_{\A}\cdot\sss^{(1)}_4-\frac{1}{\chi_2}\delta_{\ML_2}\Big)\nn^{(1)}_3+\delta_{\FF_1},\label{n1-uniquence-equ}\\
\frac{\partial\delta_{\nn_2}}{\partial t}+\vv^{(1)}\cdot\nabla\delta_{\nn_2}=&-\Big(\frac{1}{2}\delta_{\BOm}\cdot\aaa^{(1)}_1+\frac{\eta_3}{\chi_3}\delta_{\A}\cdot\sss^{(1)}_3-\frac{1}{\chi_3}\delta_{\ML_3}\Big)\nn^{(1)}_1\nonumber\\
&+\Big(\frac{1}{2}\delta_{\BOm}\cdot\aaa^{(1)}_3+\frac{\eta_1}{\chi_1}\delta_{\A}\cdot\sss^{(1)}_5-\frac{1}{\chi_1}\delta_{\ML_1}\Big)\nn^{(1)}_3+\delta_{\FF_2},\label{n2-uniquence-equ}\\
\frac{\partial\delta_{\nn_3}}{\partial t}+\vv^{(1)}\cdot\nabla\delta_{\nn_3}=&\Big(\frac{1}{2}\delta_{\BOm}\cdot\aaa^{(1)}_2+\frac{\eta_2}{\chi_2}\delta_{\A}\cdot\sss^{(1)}_4-\frac{1}{\chi_2}\delta_{\ML_2}\Big)\nn^{(1)}_1\nonumber\\
&-\Big(\frac{1}{2}\delta_{\BOm}\cdot\aaa^{(1)}_3+\frac{\eta_1}{\chi_1}\delta_{\A}\cdot\sss^{(1)}_5-\frac{1}{\chi_1}\delta_{\ML_1}\Big)\nn^{(1)}_2+\delta_{\FF_3},\label{n3-uniquence-equ}\\
\frac{\partial\delta_{\vv}}{\partial t}+\vv^{(1)}\cdot\nabla\delta_{\vv}=&-\nabla\delta_{p}+\eta\Delta\delta_{\vv}+\nabla\cdot \sigma(\Fp^{(1)},\delta_{\vv})+\delta_{\mathfrak{F}}+\nabla\cdot\delta_{\FF_4}+\delta_{\FF_5},\label{v-uniquence-equ}\\
\nabla\cdot\delta_{\vv}=&0,\label{incomp-uniquence-equ}
\end{align}
where the stress $\sigma(\Fp^{(1)},\delta_{\vv})$ is expressed by
\begin{align*}
\sigma(\Fp^{(1)},\delta_{\vv})=&\beta_1(\delta_{\A}\cdot\sss^{(1)}_1)\sss^{(1)}_1+\beta_0(\delta_{\A}\cdot\sss^{(1)}_2)\sss^{(1)}_1+\beta_0(\delta_{\A}\cdot\sss^{(1)}_1)\sss^{(1)}_2\\
&+\beta_2(\delta_{\A}\cdot\sss^{(1)}_2)\sss^{(1)}_2+\Big(\beta_3-\frac{\eta^2_3}{\chi_3}\Big)(\delta_{\A}\cdot\sss^{(1)}_3)\sss^{(1)}_3+\Big(\beta_4-\frac{\eta^2_2}{\chi_2}\Big)(\delta_{\A}\cdot\sss^{(1)}_4)\sss^{(1)}_4\\
&+\Big(\beta_5-\frac{\eta^2_1}{\chi_1}\Big)(\delta_{\A}\cdot\sss^{(1)}_5)\sss^{(1)}_5+\frac{\eta_3}{\chi_3}\sss^{(1)}_3\delta_{\ML_3}+\frac{\eta_2}{\chi_2}\sss^{(1)}_4\delta_{\ML_2}+\frac{\eta_1}{\chi_1}\sss^{(1)}_5\delta_{\ML_1}\\
&-\frac{1}{2}\big(\aaa^{(1)}_1\delta_{\ML_3}+\aaa^{(1)}_2\delta_{\ML_2}+\aaa^{(1)}_3\delta_{\ML_1}\big).
\end{align*}
The body force $\delta_{\mathfrak{F}}$ is given by
\begin{align*}
(\delta_{\mathfrak{F}})_i=&\partial_i\nn^{(1)}_1\cdot\nn^{(1)}_2\delta_{\ML_3}+\partial_i\nn^{(1)}_3\cdot\nn^{(1)}_1\delta_{\ML_2}+\partial_i\nn^{(1)}_2\cdot\nn^{(1)}_3\delta_{\ML_1}.
\end{align*}
Moreover, $\delta_{\FF_i}(i=1,\cdots,5)$ can be given by
\begin{align*}
\delta_{\FF_1}=&-\delta_{\vv}\cdot\nabla\nn^{(2)}_1+\Big(\frac{1}{2}\BOm^{(2)}\cdot\delta_{\aaa_1}+\frac{\eta_3}{\chi_3}\A^{(2)}\cdot\delta_{\sss_3}-\frac{1}{\chi_3}(\delta_{\nn_1}\cdot\hh^{(2)}_2-\delta_{\nn_2}\cdot\hh^{(2)}_1)\Big)\nn^{(1)}_2\nonumber\\
&+\Big(\frac{1}{2}\BOm^{(2)}\cdot\aaa^{(2)}_1+\frac{\eta_3}{\chi_3}\A^{(2)}\cdot\sss^{(2)}_3-\frac{1}{\chi_3}\ML^{(2)}_3\CF_{Bi}\Big)\delta_{\nn_2}\nonumber\\
&-\Big(\frac{1}{2}\BOm^{(2)}\cdot\delta_{\aaa_2}+\frac{\eta_2}{\chi_2}\A^{(2)}\cdot\delta_{\sss_4}-\frac{1}{\chi_2}(\delta_{\nn_3}\cdot\hh^{(2)}_1-\delta_{\nn_1}\cdot\hh^{(2)}_3)\Big)\nn^{(1)}_3\nonumber\\
&-\Big(\frac{1}{2}\BOm^{(2)}\cdot\aaa^{(2)}_2+\frac{\eta_2}{\chi_2}\A^{(2)}\cdot\sss^{(2)}_4-\frac{1}{\chi_2}\ML^{(2)}_2\CF_{Bi}\Big)\delta_{\nn_3},\\
\delta_{\FF_2}=&-\delta_{\vv}\cdot\nabla\nn^{(2)}_2-\Big(\frac{1}{2}\BOm^{(2)}\cdot\delta_{\aaa_1}+\frac{\eta_3}{\chi_3}\A^{(2)}\cdot\delta_{\sss_3}-\frac{1}{\chi_3}(\delta_{\nn_1}\cdot\hh^{(2)}_2-\delta_{\nn_2}\cdot\hh^{(2)}_1)\Big)\nn^{(1)}_1\nonumber\\
&-\Big(\frac{1}{2}\BOm^{(2)}\cdot\aaa^{(2)}_1+\frac{\eta_3}{\chi_3}\A^{(2)}\cdot\sss^{(2)}_3-\frac{1}{\chi_3}\ML^{(2)}_3\CF_{Bi}\Big)\delta_{\nn_1}\nonumber\\
&+\Big(\frac{1}{2}\BOm^{(2)}\cdot\delta_{\aaa_3}+\frac{\eta_1}{\chi_1}\A^{(2)}\cdot\delta_{\sss_5}-\frac{1}{\chi_1}(\delta_{\nn_2}\cdot\hh^{(2)}_3-\delta_{\nn_3}\cdot\hh^{(2)}_2)\Big)\nn^{(1)}_3\nonumber\\
&+\Big(\frac{1}{2}\BOm^{(2)}\cdot\aaa^{(2)}_3+\frac{\eta_1}{\chi_1}\A^{(2)}\cdot\sss^{(2)}_5-\frac{1}{\chi_1}\ML^{(2)}_1\CF_{Bi}\Big)\delta_{\nn_3},\\
\delta_{\FF_3}=&-\delta_{\vv}\cdot\nabla\nn^{(2)}_3+\Big(\frac{1}{2}\BOm^{(2)}\cdot\delta_{\aaa_2}+\frac{\eta_2}{\chi_2}\A^{(2)}\cdot\delta_{\sss_4}-\frac{1}{\chi_2}(\delta_{\nn_3}\cdot\hh^{(2)}_1-\delta_{\nn_1}\cdot\hh^{(2)}_3)\Big)\nn^{(1)}_1\nonumber\\
&+\Big(\frac{1}{2}\BOm^{(2)}\cdot\aaa^{(2)}_2+\frac{\eta_2}{\chi_2}\A^{(2)}\cdot\sss^{(2)}_4-\frac{1}{\chi_2}\ML^{(2)}_2\CF_{Bi}\Big)\delta_{\nn_1}\nonumber\\
&-\Big(\frac{1}{2}\BOm^{(2)}\cdot\delta_{\aaa_3}+\frac{\eta_1}{\chi_1}\A^{(2)}\cdot\delta_{\sss_5}-\frac{1}{\chi_1}(\delta_{\nn_2}\cdot\hh^{(2)}_3-\delta_{\nn_3}\cdot\hh^{(2)}_2)\Big)\nn^{(1)}_2\nonumber\\
&-\Big(\frac{1}{2}\BOm^{(2)}\cdot\aaa^{(2)}_3+\frac{\eta_1}{\chi_1}\A^{(2)}\cdot\sss^{(2)}_5-\frac{1}{\chi_1}\ML^{(2)}_1\CF_{Bi}\Big)\delta_{\nn_2},\\
\delta_{\FF_4}=&-\delta_{\vv}\otimes\vv^{(2)}
+\beta_1\big((\A^{(2)}\cdot\delta_{\sss_1})\sss^{(1)}_1+(\A^{(2)}\cdot\sss^{(2)}_1)\delta_{\sss_1}\big)\\
&+\beta_0\big((\A^{(2)}\cdot\delta_{\sss_2})\sss^{(1)}_1+(\A^{(2)}\cdot\sss^{(2)}_2)\delta_{\sss_1}\big)
+\beta_0\big((\A^{(2)}\cdot\delta_{\sss_1})\sss^{(1)}_2+(\A^{(2)}\cdot\sss^{(2)}_1)\delta_{\sss_2}\big)\\
&+\beta_2\big((\A^{(2)}\cdot\delta_{\sss_2})\sss^{(1)}_2+(\A^{(2)}\cdot\sss^{(2)}_2)\delta_{\sss_2}\big)
+\Big(\beta_3-\frac{\eta^2_3}{\chi_3}\Big)\big((\A^{(2)}\cdot\delta_{\sss_3})\sss^{(1)}_3+(\A^{(2)}\cdot\sss^{(2)}_3)\delta_{\sss_3}\big)\\
&+\Big(\beta_4-\frac{\eta^2_2}{\chi_2}\Big)\big((\A^{(2)}\cdot\delta_{\sss_4})\sss^{(1)}_4+(\A^{(2)}\cdot\sss^{(2)}_4)\delta_{\sss_4}\big)\\
&+\Big(\beta_5-\frac{\eta^2_1}{\chi_1}\Big)\big((\A^{(2)}\cdot\delta_{\sss_5})\sss^{(1)}_5+(\A^{(2)}\cdot\sss^{(2)}_5)\delta_{\sss_5}\big)\\
&+\frac{\eta_3}{\chi_3}\big(\sss^{(1)}_3(\delta_{\nn_1}\cdot\hh^{(2)}_2-\delta_{\nn_2}\cdot\hh^{(2)}_1)+\delta_{\sss_3}(\ML^{(2)}_3\CF_{Bi})\big)\\
&+\frac{\eta_2}{\chi_2}\big(\sss^{(1)}_4(\delta_{\nn_3}\cdot\hh^{(2)}_1-\delta_{\nn_1}\cdot\hh^{(2)}_3)+\delta_{\sss_4}(\ML^{(2)}_2\CF_{Bi})\big)\\
&+\frac{\eta_1}{\chi_1}\big(\sss^{(1)}_5(\delta_{\nn_2}\cdot\hh^{(2)}_3-\delta_{\nn_3}\cdot\hh^{(2)}_2)+\delta_{\sss_5}(\ML^{(2)}_1\CF_{Bi})\big)\\
&-\frac{1}{2}\big(\aaa^{(1)}_1(\delta_{\nn_1}\cdot\hh^{(2)}_2-\delta_{\nn_2}\cdot\hh^{(2)}_1)+\delta_{\aaa_1}(\ML^{(2)}_3\CF_{Bi})\big)\\
&-\frac{1}{2}\big(\aaa^{(1)}_2(\delta_{\nn_3}\cdot\hh^{(2)}_1-\delta_{\nn_1}\cdot\hh^{(2)}_3)+\delta_{\aaa_2}(\ML^{(2)}_2\CF_{Bi})\big)\\
&-\frac{1}{2}\big(\aaa^{(1)}_3(\delta_{\nn_2}\cdot\hh^{(2)}_3-\delta_{\nn_3}\cdot\hh^{(2)}_2)+\delta_{\aaa_3}(\ML^{(2)}_1\CF_{Bi})\big),\\
\delta_{\FF_5}=&\,
\partial_i\nn^{(1)}_1\cdot\nn^{(1)}_2(\delta_{\nn_1}\cdot\hh^{(2)}_2-\delta_{\nn_2}\cdot\hh^{(2)}_1)+\partial_i\nn^{(1)}_1\cdot\delta_{\nn_2}(\ML^{(2)}_3\CF_{Bi})+\partial_i\delta_{\nn_1}\cdot\nn^{(2)}_2(\ML^{(2)}_3\CF_{Bi})\\
&+\partial_i\nn^{(1)}_3\cdot\nn^{(1)}_1(\delta_{\nn_3}\cdot\hh^{(2)}_1-\delta_{\nn_1}\cdot\hh^{(2)}_3)+\partial_i\nn^{(1)}_3\cdot\delta_{\nn_1}(\ML^{(2)}_2\CF_{Bi})+\partial_i\delta_{\nn_3}\cdot\nn^{(2)}_1(\ML^{(2)}_2\CF_{Bi})\\
&+\partial_i\nn^{(1)}_2\cdot\nn^{(1)}_3(\delta_{\nn_2}\cdot\hh^{(2)}_3-\delta_{\nn_3}\cdot\hh^{(2)}_2)+\partial_i\nn^{(1)}_2\cdot\delta_{\nn_3}(\ML^{(2)}_1\CF_{Bi})+\partial_i\delta_{\nn_2}\cdot\nn^{(2)}_3(\ML^{(2)}_1\CF_{Bi}),
\end{align*}
where
\begin{align*}
&\ML^{(2)}_1\CF_{Bi}=\nn^{(2)}_2\cdot\hh^{(2)}_3-\nn^{(2)}_3\cdot\hh^{(2)}_2,~~\ML^{(2)}_2\CF_{Bi}=\nn^{(2)}_3\cdot\hh^{(2)}_1-\nn^{(2)}_1\cdot\hh^{(2)}_3,\\
&\ML^{(2)}_3\CF_{Bi}=\nn^{(2)}_1\cdot\hh^{(2)}_2-\nn^{(2)}_2\cdot\hh^{(2)}_1.
\end{align*}
A direct calculation leads to the following estimates:
\begin{align*}
\|(\delta_{\FF_1},\delta_{\FF_2},\delta_{\FF_3},\delta_{\FF_4})\|_{L^2}\leq& C\Big(\|\delta_{\vv}\|_{L^2}+\sum^3_{i=1}\|\delta_{\nn_i}\|_{L^2}\Big),\\
\|\delta_{\FF_5}\|_{L^2}\leq& C\sum^3_{i=1}\|\delta_{\nn_i}\|_{H^1}.
\end{align*}
Applying the expressions of $\frac{\delta\widetilde{\CF}_{Bi}}{\delta(\delta_{\nn_i})}(i=1,2,3)$ and integrating by parts, it follows that
\begin{align}\label{delta-FF-i}
\Big|\Big\langle\delta_{\FF_i},\frac{\delta\widetilde{\CF}_{Bi}}{\delta(\delta_{\nn_i})}\Big\rangle\Big|\leq C_{\delta}\Big(\|\delta_{\vv}\|^2_{L^2}+\sum^3_{j=1}\|\delta_{\nn_j}\|^2_{H^1}\Big)+\delta\|\nabla\delta_{\vv}\|^2_{L^2},\quad i=1,2,3.
\end{align}
Using the system (\ref{n1-uniquence-equ})--(\ref{n3-uniquence-equ}) and making $L^2$-energy estimates for $(\delta_{\nn_1},\delta_{\nn_2},\delta_{\nn_3})$, we immediately obtain
\begin{align}\label{delta-Fp-L2}
\frac{1}{2}\frac{\ud}{\ud t}\Big(\sum^3_{i=1}\|\delta_{\nn_i}\|^2_{L^2}\Big)\leq C_{\delta}\Big(\|\delta_{\vv}\|^2_{L^2}+\sum^3_{i=1}\|\delta_{\nn_i}\|^2_{L^2}\Big)+\delta\|\nabla\delta_{\vv}\|^2_{L^2}.
\end{align}
Similar to the derivation of energy law in (\ref{energy-law}), we can deduce from the system (\ref{n1-uniquence-equ})--(\ref{incomp-uniquence-equ}) and (\ref{delta-FF-i}) that
\begin{align}\label{delta-Fp-v-H1}
&\frac{\ud}{\ud t}\Big(\frac{1}{2}\|\delta_{\vv}\|^2_{L^2}+\widetilde{\CF}_{Bi}(\Fp^{(1)},\nabla\delta_{\Fp})\Big)+\eta\|\nabla\delta_{\vv}\|^2_{L^2}+\sum^3_{j=1}\frac{1}{\chi_j}\|\delta_{\ML_j}\|^2_{L^2}\nonumber\\
&\quad=-\bigg(\beta_1\|\delta_{\A}\cdot\sss^{(1)}_1\|^2_{L^2}+2\beta_0\int_{\mathbb{R}^d}(\delta_{\A}\cdot\sss^{(1)}_1)(\delta_{\A}\cdot\sss^{(1)}_2)\ud\xx+\beta_2\|\delta_{\A}\cdot\sss^{(1)}_2\|^2_{L^2}\bigg)\nonumber\\
&\qquad-\Big(\beta_3-\frac{\eta^2_3}{\chi_3}\Big)\|\delta_{\A}\cdot\sss^{(1)}_3\|^2_{L^2}
-\Big(\beta_4-\frac{\eta^2_2}{\chi_2}\Big)\|\delta_{\A}\cdot\sss^{(1)}_4\|^2_{L^2}
-\Big(\beta_5-\frac{\eta^2_1}{\chi_1}\Big)\|\delta_{\A}\cdot\sss^{(1)}_5\|^2_{L^2}\nonumber\\
&\qquad+\sum_{i=1}^3\Big\langle\frac{\delta\widetilde{\CF}_{Bi}}{\delta(\delta_{\nn_i})},\delta_{\FF_i}\Big\rangle-\big\langle\delta_{\FF_4},\nabla\delta_{\vv}\big\rangle+\langle\delta_{\FF_5},\delta_{\vv}\rangle-\sum_{i=1}^3\big\langle\vv^{(1)}\cdot\nabla\delta_{\nn_i},\tilde{\delta}_{\hh_i}\big\rangle\nonumber\\
&\qquad+\Big\langle\frac{1}{2}\delta_{\BOm}\cdot\aaa^{(1)}_1+\frac{\eta_3}{\chi_3}\delta_{\A}\cdot\sss^{(1)}_3-\frac{1}{\chi_3}\delta_{\ML_3},\nn^{(1)}_2\cdot\tilde{\delta}_{\hh_1}-\nn^{(1)}_1\cdot\tilde{\delta}_{\hh_2}\Big\rangle\nonumber\\
&\qquad+\Big\langle\frac{1}{2}\delta_{\BOm}\cdot\aaa^{(1)}_2+\frac{\eta_2}{\chi_2}\delta_{\A}\cdot\sss^{(1)}_4-\frac{1}{\chi_2}\delta_{\ML_2},\nn^{(1)}_1\cdot\tilde{\delta}_{\hh_3}-\nn^{(1)}_3\cdot\tilde{\delta}_{\hh_1}\Big\rangle\nonumber\\
&\qquad+\Big\langle\frac{1}{2}\delta_{\BOm}\cdot\aaa^{(1)}_3+\frac{\eta_1}{\chi_1}\delta_{\A}\cdot\sss^{(1)}_5-\frac{1}{\chi_1}\delta_{\ML_1},\nn^{(1)}_3\cdot\tilde{\delta}_{\hh_2}-\nn^{(1)}_2\cdot\tilde{\delta}_{\hh_3}\Big\rangle\nonumber\\
&\quad\leq C_{\delta}\Big(\|\delta_{\vv}\|^2_{L^2}+\sum^3_{j=1}\|\delta_{\nn_j}\|^2_{H^1}\Big)+\delta\Big(\|\nabla\delta_{\vv}\|^2_{L^2}+\sum^3_{j=1}\|\delta_{\ML_j}\|^2_{L^2}\Big).
\end{align}
Hence, combining (\ref{delta-Fp-L2}) with (\ref{delta-Fp-v-H1}) yields that
\begin{align*}
\frac{\ud}{\ud t}\Big(\|\delta_{\vv}\|^2_{L^2}+\sum^3_{i=1}\|\delta_{\nn_i}\|^2_{H^1}\Big)\leq C \Big(\|\delta_{\vv}\|^2_{L^2}+\sum^3_{i=1}\|\delta_{\nn_i}\|^2_{H^1}\Big),
\end{align*}
which implies from Gronwall's inequality that $\delta_{\vv}(t)=0$ and $\delta_{\nn_i}(t)=0(i=1,2,3)$ on $[0,T]$.

\subsection{Blow-up criterion}
Assume that $(\Fp,\vv)$ is the local smooth solution to the biaxial frame system (\ref{new-frame-equation-n1})--(\ref{imcompressible-v}) with the initial data $(\nabla\Fp^{(0)},\vv^{(0)})\in H^{2s}(\mathbb{R}^d)\times H^{2s}(\mathbb{R}^d)(d=2~\text{or}~3)$ and $\nabla\cdot\vv^{(0)}=0$.  In this subsection, we will establish the BKM type blow-up criterion: Let $T^*$ be the maximal existence time of the smooth solution; if $T^*<\infty$, then it is necessary to hold that
\begin{align*}
\int^{T^*}_0\big(\|\nabla\times\vv(t)\|_{L^{\infty}}+\|\nabla\Fp(t)\|^2_{L^{\infty}}\big)\ud t=+\infty,\quad \|\nabla\Fp(t)\|^2_{L^{\infty}}=\sum^3_{i=1}\|\nabla\nn_i(t)\|^2_{L^{\infty}}.
\end{align*}

To obtain the blow-up criterion, the main goal here is derive a much finer a prior energy estimate:
\begin{align}\label{energy-estimate-blowup}
\frac{\ud}{\ud t}\CE_s(\Fp,\vv)\leq C(1+\|\nabla\Fp\|^2_{L^{\infty}}+\|\nabla\vv\|_{L^{\infty}})\CE_s(\Fp,\vv),
\end{align}
where $\CE_s(\Fp,\vv)$ is given by
\begin{align*}
\CE_s(\Fp,\vv)=&\frac{1}{2}\|\Fp-\Fp^{(0)}\|^2_{L^2}+\CF_{Bi}[\Fp]+\frac{1}{2}\|\vv\|^2_{L^2}+\widetilde{\CE}_s(\Fp)+\frac{1}{2}\|\Delta^s\vv\|^2_{L^2}.
\end{align*}
Here, the higher-order derivative terms $\widetilde{\CE}_{s}(\Fp)$ are given by
\begin{align*}
\widetilde{\CE}_{s}(\Fp)=&\frac{1}{2}\sum^3_{i=1}\Big(\gamma_i\|\Delta^s\nabla\nn_i\|^2_{L^2}+k_i\|\Delta^s\rm{div}\nn_i\|^2_{L^2}+\sum^3_{j=1}k_{ji}\|\Delta^s(\nabla\times\nn_i)\cdot\nn_j\|^2_{L^2}\Big).
\end{align*}
If the energy inequality (\ref{energy-estimate-blowup}) holds, then from the logarithmic Sobolev inequality (see \cite{BKM} for details)
\begin{align*}
\|\nabla\vv\|_{L^{\infty}}\leq C\big(1+\|\nabla\vv\|_{L^2}+\|\nabla\times\vv\|_{L^{\infty}}\ln(3+\|\vv\|_{H^k})\big)
\end{align*}
for any $k\geq3$, we can infer that
\begin{align}\label{energy-blow-up-infty}
\frac{\ud}{\ud t}\CE_s(\Fp,\vv)\leq C\big(1+\|\nabla\vv\|_{L^{2}}+\|\nabla\times\vv\|_{L^{\infty}}+\|\nabla\Fp\|^2_{L^{\infty}}\big)\ln(3+\CE_s(\Fp,\vv))\CE_s(\Fp,\vv).
\end{align}
Then (\ref{energy-blow-up-infty}) and Gronwall's inequality implies
\begin{align*}
\CE_s(\Fp,\vv)\leq\big(3+\CE_s(\Fp^{(0)},\vv^{(0)})\big)\exp\Big(Ce^{\int_{0}^{t}(1+\|\nabla\vv\|_{L^{2}}+\|\nabla\times\vv\|_{L^{\infty}}+\|\nabla\Fp\|^2_{L^{\infty}})\ud\tau}\Big),
\end{align*}
for any $t\in[0,T^*)$. Furthermore, if $T^{*}<+\infty$ and
\begin{align*}
    \int^{T^*}_0\big(\|\nabla\times\vv\|_{L^{\infty}}+\|\nabla\Fp\|^2_{L^{\infty}}\big)\ud t<+\infty,
\end{align*}
then $\CE_s(\Fp,\vv)(t)\leq C$ for any $t\in[0,T^*)$.
Therefore, the smooth solution $(\Fp,\vv)$ can be extended after $t=T^*$, which contradicts the definition of $T^*$. Consequently, the BKM blow-up criterion holds.

Next, it suffices to prove the a prior energy inequality (\ref{energy-estimate-blowup}). We will omit some non-essential details, since the proof is fairly analogous to the arguments of the previous Subsection \ref{local-well-subsect}. Before we continue, let us give a useful Gagliardo--Sobolev inequality (for example, see \cite{Adams,WW}).
\begin{lemma}[Gagliardo--Sobolev inequality]\label{Gagliardo-Sobolev inequality}
Assume that $\alpha\in\mathbb{N}$ and $\alpha\geq 2s-1$, then for $1\leq j\leq[\frac{\alpha}{2}], [\frac{\alpha}{2}]+1\leq k\leq\alpha$ and $f\in H^{\alpha+1}(\mathbb{R}^d)$, it holds that
\begin{align*}
\|\nabla^jf\|_{L^{\infty}}\leq& C\|\nabla f\|^{\frac{j}{\alpha+1-d/2}}_{H^{\alpha}}\|f\|^{1-\frac{j}{\alpha+1-d/2}}_{L^{\infty}},\\
\|\nabla^kf\|_{L^2}\leq&C\|\nabla f\|^{\frac{k-d/2}{\alpha+1-d/2}}_{H^{\alpha}}\|f\|^{1-\frac{k-d/2}{\alpha+1-d/2}}_{L^{\infty}}.
\end{align*}
\end{lemma}

Now let us turn to the proof of (\ref{energy-blow-up-infty}). Note that the solution to the system (\ref{new-frame-equation-n1})--(\ref{imcompressible-v}) satisfies the orthonormal frame $\Fp=(\nn_1,\nn_2,\nn_3)\in SO(3)$, that is, $|\nn_i|=1$ and $\nn_i$ being mutually orthogonal for $i=1,2,3$.   Using the equations (\ref{new-frame-equation-n1})--(\ref{new-frame-equation-n3}) and Lemma \ref{h-decomposition}, and integrating by parts leads to
\begin{align*}
&\frac{1}{2}\frac{\ud}{\ud t}\|\Fp-\Fp^{(0)}\|^2_{L^2}=\sum^3_{i=1}\big\langle\partial_t\nn_i,\nn_i-\nn^{(0)}_i\big\rangle\\
&\quad\leq C\Big(\sum^3_{i=1}(\|\nabla\nn^{(0)}_i\|_{L^2}\|\vv\|_{L^2}+\|\nabla\vv\|_{L^2})\|\nn_i-\nn^{(0)}_i\|_{L^2}\\
&\qquad+\sum^3_{i=1}\big(\|\nabla\nn_i\|_{L^{\infty}}\|\nabla\nn_i\|_{L^2}\|\nn_i-\nn^{(0)}_i\|_{L^2}+\|\nabla\nn^{(0)}_i\|_{L^2}\|\nabla\nn_i\|^2_{L^2}\big)\Big)\\
&\quad\leq C_{\delta}(1+\|\nabla\Fp\|^2_{L^{\infty}})\CE_s(\Fp,\vv)+\delta\|\nabla\vv\|^2_{L^2},
\end{align*}
which together with the energy dissipation law (\ref{energy-law}) implies
\begin{align}\label{low-energy-Fp-v-L2}
\frac{\ud}{\ud t}\Big(\frac{1}{2}\|\Fp-\Fp^{(0)}\|^2_{L^2}+\CF_{Bi}[\Fp]+\frac{1}{2}\|\vv\|^2_{L^2}\Big)\leq C(1+\|\nabla\Fp\|^2_{L^{\infty}})\CE_s(\Fp,\vv).
\end{align}

We now consider the estimates of the higher order derivatives of the frame $\Fp\in SO(3)$. First, similar to (\ref{Deltas-CJ-hh-ve-comp}), using the definitions of $\hh_i(i=1,2,3)$ we also write
\begin{align*}
\Delta^s\hh_i\eqdefa&\CL_i(\Fp)+\CG_{i}(\Fp),
\end{align*}
where $\CL_i(\Fp)$ and $\CG_{i}(\Fp)$ are respectively higher-order and lower-order derivative terms,
\begin{align*}
\CL_i(\Fp)=&\gamma_i\Delta^{s+1}\nn_i+k_i\Delta^s\nabla{\rm div}\nn_i-\sum^3_{j=1}k_{ji}\Delta^s\nabla\times(\nabla\times\nn_i\cdot\nn^2_j),\\
\CG_i(\Fp)=&-\sum^3_{j=1}k_{ij}\Delta^s[(\nn_i\cdot\nabla\times\nn_j)(\nabla\times\nn_j)].
\end{align*}
According to the equations (\ref{new-frame-equation-n1})--(\ref{new-frame-equation-n3}), we define
\begin{align*}
\Delta^s\dot{\nn}_1=&\Big(\frac{1}{2}\Delta^s\BOm\cdot\aaa_1+\frac{\eta_3}{\chi_3}\Delta^s\A\cdot\sss_3-\frac{1}{\chi_3}\CH^{\Delta^s}_3\Big)\nn_2\nonumber\\
&-\Big(\frac{1}{2}\Delta^s\BOm\cdot\aaa_2+\frac{\eta_2}{\chi_2}\Delta^s\A\cdot\sss_4-\frac{1}{\chi_2}\CH^{\Delta^s}_2\Big)\nn_3+\CR_1(\Fp),\\
\Delta^s\dot{\nn}_2=&-\Big(\frac{1}{2}\Delta^s\BOm\cdot\aaa_1+\frac{\eta_3}{\chi_3}\Delta^s\A\cdot\sss_3-\frac{1}{\chi_3}\CH^{\Delta^s}_3\Big)\nn_1\nonumber\\
&+\Big(\frac{1}{2}\Delta^s\BOm\cdot\aaa_3+\frac{\eta_1}{\chi_1}\Delta^s\A\cdot\sss_5-\frac{1}{\chi_1}\CH^{\Delta^s}_1\Big)\nn_3+\CR_2(\Fp),\\
\Delta^s\dot{\nn}_3=&\Big(\frac{1}{2}\Delta^s\BOm\cdot\aaa_2+\frac{\eta_2}{\chi_2}\Delta^s\A\cdot\sss_4-\frac{1}{\chi_2}\CH^{\Delta^s}_2\Big)\nn_1\nonumber\\
&-\Big(\frac{1}{2}\Delta^s\BOm\cdot\aaa_3+\frac{\eta_1}{\chi_1}\Delta^s\A\cdot\sss_5-\frac{1}{\chi_1}\CH^{\Delta^s}_1\Big)\nn_2+\CR_3(\Fp),
\end{align*}
where $\CR_i(\Fp)(i=1,2,3)$ are derivative terms that can be easily handled by commutator estimates, for example,
\begin{align*}
\CR_1(\Fp)=&\frac{1}{2}[\Delta^s,\nn_2\otimes\aaa_1\cdot]\BOm+\frac{\eta_3}{\chi_3}[\Delta^s,\nn_2\otimes\sss_3\cdot]\A-\frac{1}{\chi_3}\big([\Delta^s,\nn_2\otimes\nn_1\cdot]\hh_2-[\Delta^s,\nn^2_2\cdot]\hh_1\big)\\
&-\frac{1}{2}[\Delta^s,\nn_3\otimes\aaa_2\cdot]\BOm-\frac{\eta_2}{\chi_2}[\Delta^s,\nn_3\otimes\sss_4\cdot]\A+\frac{1}{\chi_2}\big([\Delta^s,\nn^2_3\cdot]\hh_1-[\Delta^s,\nn_3\otimes\nn_1\cdot]\hh_3\big),
\end{align*}
and $\CH^{\Delta^s}_i(i=1,2,3)$ are denoted by
\begin{align*}
\CH^{\Delta^s}_{1}=&\nn_2\cdot\Delta^s\hh_3-\nn_3\cdot\Delta^s\hh_2,\\
\CH^{\Delta^s}_{2}=&\nn_3\cdot\Delta^s\hh_1-\nn_1\cdot\Delta^s\hh_3,\\
\CH^{\Delta^s}_{3}=&\nn_1\cdot\Delta^s\hh_2-\nn_2\cdot\Delta^s\hh_1.
\end{align*}

By a direct calculation, we obtain
\begin{align}
\frac{1}{2}\frac{\ud}{\ud t}\|\Delta^s\nabla\nn_i\|^2_{L^2}=&-\langle\Delta^s\dot{\nn}_i,\Delta^{s+1}\nn_i\rangle-\langle\nabla\Delta^s(\vv\cdot\nabla\nn_i),\Delta^s\nabla\nn_i\rangle,\label{Deltas-nablani-L2}\\
\frac{1}{2}\frac{\ud}{\ud t}\|\Delta^s{\rm div}\nn_i\|^2_{L^2}=&-\langle\Delta^s\dot{\nn}_i,\Delta^s\nabla{\rm div}\nn_i\rangle-\langle\Delta^s{\rm div}(\vv\cdot\nabla\nn_i),\Delta^s{\rm div}\nn_i\rangle.\label{Deltas-divni-L2}
\end{align}
By the relation
$
\nabla\times(fu)=f\nabla\times u+\nabla f\times u
$
for any scalar field $f$ and any vector field $u$, together with $\nabla\cdot\vv=0$, we can derive that
\begin{align}\label{energy-dot-nn-4term}
&\frac{1}{2}\frac{\ud}{\ud t}\|\nn_j\cdot\Delta^s(\nabla\times\nn_i)\|^2_{L^2}\nonumber\\
&\quad=\big\langle\dot{\nn}_j\cdot\Delta^s(\nabla\times\nn_i)+\nn_j\cdot\Delta^s(\nabla\times\dot{\nn}_i),\nn_j\cdot\Delta^s(\nabla\times\nn_i)\big\rangle\nonumber\\
&\qquad-\big\langle(\vv\cdot\nabla\nn_j)\cdot\Delta^s(\nabla\times\nn_i)+\nn_j\cdot\Delta^s\big(\nabla\times(\vv\cdot\nabla\nn_i)\big),\nn_j\cdot\Delta^s(\nabla\times\nn_i)\big\rangle\nonumber\\
&\quad=\big\langle\Delta^s\dot{\nn}_i,\Delta^s\nabla\times(\nabla\times\nn_i\cdot\nn^2_j)\big\rangle+\big\langle\dot{\nn}_j\cdot\Delta^s(\nabla\times\nn_i),\nn_j\cdot\Delta^s(\nabla\times\nn_i)\big\rangle\nonumber\\
&\qquad+\big\langle\Delta^s\dot{\nn}_i,\nabla\times\big(\Delta^s(\nabla\times\nn_i)\cdot\nn^2_j\big)-\Delta^s\nabla\times(\nabla\times\nn_i\cdot\nn^2_j)\big\rangle\nonumber\\
&\qquad-\big\langle[\Delta^s\nabla\times,\vv\cdot]\nabla\nn_i,\Delta^s(\nabla\times\nn_i)\cdot\nn^2_j\big\rangle.
\end{align}
Then, the equations (\ref{new-frame-equation-n1})--(\ref{new-frame-equation-n3}) and (\ref{Deltas-nablani-L2})--(\ref{energy-dot-nn-4term}) yield
\begin{align}\label{energy-es-fp-tilde}
\frac{\ud}{\ud t}\widetilde{\CE}_s(\Fp)
\eqdefa&\CK_1+\CK_2+\CK_3+\CK_4+\CK_5,
\end{align}
where the terms $\CK_i(i=1,\cdots,5)$ are expressed as, respectively,
\begin{align*}
\CK_1=&-\sum^3_{i=1}\Big(\gamma_i\langle\nabla\Delta^s(\vv\cdot\nabla\nn_i),\Delta^s\nabla\nn_i\rangle+k_i\langle\Delta^s{\rm div}(\vv\cdot\nabla\nn_i),\Delta^s{\rm div}\nn_i\rangle\\
&+\sum^3_{j=1}k_{ji}\langle[\Delta^s\nabla\times,\vv\cdot]\nabla\nn_i,\Delta^s(\nabla\times\nn_i)\cdot\nn^2_j\rangle\Big),\\
\CK_2=&\int_{\mathbb{R}^d}\Big(\frac{1}{2}\Delta^s\BOm\cdot\aaa_1+\frac{\eta_3}{\chi_3}\Delta^s\A\cdot\sss_3-\frac{1}{\chi_3}\CH^{\Delta^s}_3\Big)\CH^{\Delta^s}_3\ud\xx\\
&+\int_{\mathbb{R}^d}\Big(\frac{1}{2}\Delta^s\BOm\cdot\aaa_2+\frac{\eta_2}{\chi_2}\Delta^s\A\cdot\sss_4-\frac{1}{\chi_2}\CH^{\Delta^s}_2\Big)\CH^{\Delta^s}_2\ud\xx\\
&+\int_{\mathbb{R}^d}\Big(\frac{1}{2}\Delta^s\BOm\cdot\aaa_3+\frac{\eta_1}{\chi_1}\Delta^s\A\cdot\sss_5-\frac{1}{\chi_1}\CH^{\Delta^s}_1\Big)\CH^{\Delta^s}_1\ud\xx,\\
\CK_3=&-\Big\langle\frac{1}{2}\Delta^s\BOm\cdot\aaa_1+\frac{\eta_3}{\chi_3}\Delta^s\A\cdot\sss_3-\frac{1}{\chi_3}\CH^{\Delta^s}_3,\nn_1\cdot\CG_2(\Fp)-\nn_2\cdot\CG_1(\Fp)\Big\rangle\\
&-\Big\langle\frac{1}{2}\Delta^s\BOm\cdot\aaa_2+\frac{\eta_2}{\chi_2}\Delta^s\A\cdot\sss_4-\frac{1}{\chi_2}\CH^{\Delta^s}_2,\nn_3\cdot\CG_1(\Fp)-\nn_1\cdot\CG_3(\Fp)\Big\rangle\\
&-\Big\langle\frac{1}{2}\Delta^s\BOm\cdot\aaa_3+\frac{\eta_1}{\chi_1}\Delta^s\A\cdot\sss_5-\frac{1}{\chi_1}\CH^{\Delta^s}_1,\nn_2\cdot\CG_3(\Fp)-\nn_3\cdot\CG_2(\Fp)\Big\rangle,\\
\CK_4=&-\sum^3_{i=1}\big\langle\CR_i(\Fp),\CL_i(\Fp)\big\rangle,\\
\CK_5=&\sum^3_{i,j=1}k_{ji}\Big(\big\langle\dot{\nn}_j\cdot\Delta^s(\nabla\times\nn_i),\nn_j\cdot\Delta^s(\nabla\times\nn_i)\big\rangle\\
&+\big\langle\Delta^s\dot{\nn}_i,\nabla\times\big(\Delta^s(\nabla\times\nn_i)\cdot\nn^2_j\big)-\Delta^s\nabla\times(\nabla\times\nn_i\cdot\nn^2_j)\big\rangle\Big).
\end{align*}

Applying Lemma \ref{product-estimate-lemma}, we easily obtain
\begin{align*}
\CK_1\leq &C\sum^3_{i=1}(\|\nabla\vv\|_{L^{\infty}}\|\nabla\nn_i\|_{H^{2s}}+\|\nabla\nn_i\|_{L^{\infty}}\|\nabla\vv\|_{H^{2s}})\|\nabla\nn_i\|_{H^{2s}}\nonumber\\
\leq&C_{\delta}(\|\nabla\vv\|_{L^{\infty}}+\|\nabla\Fp\|^2_{L^{\infty}})\|\nabla\Fp\|^2_{H^{2s}}+\delta\|\nabla\vv\|^2_{H^{2s}},
\end{align*}
where $\delta$ denotes a small positive constant to be determined later.

Using Lemma \ref{Gagliardo-Sobolev inequality}, for $\alpha\geq2s-1$ with $s\geq2$ and $i=1,2,3$, it follows that
\begin{align}\label{Gagliardo-Sobolev-ineq12}
 \left\{
\begin{array}{l}
\|\nabla^{\alpha+1}\nn_i\|_{L^2}\|\nabla\nn_i\|_{L^{\infty}}+\|\nabla^{\alpha}\nn_i\|_{L^2}\|\nabla^2\nn_i\|_{L^{\infty}}\vspace{1ex}\\
\qquad+\|\nabla^{\alpha}\nn_i\|_{L^2}\|\nabla\nn_i\|^2_{L^{\infty}}
\leq C\|\nabla\nn_i\|_{H^{\alpha+1}},\vspace{1ex}\\
(\|\nabla^2\nn_i\|_{L^{\infty}}+\|\nabla\nn_i\|^2_{L^{\infty}})\|\nabla^{\alpha}\nn_i\|_{L^2}\leq C\|\nabla\nn_i\|_{L^{\infty}}\|\nabla^{\alpha+1}\nn_i\|_{L^2}.
    \end{array}
    \right.
\end{align}
It can seen from Lemma \ref{product-estimate-lemma} and (\ref{Gagliardo-Sobolev-ineq12}) that $\CK_3$ can be estimated as
\begin{align*}
\CK_3\leq&C\Big(\|\Delta^s\nabla\vv\|_{L^2}+\sum^3_{i=1}\|\CH^{\Delta^s}_i\|_{L^2}\Big)\Big(\sum^3_{i=1}\|\CG_i(\Fp)\|_{L^2}\Big)\\
\leq&C\Big(\|\nabla\vv\|_{H^{2s}}+\sum^3_{i=1}\|\CL_i(\Fp)+\CG_i(\Fp)\|_{L^2}\Big)\Big(\sum^3_{i=1}\|\CG_i(\Fp)\|_{L^2}\Big)\\
\leq&\delta(\|\Delta^{s+1}\Fp\|^2_{L^2}+\|\nabla\vv\|^2_{H^{2s}})+C_{\delta}(1+\|\nabla\Fp\|_{L^{\infty}}^2)\|\nabla\Fp\|^2_{H^{2s}},
\end{align*}
where we have used the estimates of $\CL_i(\Fp)$ and $\CG_i(\Fp)(i=1,2,3)$, for example,
\begin{align*}
&\|\Delta^s\nabla\times(\nabla\times\nn_i\cdot\nn_j^2)\|_{L^2}\\
&\quad\leq \|[\Delta^s\nabla\times,\nn_j^2\cdot]\nabla\times\nn_i\|_{L^2}+\|\nn_j^2\cdot\Delta^s\nabla\times(\nabla\times\nn_i)\|_{L^2}\\
&\quad\leq C\big(\|\Delta^s\nabla(\nn_j^2)\|_{L^2}\|\nabla\nn_i\|_{L^{\infty}}+\|\Delta^s\nabla\nn_i\|_{L^2}\|\nabla\nn_j\|_{L^{\infty}}+\|\Delta^{s+1}\nn_i\|_{L^2}\big)\\
%&\leq\big( (\|[\Delta^s\nabla,\nn_2\otimes]\nn_2]\|_{L^2}+\|\nn_2\otimes\Delta^s\nabla\nn_2\|_{L^2})\|\nabla\nn_1\|_{L^{\infty}}+\|\Delta^s\nabla\nn_1\|_{L^2}\|\nabla\nn_2\|_{L^{\infty}}+\|\Delta^{s+1}\nn_1\|_{L^2})\big)\\
&\quad\leq C\big(\|\nabla\Fp\|_{L^{\infty}}\|\nabla\Fp\|_{H^{2s}}+\|\Delta^{s+1}\Fp\|_{L^2}\big),\\
&\|\Delta^s[(\nabla\times\nn_i\cdot\nn_i)(\nabla\times\nn_i)]\|_{L^2}\leq C\|\nabla\Fp\|_{L^{\infty}}\|\nabla\Fp\|_{H^{2s}}.
\end{align*}

For the term $\CK_4$, we only show how to estimate the first term $-\langle\CR_1(\Fp),\CL_1(\Fp)\rangle$, since the estimates of the remaining terms are similar. By virtue of Lemma \ref{product-estimate-lemma} and (\ref{Gagliardo-Sobolev-ineq12}), the definitions of $\hh_i(i=1,2,3)$ and integrating by parts, there holds
\begin{align*}
\big|\langle\CR_1(\Fp),\CL_1(\Fp)\rangle\big|\leq& C\Big(\|\nabla[\Delta^s,\nn_2\otimes\aaa_1\cdot]\BOm\|_{L^2}+\|\nabla[\Delta^s,\nn_3\otimes\aaa_2\cdot]\BOm\|_{L^2}\\
&+\|\nabla[\Delta^s,\nn_2\otimes\sss_3\cdot]\A\|_{L^2}+\|\nabla[\Delta^s,\nn_3\otimes\sss_4\cdot]\A\|_{L^2}\\
&+\|\nabla[\Delta^s,\nn_2\otimes\nn_1\cdot]\hh_2\|_{L^2}+\|\nabla[\Delta^s,\nn_3\otimes\nn_1\cdot]\hh_3\|_{L^2}\\
&+\|\nabla[\Delta^s,\nn^2_2\cdot]\hh_1\|_{L^2}+\|\nabla[\Delta^s,\nn^2_3\cdot]\hh_1\|_{L^2}\Big)\|\nabla\Fp\|_{H^{2s}}\\
\leq&C\Big(\|\nabla\vv\|_{L^{\infty}}\|\nabla\Fp\|_{H^{2s}}+(\|\nabla^2\Fp\|_{L^{\infty}}+\|\nabla\Fp\|^2_{L^{\infty}})\|\vv\|_{H^{2s}}\\
&+\|\nabla\Fp\|_{L^{\infty}}\|\nabla\vv\|_{H^{2s}}+(\|\nabla^2\Fp\|_{L^{\infty}}+\|\nabla\Fp\|^2_{L^{\infty}})\|\nabla\Fp\|_{H^{2s}}\\
&+\|\nabla\Fp\|_{L^{\infty}}\|\Delta^{s+1}\Fp\|_{L^2}\Big)\|\nabla\Fp\|_{H^{2s}}\\
\leq& C_{\delta}(1+\|\nabla\vv\|_{L^{\infty}}+\|\nabla\Fp\|^2_{L^{\infty}})(\|\vv\|_{H^{2s}}+\|\nabla\Fp\|^2_{H^{2s}})\\
&+\delta(\|\nabla\vv\|_{H^{2s}}+\|\Delta^{s+1}\Fp\|^2_{L^2}).
\end{align*}

For the term $\CK_5$, employing the equations (\ref{new-frame-equation-n1})--(\ref{new-frame-equation-n3}), Lemma \ref{product-estimate-lemma} and the inequality (\ref{Gagliardo-Sobolev-ineq12}), we derive that
\begin{align*}
\CK_5\leq&C\sum^3_{i=1}\|\dot{\nn}_i\|_{L^{\infty}}\|\nabla\Fp\|_{H^{2s}}
+C\Big(\sum^3_{i=1}\|\Delta^s\dot{\nn}_i\|_{L^2}\Big)\\
&\times\Big(\sum^3_{\alpha,\beta=1}\big\|\nabla\times\big(\Delta^s(\nabla\times\nn_{\alpha})\cdot\nn^2_{\beta}\big)-\Delta^s\nabla\times(\nabla\times\nn_{\alpha}\cdot\nn^2_{\beta})\big\|_{L^2}\Big)\\
\leq&C\Big(\|\Delta^s\nabla\vv\|_{L^2}+\|\Delta^s\Fp\|_{L^2}\|\nabla\vv\|_{L^{\infty}}\\
&+(\|\nabla\Fp\|^2_{L^{\infty}}+\|\nabla^2\Fp\|_{L^{\infty}})\|\Delta^s\Fp\|_{L^2}+\sum^3_{i=1}\frac{1}{\chi_i}\|\CH^{\Delta^s}_i\|_{L^2}\Big)\\
&\quad\times\Big((1+\|\nabla\Fp\|_{L^{\infty}})\|\Delta^s\nabla\Fp\|_{L^2}+\|\Delta^s\Fp\|_{L^2}\|\nabla^2\Fp\|_{L^{\infty}}\Big)\\
&+C(\|\nabla\vv\|_{L^{\infty}}+\|\nabla\Fp\|^2_{L^{\infty}}+\|\nabla^2\Fp\|_{L^{\infty}})\|\nabla\Fp\|_{H^{2s}}\\
\leq&\delta(\|\nabla\vv\|^2_{H^{2s}}+\|\Delta^{s+1}\Fp\|_{L^2})+C_{\delta}(1+\|\nabla\vv\|_{L^{\infty}}+\|\nabla\Fp\|^2_{L^{\infty}})\|\nabla\Fp\|^2_{H^{2s}},
\end{align*}
where for $\alpha,\beta=1,2,3$, we have used the following fact:
\begin{align*}
&\big\|\nabla\times\big(\Delta^s(\nabla\times\nn_{\alpha})\cdot\nn^2_{\beta}\big)-\Delta^s\nabla\times(\nabla\times\nn_{\alpha}\cdot\nn^2_{\beta})\big\|_{L^2}\\
&\quad\leq\sum^{2s-1}_{k=1}\|\nabla\times[\partial^{2s-k-1}(\nabla\times\nn_{\alpha})\cdot\partial^k(\nn^2_{\beta})]\|_{L^2}+\|\nabla\times(\nabla\times\nn_{\alpha}\cdot\partial^{2s}(\nn^2_{\beta}))\|_{L^2}\\
&\quad\leq C\Big((1+\|\nabla\Fp\|_{L^{\infty}})\|\Delta^s\nabla\Fp\|_{L^2}+\|\Delta^s\Fp\|_{L^2}\|\nabla^2\Fp\|_{L^{\infty}}\Big).
\end{align*}

In order to control the higher order derivative term $\|\Delta^{s+1}\Fp\|_{L^2}$, we need to provide a higher order dissipated estimate for $\sum\limits^3_{i=1}\frac{1}{\chi_i}\|\CH^{\Delta^s}_i\|^2_{L^2}$, that is,
\begin{align*}
\sum^3_{i=1}\frac{1}{\chi_i}\|\CH^{\Delta^s}_i\|^2_{L^2}\geq\frac{2\gamma^2}{\chi}\|\Delta^{s+1}\Fp\|^2_{L^2}+L.O.T.,
\end{align*}
where $\gamma=\min\{\gamma_1,\gamma_2,\gamma_3\}$, $\chi=\max\{\chi_1,\chi_2,\chi_3\}$, and L.O.T represents the lower order derivative terms.

Note that for any matrix $A\in\mathbb{R}^{3\times3}$, we could write
\begin{align}\label{A-DFp-orthogonal-decomp-Blowup}
A\cdot\Delta^{s+1}\Fp=\sum^3_{k=1}\frac{1}{|V_k|^2}(A\cdot V_k)(\Delta^{s+1}\Fp\cdot V_k)+\sum^6_{k=1}\frac{1}{|W_k|^2}(A\cdot W_k)(\Delta^{s+1}\Fp\cdot W_k)
\end{align}
by using \eqref{AB-orthogonal-decomposition}, where we recall the definition of $V_k(k=1,2,3)$ and $W_k(k=1,\cdots,6)$ before \eqref{ML-123}.

Taking $A=\Delta^{s+1}\Fp$ in (\ref{A-DFp-orthogonal-decomp-Blowup}), and then integrating over $\mathbb{R}^d$, we deduce that
\begin{align*}
\|\Delta^{s+1}\Fp\|^2_{L^2}\leq\frac{1}{2}\sum^3_{k=1}\|\Delta^{s+1}\Fp\cdot V_k\|^2_{L^2}+C_{\delta}\|\nabla\Fp\|^2_{L^{\infty}}\|\nabla\Fp\|^2_{H^{2s}}+\delta\|\Delta^{s+1}\Fp\|^2_{L^2}.
\end{align*}
Furthermore, taking
\begin{align*}
A=\Delta^s\nabla\cdot\frac{\partial f_{Bi}}{\partial(\nabla\Fp)}-\gamma
\Delta^{s+1}\Fp
\end{align*}
into (\ref{A-DFp-orthogonal-decomp-Blowup}) and then integrating over $\mathbb{R}^d$, we obtain
\begin{align*}
&\int_{\mathbb{R}^d}\Big(\Delta^s\nabla\cdot\frac{\partial f_{Bi}}{\partial(\nabla\Fp)}-\gamma
\Delta^{s+1}\Fp\Big)\cdot\Delta^{s+1}\Fp\ud\xx\\
&\quad\leq\frac{1}{2}\int_{\mathbb{R}^d}\sum^3_{k=1}\Big[\Big(\Delta^s\nabla\cdot\frac{\partial f_{Bi}}{\partial(\nabla\Fp)}-\gamma
\Delta^{s+1}\Fp\Big)\cdot V_k\Big](\Delta^{s+1}\Fp\cdot V_k)\ud\xx\\
&\qquad+C_{\delta}\|\nabla\Fp\|^2_{L^{\infty}}\|\nabla\Fp\|^2_{H^{2s}}+\delta\|\Delta^{s+1}\Fp\|^2_{L^2}.
\end{align*}
Using Lemma \ref{h-decomposition} and integrating by parts, one can give
\begin{align*}
&\int_{\mathbb{R}^d}\Big(\Delta^s\nabla\cdot\frac{\partial f_{Bi}}{\partial(\nabla\Fp)}-\gamma
\Delta^{s+1}\Fp\Big)\cdot\Delta^{s+1}\Fp\ud\xx\\
&\quad\geq\sum^3_{i=1}\int_{\mathbb{R}^d}\Big(k_i|\Delta^s\nabla{\rm div}\nn_i|^2+\sum^3_{j=1}k_{ji}|\Delta^s\nabla\times(\nabla\times\nn_i\cdot\nn_j)|^2\Big)\ud\xx\\
&\qquad-C_{\delta}(1+\|\nabla\Fp\|^2_{L^{\infty}})\|\nabla\Fp\|^2_{H^{2s}}-\delta\|\Delta^{s+1}\Fp\|^2_{L^2}.
\end{align*}
Thus, by virtue of the definitions of $\CH^{\Delta^s}_k$ and $\ML_k(k=1,2,3)$, we have
\begin{align*}
&\sum^3_{k=1}\frac{1}{\chi_k}\|\CH^{\Delta^s}_k\|^2_{L^2}\geq\sum^3_{k=1}\frac{1}{\chi_k}\|\Delta^s(\ML_k\CF_{Bi})\|^2_{L^2}-C\|\nabla\Fp\|^2_{L^{\infty}}\|\nabla\Fp\|^2_{H^{2s}}\nonumber\\
&\quad=\sum^3_{k=1}\frac{1}{\chi_k}\int_{\mathbb{R}^d}\Big[\Delta^s\Big(V_k\cdot\frac{\delta\CF_{Bi}}{\delta\Fp}\Big)\Big]^2\ud\xx-C\|\nabla\Fp\|^2_{L^{\infty}}\|\nabla\Fp\|^2_{H^{2s}}\nonumber\\
&\quad\geq\frac{1}{\chi}\int_{\mathbb{R}^d}\sum^3_{k=1}\Big[V_k\cdot\Big(\Delta^s\nabla\cdot\frac{\partial f_{Bi}}{\partial(\nabla\Fp)}\Big)\Big]^2\ud\xx-C_{\delta}\|\nabla\Fp\|^2_{L^{\infty}}\|\nabla\Fp\|^2_{H^{2s}}-\delta\|\Delta^{s+1}\Fp\|^2_{L^2}\nonumber\\
&\quad\geq\frac{2\gamma}{\chi}\int_{\mathbb{R}^d}\sum^3_{k=1}\Big[\Big(\Delta^s\nabla\cdot\frac{\partial f_{Bi}}{\partial(\nabla\Fp)}-\gamma\Delta^{s+1}\Fp\Big)\cdot V_{k}\Big](\Delta^{s+1}\Fp\cdot V_{k})\ud\xx\nonumber\\
&\qquad+\frac{\gamma^2}{\chi}\int_{\mathbb{R}^d}\sum^3_{k=1}(\Delta^{s+1}\Fp\cdot V_{k})^2\ud\xx-C_{\delta}\|\nabla\Fp\|^2_{L^{\infty}}\|\nabla\Fp\|^2_{H^{2s}}-\delta\|\Delta^{s+1}\Fp\|^2_{L^2}\nonumber\\
&\quad\geq \sum^3_{i=1}\int_{\mathbb{R}^d}\Big(k_i|\Delta^s\nabla{\rm div}\nn_i|^2+\sum^3_{j=1}k_{ji}|\Delta^s\nabla\times(\nabla\times\nn_i\cdot\nn_j)|^2\Big)\ud\xx\nonumber\\
&\qquad+\frac{2\gamma}{\chi}\|\Delta^{s+1}\Fp\|^2_{L^2}-C_{\delta}(1+\|\nabla\Fp\|^2_{L^{\infty}})\|\nabla\Fp\|^2_{H^{2s}}-\delta\|\Delta^{s+1}\Fp\|^2_{L^2},
\end{align*}
which implies
\begin{align}\label{Deltas+1-Fp-L2-estimate}
-\Big(\frac{2\gamma}{\chi}-\delta\Big)\|\Delta^{s+1}\Fp\|^2_{L^2}\leq-\sum^3_{k=1}\frac{1}{\chi_k}\|\CH^{\Delta^s}_k\|^2_{L^2}+C_{\delta}(1+\|\nabla\Fp\|^2_{L^{\infty}})\|\nabla\Fp\|^2_{H^{2s}}.
\end{align}
Consequently, plugging the estimates of $\CK_i(i=1,\cdots,5)$ into (\ref{energy-es-fp-tilde}), we obtain
\begin{align}\label{Wtilde-CE-energy-higher-L2}
&\frac{\ud}{\ud t}\widetilde{\CE}_s(\Fp)+\Big(\frac{2\gamma}{\chi}-\delta\Big)\|\Delta^{s+1}\Fp\|^2_{L^2}+\delta\|\Delta^s\nabla\vv\|^2_{L^2}\nonumber\\
&\quad\leq\Big\langle\frac{\eta_3}{\chi_3}\CH^{\Delta^s}_3\sss_3+\frac{\eta_2}{\chi_2}\CH^{\Delta^s}_2\sss_4+\frac{\eta_1}{\chi_1}\CH^{\Delta^s}_1\sss_5,\Delta^s\A\Big\rangle\nonumber\\
&\qquad+\frac{1}{2}\big\langle\CH^{\Delta^s}_3\aaa_1+\CH^{\Delta^s}_2\aaa_2+\CH^{\Delta^s}_1\aaa_3,\Delta^s\BOm\big\rangle\nonumber\\
&\qquad+C_{\delta}(1+\|\nabla\vv\|_{L^{\infty}}+\|\nabla\Fp\|^2_{L^{\infty}})(\|\vv\|_{H^{2s}}+\|\nabla\Fp\|^2_{H^{2s}}).
\end{align}

On the other hand, we need to deal with estimates of higher order derivatives of $\vv$. Similar to (\ref{v-highderivative-L2}), we have
\begin{align}\label{v-highderivative-L2-blowup}
&\frac{1}{2}\frac{\ud}{\ud t}\|\Delta^s\vv\|^2_{L^2}+\eta\|\nabla\Delta^s\vv\|^2_{L^2}\nonumber\\
&\quad=-\big\langle\Delta^s(\vv\cdot\nabla\vv),\Delta^s\nabla\vv\big\rangle-\big\langle\Delta^s\sigma,\Delta^s\nabla\vv\big\rangle+\langle\Delta^s\mathfrak{F},\Delta^s\vv\rangle,\nonumber\\
&\quad\eqdefa I+II+III.
\end{align}
It can be easily derived from Lemma \ref{product-estimate-lemma} that
\begin{align*}
I\leq&C\|\nabla\vv\|_{L^{\infty}}\|\vv\|^2_{H^{2s}},\\
III\leq& C_{\delta}(1+\|\nabla\Fp\|^2_{L^{\infty}})\|\nabla\Fp\|^2_{H^{2s}}+\delta\|\nabla\vv\|^2_{H^{2s}},
\end{align*}
while the term $II$ can be estimated as
\begin{align*}
%II\leq&-\int_{\mathbb{R}^d}\big((\sss_1,\cdots,\sss_5)P(\Delta^s\A\cdot\sss_1,\cdots,\Delta^s\A\cdot\sss_5)^T\big)\cdot\Delta^s\A\ud\xx\\
%&+C_{\delta}(\|\nabla\vv\|_{L^{\infty}}+\|\nabla\Fp\|^2_{L^{\infty}})(\|\vv\|^2_{H^{2s}}+\|\nabla\Fp\|^2_{H^{2s}})+\delta\|\nabla\vv\|^2_{H^{2s}},\\
II\leq&-\bigg(\sum^2_{j=1}\beta_j\|\Delta^s\A\cdot\sss_j\|^2_{L^2}+2\beta_0\int_{\mathbb{R}^d}(\Delta^s\A\cdot\sss_1)(\Delta^s\A\cdot\sss_2)\ud\xx\bigg)\\
&-\Big(\beta_3-\frac{\eta^2_3}{\chi_3}\Big)\|\Delta^s\A\cdot\sss_3\|^2_{L^2}
-\Big(\beta_4-\frac{\eta^2_2}{\chi_2}\Big)\|\Delta^s\A\cdot\sss_4\|^2_{L^2}
-\Big(\beta_5-\frac{\eta^2_1}{\chi_1}\Big)\|\Delta^s\A\cdot\sss_5\|^2_{L^2}\\
&-\Big\langle\frac{\eta_3}{\chi_3}\CH^{\Delta^s}_3\sss_3+\frac{\eta_2}{\chi_2}\CH^{\Delta^s}_2\sss_4+\frac{\eta_1}{\chi_1}\CH^{\Delta^s}_1\sss_5,\Delta^s\A\Big\rangle\nonumber\\
&-\frac{1}{2}\big\langle\CH^{\Delta^s}_3\aaa_1+\CH^{\Delta^s}_2\aaa_2+\CH^{\Delta^s}_1\aaa_3,\Delta^s\BOm\big\rangle+\delta\|\nabla\vv\|^2_{H^{2s}}\\
&+C_{\delta}(1+\|\nabla\vv\|_{L^{\infty}}+\|\nabla\Fp\|^2_{L^{\infty}})(\|\vv\|_{H^{2s}}+\|\nabla\Fp\|^2_{H^{2s}}).
\end{align*}

Hence, substituting the estimates of $I-III$ into (\ref{v-highderivative-L2-blowup}), together with (\ref{low-energy-Fp-v-L2}) and (\ref{Wtilde-CE-energy-higher-L2}), and choosing $\delta>0$ sufficiently small, we arrive at
\begin{align*}
\frac{\ud}{\ud t}\CE_s(\Fp,\vv)\leq C(1+\|\nabla\Fp\|^2_{L^{\infty}}+\|\nabla\vv\|_{L^{\infty}})\CE_s(\Fp,\vv).
\end{align*}
The proof of (\ref{energy-estimate-blowup}) is completed.

\section{Global existence of weak solutions }\label{global-weak-solution-section}

The main purpose of this section is to prove global existence of weak solutions to the system (\ref{new-frame-equation-n1})--(\ref{imcompressible-v}) in dimension two.

For two constants $\tau$ and $T$ with $0\leq\tau<T$, we define two spaces $V(\tau,T)$ and $H(\tau,T)$ as
\begin{align*}
V(\tau,T)\eqdefa&\bigg\{\Fp=(\nn_1,\nn_2,\nn_3):\mathbb{R}^2\times[\tau,T]\rightarrow SO(3)\Big|\Fp(t)\in H^1_{\Fp^*}\big(\mathbb{R}^2,SO(3)\big)~\text{for a.e.}~ t\in[\tau,T]\\
&\quad\text{and satisfies}~\mathop{\mathrm{ess\sup}}_{\tau\leq t\leq T}\int_{\mathbb{R}^2}|\nabla\Fp(\cdot,t)|^2\ud\xx+\int^T_{\tau}\int_{\mathbb{R}^2}(|\nabla^2\Fp|^2+|\partial_t\Fp|^2)\ud\xx\ud t<\infty,\\
&\quad\text{where}~|\nabla\Fp(\cdot,t)|^2=\sum^3_{i=1}|\nabla\nn_i(\cdot,t)|^2,~|\nabla^2\Fp|^2=\sum^3_{i=1}|\nabla^2\nn_i|^2,~|\partial_t\Fp|^2=\sum^3_{i=1}|\partial_t\nn_i|^2\bigg\},\\
H(\tau,T)\eqdefa&\bigg\{\vv:\mathbb{R}^2\times[\tau,T]\rightarrow\mathbb{R}^2\Big| \vv~\text{is measurable and satisfies}\\
&\quad \mathop{\mathrm{ess\sup}}_{\tau\leq t\leq T}\int_{\mathbb{R}^2}|\vv(\cdot,t)|^2\ud\xx
+\int^T_{\tau}\int_{\mathbb{R}^2}|\nabla\vv|^2\ud\xx\ud t<\infty\bigg\}.
\end{align*}

\subsection{A priori regularity estimates}
We first present useful estimates for the $L^4$-norm of the spatial gradient of $u\in V(0,T)$ and the $L^4$-norm of $f\in H(0,T)$, which can be found in \cite{Struwe}.
\begin{lemma}\label{L4-norm-lemma}
There exists constants $C, R_0>0$ such that for any $u\in V(0,T)$ and any $R\in(0,R_0]$, it holds
\begin{align*}
    \int_{\mathbb{R}^\times[0,T]}|\nabla u|^4\ud\xx\ud t\leq& C \mathop{\mathrm{ess\sup}}_{\xx\in\mathbb{R}^2,\tau\leq t\leq T}\int_{B_R(\xx)}|\nabla u(\cdot,t)|^2\ud\xx\\
    &\cdot\bigg(\int_{\mathbb{R}^2\times[0,T]}|\nabla^2 u|^2\ud\xx\ud t+R^{-2}\int_{\mathbb{R}^2\times[0,T]}|\nabla u|^2\ud\xx\ud t\bigg).
\end{align*}
Moreover, there exists a constant $C_1>0$ such that for any $f\in H(0,T)$ and any $R>0$, it follows that
\begin{align*}
    \int_{\mathbb{R}^\times[0,T]}|f|^4\ud\xx\ud t\leq& C_1 \mathop{\mathrm{ess\sup}}_{\xx\in\mathbb{R}^2,\tau\leq t\leq T}\int_{B_R(\xx)}|f(\cdot,t)|^2\ud\xx\\
    &\cdot\bigg(\int_{\mathbb{R}^2\times[0,T]}|\nabla f|^2\ud\xx\ud t+R^{-2}\int_{\mathbb{R}^2\times[0,T]}|f|^2\ud\xx\ud t\bigg).
\end{align*}
\end{lemma}

\begin{proposition}\label{Fp-v-4-prop}
Let $(\Fp,\vv)\in V(0,T)\times H(0,T)$ be a solution to the system {\rm (\ref{new-frame-equation-n1})--(\ref{imcompressible-v})} with initial data $(\Fp^{(0)},\vv^{(0)})\in H^1_{\Fp^{*}}\big(\mathbb{R}^2,SO(3)\big)\times L^2(\mathbb{R}^2,\mathbb{R}^2)$. Then there exists constants $\ve_1, R_0>0$ such that if
\begin{align*}
  \mathop{\mathrm{ess\sup}}_{\xx\in\mathbb{R}^2,\tau\leq t\leq T}\int_{B_R(\xx)}|\nabla \Fp(\cdot,t)|^2+|\vv(\cdot,t)|^2\ud\xx<\ve_1
\end{align*}
for any $R\in(0,R_0]$, then
\begin{align}
&\int_{\mathbb{R}^2\times[0,T]}\big(|\nabla^2\Fp|^2+|\nabla\vv|^2\big)\ud\xx\ud t\leq C(1+TR^{-2})E\big(\Fp^{(0)},\vv^{(0)}\big),\label{nabla2Fp-v-L2}\\
&\int_{\mathbb{R}^2\times[0,T]}\big(|\nabla\Fp|^4+|\vv|^4\big)\ud\xx\ud t\leq C\ve_1(1+TR^{-2})E\big(\Fp^{(0)},\vv^{(0)}\big).\label{nablaFp-v-L4}
\end{align}
\end{proposition}

\begin{proof}
It immediately follows from Proposition \ref{energ-diss-prop} that
\begin{align}\label{dissip-ineq-time}
\sum^3_{k=1}\frac{1}{\chi_k}\int_{\mathbb{R}^2\times[0,T]}|\ML_k\CF_{Bi}|^2\ud\xx\ud t+\eta\int_{\mathbb{R}^2\times[0,T]}|\nabla\vv|^2\ud\xx\ud t
\leq E\big(\Fp^{(0)},\vv^{(0)}\big).
\end{align}
It remains to estimate the first term of the left-side of (\ref{dissip-ineq-time}). First, we recall again the following equality with the help of orthogonal decomposition:
\begin{align}\label{A-DFp-orthogonal-decomp}
A\cdot\Delta\Fp=\sum^3_{k=1}\frac{1}{|V_k|^2}(A\cdot V_k)(\Delta\Fp\cdot V_k)+\sum^6_{k=1}\frac{1}{|W_k|^2}(A\cdot W_k)(\Delta\Fp\cdot W_k),
\end{align}
for any matrix $A\in\mathbb{R}^{3\times3}$, where $V_k(k=1,2,3)$ and $W_k(k=1,\cdots,6)$ are an orthogonal basis of the tangent space $T_{\Fp}SO(3)$ and its associated orthogonal complement space, respectively. According the definitions of $W_k(k=1,\cdots,6)$, we have
\begin{align*}
\Delta\Fp\cdot W_k=\nabla\cdot(\nabla\Fp\cdot W_k)-\nabla\Fp\cdot\nabla W_k=-\nabla\Fp\cdot\nabla W_k.
\end{align*}
Taking $A=\Delta\Fp$ in (\ref{A-DFp-orthogonal-decomp}), we deduce that
\begin{align}\label{DeFp-decomp}
\int_{\mathbb{R}^2\times[0,T]}|\Delta\Fp|^2\ud\xx\ud t\leq&\frac{1}{2}\int_{\mathbb{R}^2\times[0,T]}\sum^3_{k=1}(\Delta\Fp\cdot V_k)^2\ud\xx\ud t\nonumber\\
&+C\int_{\mathbb{R}^2\times[0,T]}|\nabla\Fp|^2(|\nabla^2\Fp|+|\nabla\Fp|^2)\ud\xx\ud t.
\end{align}
Similarly, taking
\begin{align*}
A=\nabla\cdot\frac{\partial f_{Bi}}{\partial(\nabla\Fp)}-\gamma\Delta\Fp,\quad \gamma=\min\{\gamma_1,\gamma_2,\gamma_3\}
\end{align*}
in (\ref{A-DFp-orthogonal-decomp}), we obtain
\begin{align}\label{Y-DeFp-decomp}
&\int_{\mathbb{R}^2\times[0,T]}\Big(\nabla\cdot\frac{\partial f_{Bi}}{\partial(\nabla\Fp)}-\gamma\Delta\Fp\Big)\cdot\Delta\Fp\ud\xx\ud t\nonumber\\
&\quad\leq\frac{1}{2}\int_{\mathbb{R}\times[0,T]}\sum^3_{k=1}\Big[\Big(\nabla\cdot\frac{\partial f_{Bi}}{\partial(\nabla\Fp)}-\gamma\Delta\Fp\Big)\cdot V_k\Big](\Delta\Fp\cdot V_k)\ud\xx\ud t\nonumber\\
&\qquad+C\int_{\mathbb{R}^2\times[0,T]}|\nabla\Fp|^2(|\nabla^2\Fp|+|\nabla\Fp|^2)\ud\xx\ud t.
\end{align}
Moreover, by Lemma \ref{h-decomposition} and integration by parts, it follows that
\begin{align}\label{Y-DeFp-decomp-neq}
&\int_{\mathbb{R}^2\times[0,T]}\Big(\nabla\cdot\frac{\partial f_{Bi}}{\partial(\nabla\Fp)}-\gamma\Delta\Fp\Big)\cdot\Delta\Fp\ud\xx\ud t\nonumber\\
&\quad\geq\sum^3_{i=1}\int_{\mathbb{R}^2\times[0,T]}\Big(k_i|\nabla{\rm div}\nn_i|^2+\sum^3_{j=1}k_{ji}|\nabla(\nn_j\cdot(\nabla\times\nn_i))|^2\Big)\ud\xx\ud t\nonumber\\
&\qquad-C\int_{\mathbb{R}^2\times[0,T]}|\nabla\Fp|^2(|\nabla^2\Fp|+|\nabla\Fp|^2)\ud\xx\ud t.
\end{align}

We let $\chi^{-1}=\min\{\chi^{-1}_1,\chi^{-1}_2,\chi^{-1}_3\}$, and recall
\begin{align*}
\frac{\delta\CF_{Bi}}{\delta\Fp}=-\nabla\cdot\frac{\partial f_{Bi}}{\partial(\nabla\Fp)}+\frac{\partial f_{Bi}}{\partial\Fp},\quad \frac{\delta\CF_{Bi}}{\delta\Fp}=\Big(\frac{\delta\CF_{Bi}}{\delta\nn_1},\frac{\delta\CF_{Bi}}{\delta\nn_2},\frac{\delta\CF_{Bi}}{\delta\nn_3}\Big),
\end{align*}
where $\frac{\partial f_{Bi}}{\partial\Fp}$ consists of the lower-order derivatives. Then, from the definitions of $\ML_k(k=1,2,3)$ and (\ref{DeFp-decomp})--(\ref{Y-DeFp-decomp-neq}), there holds
\begin{align}\label{higher-deriv-estimate}
&\sum^3_{k=1}\frac{1}{\chi_k}\int_{\mathbb{R}^2\times[0,T]}|\ML_k\CF_{Bi}|^2\ud\xx\ud t
=\sum^3_{k=1}\frac{1}{\chi_k}\int_{\mathbb{R}^2\times[0,T]}\Big(V_k\cdot\frac{\delta\CF_{Bi}}{\delta\Fp}\Big)^2\ud\xx\ud t\nonumber\\
&\quad\geq\frac{1}{\chi}\int_{\mathbb{R}^2\times[0,T]}\sum^3_{k=1}\Big(V_k\cdot\Big(\nabla\cdot\frac{\partial f_{Bi}}{\partial(\nabla\Fp)}\Big)\Big)^2\ud\xx\ud t-C\int_{\mathbb{R}^2\times[0,T]}|\nabla\Fp|^2(|\nabla^2\Fp|+|\nabla\Fp|^2)\ud\xx\ud t\nonumber\\
&\quad\geq\frac{2\gamma}{\chi}\int_{\mathbb{R}^2\times[0,T]}\sum^3_{k=1}\Big[\Big(\nabla\cdot\frac{\partial f_{Bi}}{\partial(\nabla\Fp)}-\gamma\Delta\Fp\Big)\cdot V_{k}\Big](\Delta\Fp\cdot V_{k})\ud\xx\ud t\nonumber\\
&\qquad+\frac{\gamma^2}{\chi}\int_{\mathbb{R}^2\times[0,T]}\sum^3_{k=1}(\Delta\Fp\cdot V_{k})^2\ud\xx\ud t-C\int_{\mathbb{R}^2\times[0,T]}|\nabla\Fp|^2(|\nabla^2\Fp|+|\nabla\Fp|^2)\ud\xx\ud t\nonumber\\
&\quad\geq\frac{2\gamma}{\chi}\sum^3_{i=1}\int_{\mathbb{R}^2\times[0,T]}\Big(k_i|\nabla{\rm div}\nn_i|^2+\sum^3_{j=1}k_{ji}|\nabla(\nn_j\cdot(\nabla\times\nn_i))|^2\Big)\ud\xx\ud t\nonumber\\
&\qquad+\frac{2\gamma^2}{\chi}\int_{\mathbb{R}^2\times[0,T]}|\Delta\Fp|^2\ud\xx\ud t-C\int_{\mathbb{R}^2\times[0,T]}|\nabla\Fp|^2(|\nabla^2\Fp|+|\nabla\Fp|^2)\ud\xx\ud t.
\end{align}
Thus, a combination of (\ref{dissip-ineq-time}) and (\ref{higher-deriv-estimate}) leads to
\begin{align*}
&\frac{2\gamma^2}{\chi}\int_{\mathbb{R}^2\times[0,T]}|\Delta\Fp|^2\ud\xx\ud t+\eta\int_{\mathbb{R}^2\times[0,T]}|\nabla\vv|^2\ud\xx\ud t\nonumber\\
&\quad\leq \frac{\gamma^2}{\chi}\int_{\mathbb{R}^2\times[0,T]}|\nabla^2\Fp|^2\ud\xx\ud t+C\int_{\mathbb{R}^2\times[0,T]}|\nabla\Fp|^4\ud\xx\ud t+E\big(\Fp^{(0)},\vv^{(0)}\big).
\end{align*}

Notice that
\begin{align*}
    \int_{\mathbb{R}^2\times[0,T]}|\Delta\Fp|^2\ud\xx\ud t=\int_{\mathbb{R}^2\times[0,T]}|\nabla^2\Fp|^2\ud\xx\ud t.
\end{align*}
Applying Lemma \ref{L4-norm-lemma},
%and choosing sufficiently small $\ve_1>0$ {\color{red}[CHECK]},
we arrive at
\begin{align*}
\int_{\mathbb{R}^2\times[0,T]}(|\nabla\Fp|^4+|\vv|^4)\ud\xx\ud t\leq& C\ve_1\int_{\mathbb{R}^2\times[0,T]}(|\nabla^2\Fp|^2+|\nabla\vv|^2)\ud\xx\ud t\\
&+C\ve_1R^{-2}\int_{\mathbb{R}^2\times[0,T]}(|\nabla\Fp|^2+|\vv|^2)\ud\xx\ud t.
\end{align*}
Therefore, (\ref{nabla2Fp-v-L2}) and (\ref{nablaFp-v-L4}) hold by using the above two estimates and choosing $\ve_1>0$ to be sufficiently small.

\end{proof}

The following proposition gives a local energy inequality for a smooth solution $(\Fp,\vv)$ under the condition of uniformly small local energy.

\begin{proposition}\label{local-monotonic-ineq-prop}
Let $(\Fp,\vv)\in V(0,T)\times H(0,T)$ be a solution to the system {\rm (\ref{new-frame-equation-n1})--(\ref{imcompressible-v})} with initial data $(\Fp^{(0)},\vv^{(0)})\in H^1_{\Fp^{*}}\big(\mathbb{R}^2,SO(3)\big)\times L^2(\mathbb{R}^2,\mathbb{R}^2)$. Assume that there exists constants $\ve_1, R_0>0$ such that
\begin{align*}
 \mathop{\mathrm{ess\sup}}_{\xx\in\mathbb{R}^2,\tau\leq t\leq T}\int_{B_{R_0}(\xx)}\big(|\nabla \Fp(\cdot,t)|^2+|\vv(\cdot,t)|^2\big)\ud\xx<\ve_1.
\end{align*}
Then, for all $s\in[0,T], \xx_0\in\mathbb{R}^2$ and $R\leq R_0$, it follows that
\begin{align*}
&\int_{B_R(\xx_0)}e(\Fp,\vv)(\cdot,s)\ud\xx+\frac{\eta}{2}\int^s_0\int_{B_R(\xx_0)}|\nabla\vv|^2\ud\xx\ud t+\sum^3_{k=1}\frac{1}{2\chi_k}\int^s_0\int_{B_R(\xx_0)}|\ML_k\CF_{Bi}|^2\ud\xx\ud t\\
&\quad\leq\int_{B_{2R}(\xx_0)}e\big(\Fp^{(0)},\vv^{(0)}\big)\ud\xx+C_2\frac{s^{\frac{1}{2}}}{R}\Big(1+\frac{s}{R^2}\Big)^{\frac{1}{2}}E\big(\Fp^{(0)},\vv^{(0)}\big),
\end{align*}
where $C_2$ is a uniform positive constant.
\end{proposition}

\begin{proof}
Let $\phi\in C^{\infty}_c(B_{2R}(\xx_0))$ be a cut-off function with $\phi\equiv 1$ on $B_R(\xx_0)$ and $|\nabla\phi|\leq \frac{C}{R}$, $|\nabla^2\phi|\leq\frac{C}{R^2}$ for all $R\leq R_0$.

Multiplying the equation (\ref{frame-equation-v}) by $\vv\phi^2$ and integrating by parts over the space $\mathbb{R}^2$, we obtain
\begin{align}\label{v-L-2-phi}
&\frac{1}{2}\frac{\ud}{\ud t}\int_{\mathbb{R}^2}|\vv|^2\phi^2\ud\xx+\eta\int_{\mathbb{R}^2}|\nabla\vv|^2\phi^2\ud\xx\nonumber\\
&\quad=\int_{\mathbb{R}^2}(|\vv|^2+2p)\phi\vv\cdot\nabla\phi\ud\xx+\eta\int_{\mathbb{R}^2}|\vv|^2(|\nabla\phi|^2+\phi\Delta\phi)\ud\xx\nonumber\\
&\qquad\underbrace{-\langle\sigma,\nabla(\vv\phi^2)\rangle}_{\CA_1}
%\underbrace{-\langle\sigma_2,\nabla(\vv\phi^2)\rangle}_{\CA_2}
+\underbrace{\langle\mathfrak{F},\vv\phi^2\rangle}_{\CA_2},
\end{align}
where we have used the following facts:
\begin{align*}
\big\langle\partial_t\vv-\eta\Delta\vv,\vv\phi\big\rangle=&\frac{1}{2}\frac{\ud}{\ud t}\int_{\mathbb{R}^2}|\vv|^2\phi^2\ud\xx+\eta\int_{\mathbb{R}^2}|\nabla\vv|^2\phi^2\ud\xx\\
&-\eta\int_{\mathbb{R}^2}|\vv|^2(|\nabla\phi|^2+\phi\Delta\phi)\ud\xx,\\
\langle\vv\cdot\nabla\vv+\nabla p,\vv\phi^2\rangle=&-\int_{\mathbb{R}^2}(|\vv|^2+2p)\phi\vv\cdot\nabla\phi\ud\xx.
\end{align*}
Using (\ref{sigma-e-nabla-v}) and integrating by parts, we deduce that
\begin{align*}
%&\CA_1+\int_{\mathbb{R}^2}(\sigma_1)_{ij}v_i\partial_j(\phi^2)\ud\xx=-\int_{\mathbb{R}^2}\omega^T_sP\omega_s\phi^2\ud\xx\leq 0,\\
&\CA_1+\int_{\mathbb{R}^2}\sigma_{ij}v_i\partial_j(\phi^2)\ud\xx\\
&\quad\leq-\Big(\frac{\eta_3}{\chi_3}\int_{\mathbb{R}^2}(\A\cdot\sss_3)\ML_3\CF_{Bi}\phi^2\ud\xx+\frac{\eta_2}{\chi_2}\int_{\mathbb{R}^2}(\A\cdot\sss_4)\ML_2\CF_{Bi}\phi^2\ud\xx+\frac{\eta_1}{\chi_1}\int_{\mathbb{R}^2}(\A\cdot\sss_5)\ML_1\CF_{Bi}\phi^2\ud\xx\Big)\\
&\qquad-\frac{1}{2}\int_{\mathbb{R}^2}\Big((\BOm\cdot\aaa_1)\ML_3\CF_{Bi}+(\BOm\cdot\aaa_2)\ML_2\CF_{Bi}+(\BOm\cdot\aaa_3)\ML_1\CF_{Bi}\Big)\phi^2\ud\xx.
\end{align*}
From the definition of $\mathfrak{F}$ in \eqref{external-force-F}, and \eqref{F-v-inner}, it holds that
\begin{align*}
\CA_2=-\int_{\mathbb{R}^2}\phi^2\big((v_i\partial_i\nn_1)\cdot\hh_1+(v_i\partial_i\nn_2)\cdot\hh_2+(v_i\partial_i\nn_3)\cdot\hh_3\big)\ud\xx.
\end{align*}
Recalling \eqref{energy-p-time},  we have
\begin{align}\label{Fp-L-2-phi}
\frac{\ud}{\ud t}\int_{\mathbb{R}^2}f_{Bi}(\Fp,\nabla\Fp)\phi^2\ud\xx
&=-\CA_2+\int_{\mathbb{R}^2}\Big(\frac{1}{2}\BOm\cdot\aaa_1+\frac{\eta_3}{\chi_3}\A\cdot\sss_3-\frac{1}{\chi_3}\ML_3\CF_{Bi}\Big)\ML_3\CF_{Bi}\phi^2\ud\xx\nonumber\\
&\quad+\int_{\mathbb{R}^2}\Big(\frac{1}{2}\BOm\cdot\aaa_2+\frac{\eta_2}{\chi_2}\A\cdot\sss_4-\frac{1}{\chi_2}\ML_2\CF_{Bi}\Big)\ML_2\CF_{Bi}\phi^2\ud\xx\nonumber\\
&\quad+\int_{\mathbb{R}^2}\Big(\frac{1}{2}\BOm\cdot\aaa_3+\frac{\eta_1}{\chi_1}\A\cdot\sss_5-\frac{1}{\chi_1}\ML_1\CF_{Bi}\Big)\ML_1\CF_{Bi}\phi^2\ud\xx.
\end{align}
Thus, combining \eqref{v-L-2-phi} and \eqref{Fp-L-2-phi}, we arrive at
\begin{align}\label{energy-L2-phi}
&\frac{\ud}{\ud t}\int_{\mathbb{R}^2}e(\Fp,\vv)\phi^2\ud\xx+\eta\int_{\mathbb{R}^2}|\nabla\vv|^2\phi^2\ud\xx+\sum^3_{k=1}\frac{1}{\chi_k}\int_{\mathbb{R}^2}|\ML_k\CF_{Bi}|^2\phi^2\ud\xx\nonumber\\
&\quad\leq\underbrace{\int_{\mathbb{R}^2}(|\vv|^2+2p)\phi\vv\cdot\nabla\phi\ud\xx}_{\CA_3}+\underbrace{\eta\int_{\mathbb{R}^2}|\vv|^2(|\nabla\phi|^2+\phi\Delta\phi)\ud\xx}_{\CA_4}\underbrace{-\int_{\mathbb{R}^2}\sigma_{ij}v_i\partial_j(\phi^2)\ud\xx}_{\CA_5}.
\end{align}
Making use of the equations (\ref{frame-equation-n1})--(\ref{frame-equation-n3}) and the definition of the stress $\sigma$, it follows that
\begin{align*}
\CA_5\leq &C\int_{\mathbb{R}^2}\Big(|\nabla\vv|+\sum^3_{k=1}|\ML_k\CF_{Bi}|\Big)|\vv||\phi||\nabla\phi|\ud\xx\\
\leq&\frac{\eta}{4}\int_{\mathbb{R}^2}|\nabla\vv|^2\phi^2\ud\xx+\sum^3_{k=1}\frac{1}{2\chi_k}\int_{\mathbb{R}^2}|\ML_k\CF_{Bi}|^2\phi^2\ud\xx+C\int_{\mathbb{R}^2}|\vv|^2|\nabla\phi^2|\ud\xx.
\end{align*}
Consequently, using the estimate of $\CA_5$, \eqref{energy-L2-phi} immediately becomes
\begin{align*}
&\frac{\ud}{\ud t}\int_{\mathbb{R}^2}e(\Fp,\vv)\phi^2\ud\xx+\frac{\eta}{2}\int_{\mathbb{R}^2}|\nabla\vv|^2\phi^2\ud\xx+\sum^3_{k=1}\frac{1}{2\chi_k}\int_{\mathbb{R}^2}|\ML_k\CF_{Bi}|^2\phi^2\ud\xx\nonumber\\
&\quad\leq C(\CA_3+\CA_4).
\end{align*}

It can be deduced from Proposition \ref{Fp-v-4-prop} that
\begin{align*}
\int^s_0|\CA_4|\ud t\leq C\frac{s}{R^2}E\big(\Fp^{(0)},\vv^{(0)}\big).
\end{align*}
Notice that from \eqref{frame-equation-v}, there has
\begin{align*}
\Delta p=\partial_i\partial_j\big(\sigma_{ij}-v_iv_j\big)+\partial_i\mathfrak{F}_i\quad \text{on}~\mathbb{R}^2\times[0,T].
\end{align*}
Applying the Calder$\acute{\rm o}$n--Zygmund estimate (see \cite{CKN}), Proposition \ref{energ-diss-prop} and Proposition \ref{Fp-v-4-prop}, we obtain
\begin{align*}
\int_{\mathbb{R}^2\times[0,s]}|p|^2\ud\xx\ud t\leq& C\int_{\mathbb{R}^2\times[0,s]}\big(|\nabla\Fp|^4+|\vv|^4+|\nabla\vv|^2+|\nabla^2\Fp|^2\big)\ud\xx\ud t\\
\leq &C\ve_1\Big(1+\frac{s}{R^2}\Big)E\big(\Fp^{(0)},\vv^{(0)}\big).
\end{align*}
By H$\ddot{\rm o}$lder inequality, Proposition \ref{energ-diss-prop} and Proposition \ref{Fp-v-4-prop}, it follows that
\begin{align*}
\int^s_0|\CA_3|\ud t\leq& C\bigg[\bigg(\int_{\mathbb{R}^2\times[0,s]}|\vv|^4\ud\xx\ud t\bigg)^{\frac{1}{2}}+\bigg(\int_{\mathbb{R}^2\times[0,s]}|p|^2\ud\xx\ud t\bigg)^{\frac{1}{2}}\bigg]\bigg(\int_{\mathbb{R}^2\times[0,s]}\frac{|\vv|^2}{R^2}\ud\xx\ud t\bigg)^{\frac{1}{2}}\\
\leq& C\ve^{\frac{1}{2}}_1\frac{s^{\frac{1}{2}}}{R}\Big(1+\frac{s}{R^2}\Big)^{\frac{1}{2}}E\big(\Fp^{(0)},\vv^{(0)}\big).
\end{align*}
The proposition follows from the above estimates.

\end{proof}

The following proposition will be devoted to studying the higher regularity of the solution $(\Fp,\vv)$ to the system (\ref{new-frame-equation-n1})--(\ref{imcompressible-v}).

\begin{proposition}\label{higher-regularity-l-prop}
Let $(\Fp,\vv)\in V(0,T)\times H(0,T)$ be a solution to the system {\rm (\ref{new-frame-equation-n1})--(\ref{imcompressible-v})} with initial data $(\Fp^{(0)},\vv^{(0)})\in H^1_{\Fp^{*}}\big(\mathbb{R}^2,SO(3)\big)\times L^2(\mathbb{R}^2,\mathbb{R}^2)$. Assume that there exists constants $\ve_1, R_0>0$ such that
\begin{align*}
 \mathop{\mathrm{ess\sup}}_{\xx\in\mathbb{R}^2,\tau\leq t\leq T}\int_{B_{R}(\xx)}\big(|\nabla \Fp(\cdot,t)|^2+|\vv(\cdot,t)|^2\big)\ud\xx<\ve_1.
\end{align*}
Then for all $t\in[\tau,T]$ with $\tau\in(0,T)$ and for all $l\geq 1$, it follows that
\begin{align}\label{higher-regularity-ineq}
&\int_{\mathbb{R}^2}\big(|\nabla^{l+1}\Fp(\cdot,t)|^2+|\nabla^l\vv(\cdot,t)|^2\big)\ud\xx+\int^t_{\tau}\int_{\mathbb{R}^2}\big(|\nabla^{l+2}\Fp(\cdot,s)|^2+|\nabla^{l+1}\vv(\cdot,s)|^2\big)\ud\xx\ud s\nonumber\\
&\quad\leq C\Big(l,\ve_1,E_0,\tau,T,\frac{T}{R^2}\Big),
\end{align}
where $E_0=E(\Fp^{(0)},\vv^{(0)})$.
Moreover, the solution $(\Fp,\vv)$ is regular for all $t\in(0,T)$.
\end{proposition}

\begin{proof}
The proof is an induction on $l$, which will be divided into two steps.

{\bf Step 1}. For $l=1,2,3$,
we only provide here the arguments of (\ref{higher-regularity-ineq}) for the case $l=1$. We
relegate the proof of the case $l=2$ in (\ref{higher-regularity-ineq}) to the appendix so as not to distract
from the main body of this paper. The case of $l=3$ will be omitted due to the high similarity with the proof of the case $l=2$.

Multiplying the equation (\ref{frame-equation-v}) by $\Delta\vv$ and integrating by parts, we obtain
\begin{align}\label{Delta-v-L2}
&\frac{1}{2}\frac{\ud}{\ud t}\int_{\mathbb{R}^2}|\nabla\vv|^2\ud\xx+\eta\int_{\mathbb{R}^2}|\Delta\vv|^2\ud\xx\nonumber\\
&\quad=\langle\vv\cdot\nabla\vv,\Delta\vv\rangle-\langle\nabla\cdot\sigma,\Delta\vv\rangle-\langle\mathfrak{F},\Delta\vv\rangle\nonumber\\
&\quad\leq\frac{\eta}{4}\int_{\mathbb{R}^2}|\Delta\vv|^2\ud\xx+C\int_{\mathbb{R}^2}|\vv\cdot\nabla\vv|^2\ud\xx+C\int_{\mathbb{R}^2}\big(|\nabla^2\Fp|^2+|\nabla\Fp|^4\big)|\nabla\Fp|^2\ud\xx\nonumber\\
&\qquad-\langle\nabla\cdot\sigma,\Delta\vv\rangle,
\end{align}
where
\begin{align*}
|\nabla\Fp|^2=\sum^3_{i=1}|\nabla\nn_i|^2,\quad|\nabla\Fp|^4=\sum^3_{i=1}|\nabla\nn_i|^4,\quad |\nabla^2\Fp|^2=\sum^3_{i=1}|\nabla^2\nn_i|^2.
\end{align*}
Using the definition of the stress $\sigma$ and the equations (\ref{frame-equation-n1})--(\ref{frame-equation-n3}), it gives
\begin{align*}
 &-\langle\nabla\cdot\sigma,\Delta\vv\rangle=-\int_{\mathbb{R}^2}\partial_k\sigma\cdot\partial_k(\A+\BOm)\ud\xx\\
 &\quad\leq-\Big(\frac{\eta_3}{\chi_3}\int_{\mathbb{R}^2}(\partial_k\A\cdot\sss_3)\CH^{\nabla}_3\ud\xx+\frac{\eta_2}{\chi_2}\int_{\mathbb{R}^2}(\partial_k\A\cdot\sss_4)\CH^{\nabla}_2\ud\xx+\frac{\eta_1}{\chi_1}\int_{\mathbb{R}^2}(\partial_k\A\cdot\sss_5)\CH^{\nabla}_1\ud\xx\Big)\\
&\qquad-\frac{1}{2}\int_{\mathbb{R}^2}\Big((\partial_k\BOm\cdot\aaa_1)\CH^{\nabla}_3+(\partial_k\BOm\cdot\aaa_2)\CH^{\nabla}_2+(\partial_k\BOm\cdot\aaa_3)\CH^{\nabla}_1\Big)\ud\xx\\
&\qquad+C\int_{\mathbb{R}^2}|\nabla\Fp||\nabla\vv||\Delta\vv|\ud\xx+C\int_{\mathbb{R}^2}(|\nabla^2\Fp|+|\nabla\Fp|^2)|\nabla\Fp||\Delta\vv|\ud\xx,
\end{align*}
where
\begin{align*}
\CH^{\nabla}_1=\nn_2\cdot\partial_k\hh_3-\nn_3\cdot\partial_k\hh_2,\quad
\CH^{\nabla}_2=\nn_3\cdot\partial_k\hh_1-\nn_1\cdot\partial_k\hh_3,\quad
\CH^{\nabla}_3=\nn_1\cdot\partial_k\hh_2-\nn_2\cdot\partial_k\hh_1.
\end{align*}
Then, it follows from (\ref{Delta-v-L2}) and the above estimates that
\begin{align}\label{Delta-v-L2-Final}
&\frac{1}{2}\frac{\ud}{\ud t}\int_{\mathbb{R}^2}|\nabla\vv|^2\ud\xx+\frac{\eta}{2}\int_{\mathbb{R}^2}|\Delta\vv|^2\ud\xx\nonumber\\
&\quad\leq-\Big(\frac{\eta_3}{\chi_3}\int_{\mathbb{R}^2}(\partial_k\A\cdot\sss_3)\CH^{\nabla}_3\ud\xx+\frac{\eta_2}{\chi_2}\int_{\mathbb{R}^2}(\partial_k\A\cdot\sss_4)\CH^{\nabla}_2\ud\xx+\frac{\eta_1}{\chi_1}\int_{\mathbb{R}^2}(\partial_k\A\cdot\sss_5)\CH^{\nabla}_1\ud\xx\Big)\nonumber\\
&\qquad-\frac{1}{2}\int_{\mathbb{R}^2}\Big((\partial_k\BOm\cdot\aaa_1)\CH^{\nabla}_3+(\partial_k\BOm\cdot\aaa_2)\CH^{\nabla}_2+(\partial_k\BOm\cdot\aaa_3)\CH^{\nabla}_1\Big)\ud\xx\nonumber\\
&\qquad+C\int_{\mathbb{R}^2}(|\vv|^2+|\nabla\Fp|^2)(|\nabla^2\Fp|^2+|\nabla\Fp|^4+|\nabla\vv|^2)\ud\xx.
\end{align}

On the other hand, we first denote
\begin{align*}
\CF^{\nabla}_{Bi}(\Fp)\eqdefa&\frac{1}{2}\sum^3_{i=1}\Big(\gamma_i\|\Delta\nn_i\|^2_{L^2}+k_i\|\nabla{\rm div}\nn_i\|^2_{L^2}+\sum^3_{j=1}k_{ji}\|\nabla(\nn_j\cdot\nabla\times\nn_i)\|^2_{L^2}\Big).
\end{align*}
Using the definitions of $\hh_i(i=1,2,3)$ and the equations (\ref{new-frame-equation-n1})--(\ref{new-frame-equation-n3}), we have the following facts:
\begin{align*}
&-\nn_i\cdot\Delta\nn_i=|\nabla\nn_i|^2,\quad |\hh_i|\leq C(|\nabla\Fp|^2+|\nabla^2\Fp|),\\
&|\nabla\hh_i|\leq C(|\nabla\Fp|^3+|\nabla^2\Fp||\nabla\Fp|+|\nabla^3\Fp|),\\
&|\nabla\partial_t\nn_i|\leq C\Big(|\Delta\vv|+|\nabla\vv||\nabla\Fp|+\sum^3_{i=1}|\hh_i||\nabla\Fp|+\sum^3_{i=1}|\CH^{\nabla}_i|\Big).
\end{align*}
Differentiating the equations (\ref{new-frame-equation-n1})--(\ref{new-frame-equation-n3}) in $x_{\alpha}$, respectively, multiplying its by $-\partial_{\alpha}\hh_i(i=1,2,3)$ and then integrating by parts, one can obtain
 \begin{align}
&\frac{\ud}{\ud t}\CF^{\nabla}_{Bi}(\Fp)-\sum^3_{i=1}\int_{\mathbb{R}^2}[(\partial_{\alpha}\vv\cdot\nabla)\nn_i+(\vv\cdot\nabla)\partial_{\alpha}\nn_i]\cdot\partial_{\alpha}\hh_i\ud\xx\nonumber\\
&\quad\leq\int_{\mathbb{R}^2}\Big(\frac{1}{2}\partial_{\alpha}\BOm\cdot\aaa_1+\frac{\eta_3}{\chi_3}\partial_{\alpha}\A\cdot\sss_3-\frac{1}{\chi_3}\CH^{\nabla}_3\Big)\CH^{\nabla}_3\ud\xx\nonumber\\
&\qquad+\int_{\mathbb{R}^2}\Big(\frac{1}{2}\partial_{\alpha}\BOm\cdot\aaa_2+\frac{\eta_2}{\chi_2}\partial_{\alpha}\A\cdot\sss_4-\frac{1}{\chi_2}\CH^{\nabla}_2\Big)\CH^{\nabla}_2\ud\xx\nonumber\\
&\qquad+\int_{\mathbb{R}^2}\Big(\frac{1}{2}\partial_{\alpha}\BOm\cdot\aaa_3+\frac{\eta_1}{\chi_1}\partial_{\alpha}\A\cdot\sss_5-\frac{1}{\chi_1}\CH^{\nabla}_1\Big)\CH^{\nabla}_1\ud\xx\nonumber\\
&\qquad+\delta\sum^3_{i=1}\int_{\mathbb{R}^2}|\nabla\partial_t\nn_i|^2\ud\xx+C_{\delta}\int_{\mathbb{R}^2}|\nabla\Fp|^2(|\nabla^2\Fp|^2+|\nabla\Fp|^4)\ud\xx\nonumber\\
&\qquad+C\sum^3_{i=1}\int_{\mathbb{R}^2}|\nabla\vv||\nabla\Fp|||\nabla\hh_i|\ud\xx,
 \end{align}
where $\delta>0$ is a small parameter to be determined later.
Similar to the estimate of (\ref{higher-deriv-estimate}), we have
\begin{align}\label{higher-deriv-estimate-CHn}
-\sum^3_{i=1}\frac{1}{\chi_i}\int_{\mathbb{R}^2}|\CH^{\nabla}_i|^2\ud\xx\leq&-\frac{2\gamma^2}{\chi}\int_{\mathbb{R}^2}|\nabla^3\Fp|^2\ud\xx+C\int_{\mathbb{R}^2}|\nabla\Fp|^2(|\nabla^2\Fp|^2+|\nabla\Fp|^4)\ud\xx.
\end{align}
Thus, from (\ref{Delta-v-L2-Final})--(\ref{higher-deriv-estimate-CHn}), choosing the appropriate small parameter $\delta>0$, and using the Gagliardo--Nirenberg--Sobolev inequality, we obtain
\begin{align*}
&\frac{\ud}{\ud t}\int_{\mathbb{R}^2}\Big(\frac{1}{2}|\nabla\vv|^2+\CF^{\nabla}_{Bi}(\Fp)\Big)\ud\xx+\int_{\mathbb{R}^2}\Big(\frac{\eta}{4}|\Delta\vv|^2+\frac{\gamma^2}{\chi}|\nabla^3\Fp|^2\Big)\ud\xx\\
&\quad\leq C\int_{\mathbb{R}^2}(|\vv|^2+|\nabla\Fp|^2)(|\nabla\vv|^2+|\nabla^2\Fp|^2)\ud\xx\\
&\quad \leq C\Big(\int_{\mathbb{R}^2}(|\vv|^4+|\nabla\Fp|^4)\ud\xx\Big)^{\frac{1}{2}}\Big(\int_{\mathbb{R}^2}(|\nabla\vv|^4+|\nabla^2\Fp|^4)\ud\xx\Big)^{\frac{1}{2}}\\
&\quad\leq\min\Big\{\frac{\eta}{8},\frac{\gamma^2}{2\chi}\Big\}\int_{\mathbb{R}^2}(|\Delta\vv|^2+|\nabla^3\Fp|^2)\ud\xx\\
&\qquad+C\Big(\int_{\mathbb{R}^2}(|\vv|^4+|\nabla\Fp|^4)\ud\xx\Big)\Big(\int_{\mathbb{R}^2}(|\nabla\vv|^2+|\nabla^2\Fp|^2)\ud\xx\Big).
\end{align*}
It can be seen from (\ref{nabla2Fp-v-L2}) that for $\tau\in (0,T)$, there exists $\tau'\in (0,\tau)$, such that
\begin{align*}
\int_{\mathbb{R}^2}(|\nabla^2\Fp|^2+|\nabla\vv|^2)(\cdot,\tau')\ud\xx<C\Big(\tau,E_0,\frac{T}{R^2}\Big).
\end{align*}
Then, by Gronwall's inequality and Proposition \ref{Fp-v-4-prop}, it holds that
\begin{align}\label{higher-regular-l=1}
&\int_{\mathbb{R}^2}(|\nabla^2\Fp|^2+|\nabla\vv|^2)(\cdot,t)\ud\xx\nonumber\\
&\quad+\int^t_{\tau}\int_{\mathbb{R}^2}(|\nabla^3\Fp|^2+|\Delta\vv|^2)(\cdot,s)\ud\xx\ud s\leq C\Big(\ve_1,E_0,\tau,T,\frac{T}{R^2}\Big),
\end{align}
for all $t\in[\tau,T]$ with $\tau\in (0,T)$.

{\bf Step 2}. Assume now the estimate (\ref{higher-regularity-ineq}) is valid for some nonnegative integer $l_0\geq3$, and we prove that the case $l=l_0+1$ is still valid. Further, we only prove the case that $l_0$ is odd, that is, for $l=1,\cdots,2k-1$ and $k\geq2$, (\ref{higher-regularity-ineq}) holds. In other words, we prove the case $l=2k$. We also omit the case that $l_0$ is even, because the proof can be obtained by the same way.

By the assumptions and Sobolev's  embedding inequality, we have
\begin{align}
&|\nabla\Fp|+|\nabla^2\Fp|+|\vv|+|\nabla\vv|\leq C\Big(\ve_1,E_0,\tau,T,\frac{T}{R^2}\Big),\quad \text{in}~\mathbb{R}^2\times(\tau,T),\label{low-ineq-estimate-1}\\
&\int_{\mathbb{R}^2}(|\nabla^{2k}\Fp|^2+|\nabla^{2k-1}\vv|)(\cdot,t)\ud\xx\nonumber\\
&\quad+\int^t_{\tau}\int_{\mathbb{R}^2}(|\nabla^{2k+1}\Fp|^2+|\nabla^{2k}\vv|^2)(\cdot,s)\ud\xx\ud s+C\Big(k,\ve_1,E_0,\tau,T,\frac{T}{R^2}\Big),~~\forall t\in(\tau, T), \label{high-ineq-estimate-2}
\end{align}
which implies the following estimate:
\begin{align}\label{mid-ineq-estimate-3}
|\nabla^{2k-2}\Fp|+|\nabla^{2k-3}\vv|\leq C\Big(k,\ve_1,E_0,\tau,T,\frac{T}{R^2}\Big),\quad \text{in}~\mathbb{R}^2\times(\tau,T).
\end{align}
Similar to the process of Step 1, multiplying the equation (\ref{frame-equation-v}) by $\Delta^{2k}\vv$ and integrating by parts, we deduce that
\begin{align}\label{Delta-k-v-L2}
&\frac{1}{2}\frac{\ud}{\ud t}\int_{\mathbb{R}^2}|\nabla^{2k}\vv|^2\ud\xx+\eta\int_{\mathbb{R}^2}|\nabla\Delta^k\vv|^2\ud\xx\nonumber\\
&\quad=\langle\vv\cdot\nabla\vv,\Delta^{2k}\vv\rangle-\langle\nabla\cdot\sigma,\Delta^{2k}\vv\rangle-\langle\mathfrak{F},\Delta^{2k}\vv\rangle\nonumber\\
&\quad\leq\frac{\eta}{4}\int_{\mathbb{R}^2}|\nabla\Delta^k\vv|^2\ud\xx+C\int_{\mathbb{R}^2}|\nabla^{2k-1}(\vv\cdot\nabla\vv)|^2\ud\xx\nonumber\\
&\qquad+\int_{\mathbb{R}^2}(\nabla\cdot\sigma)\cdot\Delta^{2k}\vv\ud\xx+C\int_{\mathbb{R}^2}|\nabla^{2k-1}\mathfrak{F}|^2\ud\xx.
\end{align}
It follows from the definition of $\mathfrak{F}$ and (\ref{low-ineq-estimate-1})--(\ref{mid-ineq-estimate-3}) that
\begin{align*}
\int_{\mathbb{R}^2}|\nabla^{2k-1}\mathfrak{F}|^2\ud\xx\leq C\Big(k,\ve_1,E_0,\tau,T,\frac{T}{R^2}\Big)\Big(\int_{\mathbb{R}^2}|\nabla^{2k+1}\Fp|^2\ud\xx+1\Big).
\end{align*}
Similarly, we have
\begin{align*}
\int_{\mathbb{R}^2}|\nabla^{2k-1}(\vv\cdot\nabla\vv)|^2\ud\xx\leq C\Big(k,\ve_1,E_0,\tau,T,\frac{T}{R^2}\Big)\Big(\int_{\mathbb{R}^2}|\nabla^{2k}\vv|^2\ud\xx+1\Big).
\end{align*}
Applying the definition of $\sigma$ and (\ref{low-ineq-estimate-1})--(\ref{mid-ineq-estimate-3}), we can obtain from the equations (\ref{frame-equation-n1})--(\ref{frame-equation-n3}) that
\begin{align*}
&\int_{\mathbb{R}^2}(\nabla\cdot\sigma)\cdot\Delta^{2k}\vv\ud\xx
=-\int_{\mathbb{R}^2}\sigma\cdot\Delta^{2k}(\A+\BOm)\ud\xx\\
&\quad\leq-\Big(\beta_1\|\Delta^k\A\cdot\sss_1\|^2_{L^2}+2\beta_0\int_{\mathbb{R}^2}(\Delta^k\A\cdot\sss_1)(\Delta^k\A\cdot\sss_2)\ud\xx+\beta_2\|\Delta^k\A\cdot\sss_2\|^2_{L^2}\Big)\\
&\qquad-\Big(\beta_3-\frac{\eta^2_3}{\chi_3}\Big)\|\Delta^k\A\cdot\sss_3\|^2_{L^2}-\Big(\beta_4-\frac{\eta^2_2}{\chi_2}\Big)\|\Delta^k\A\cdot\sss_4\|^2_{L^2}-\Big(\beta_5-\frac{\eta^2_1}{\chi_1}\Big)\|\Delta^k\A\cdot\sss_5\|^2_{L^2}\\
&\qquad-\Big(\frac{\eta_3}{\chi_3}\int_{\mathbb{R}^2}(\Delta^k\A\cdot\sss_3)\CH^{\Delta^{k}}_3\ud\xx+\frac{\eta_2}{\chi_2}\int_{\mathbb{R}^2}(\Delta^k\A\cdot\sss_4)\CH^{\Delta^k}_2\ud\xx+\frac{\eta_1}{\chi_1}\int_{\mathbb{R}^2}(\Delta^k\A\cdot\sss_5)\CH^{\Delta^k}_1\ud\xx\Big)\\
&\qquad-\frac{1}{2}\int_{\mathbb{R}^2}\Big((\Delta^k\BOm\cdot\aaa_1)\CH^{\Delta^k}_3+(\Delta^k\BOm\cdot\aaa_2)\CH^{\Delta^k}_2+(\Delta^k\BOm\cdot\aaa_3)\CH^{\Delta^k}_1\Big)\ud\xx\\
&\qquad+\delta\int_{\mathbb{R}^2}|\nabla^{2k+1}\vv|^2\ud\xx+C\Big(\delta,k,\ve_1,E_0,\tau,T,\frac{T}{R^2}\Big)\Big(\int_{\mathbb{R}^2}(|\nabla^{2k}\vv|^2+|\nabla^{2k+1}\Fp|^2)\ud\xx+1\Big),
\end{align*}
where the small parameter $\delta>0$ can be suitably selected, and
\begin{align*}
&\CH^{\Delta^k}_1=\nn_2\cdot\Delta^k\hh_3-\nn_3\cdot\Delta^k\hh_2,\quad
\CH^{\Delta^k}_2=\nn_3\cdot\Delta^k\hh_1-\nn_1\cdot\Delta^k\hh_3,\\
&
\CH^{\Delta^k}_3=\nn_1\cdot\Delta^k\hh_2-\nn_2\cdot\Delta^k\hh_1.
\end{align*}
Consequently, combining (\ref{Delta-k-v-L2}) with the above estimates leads to
\begin{align}\label{v-energ-estimate-Delta-k}
&\frac{1}{2}\frac{\ud}{\ud t}\int_{\mathbb{R}^2}|\nabla^{2k}\vv|^2\ud\xx+\frac{\eta}{2}\int_{\mathbb{R}^2}|\nabla\Delta^k\vv|^2\ud\xx\nonumber\\
&\quad\leq-\Big(\frac{\eta_3}{\chi_3}\int_{\mathbb{R}^2}(\Delta^k\A\cdot\sss_3)\CH^{\Delta^{k}}_3\ud\xx+\frac{\eta_2}{\chi_2}\int_{\mathbb{R}^2}(\Delta^k\A\cdot\sss_4)\CH^{\Delta^k}_2\ud\xx+\frac{\eta_1}{\chi_1}\int_{\mathbb{R}^2}(\Delta^k\A\cdot\sss_5)\CH^{\Delta^k}_1\ud\xx\Big)\nonumber\\
&\qquad-\frac{1}{2}\int_{\mathbb{R}^2}\Big((\Delta^k\BOm\cdot\aaa_1)\CH^{\Delta^k}_3+(\Delta^k\BOm\cdot\aaa_2)\CH^{\Delta^k}_2+(\Delta^k\BOm\cdot\aaa_3)\CH^{\Delta^k}_1\Big)\ud\xx\nonumber\\
&\qquad+\delta\int_{\mathbb{R}^2}|\nabla^{2k+1}\vv|^2\ud\xx+C\Big(\delta,k,\ve_1,E_0,\tau,T,\frac{T}{R^2}\Big)\Big(\int_{\mathbb{R}^2}(|\nabla^{2k}\vv|^2+|\nabla^{2k+1}\Fp|^2)\ud\xx+1\Big).
\end{align}

We define the following energy functional:
\begin{align*}
\CF^{\Delta^k}_{Bi}(\Fp)\eqdefa&\frac{1}{2}\sum^3_{i=1}\Big(\gamma_i\|\Delta^k\nabla\nn_i\|^2_{L^2}+k_i\|\Delta^k{\rm div}\nn_i\|^2_{L^2}+\sum^3_{j=1}k_{ji}\|\Delta^k(\nn_j\cdot\nabla\times\nn_i)\|^2_{L^2}\Big).
\end{align*}
Acting the differential operator $\partial_{\beta}$ on (\ref{new-frame-equation-n1})--(\ref{new-frame-equation-n3}), multiplying these equations by
$\partial_{\beta}\Delta^{2k-1}\hh_i(i=1,2,3)$, respectively, and then integrating by parts, we obtain from  (\ref{low-ineq-estimate-1})--(\ref{mid-ineq-estimate-3}) that
\begin{align}\label{Fp-energ-estimate-Delta-k}
&\frac{\ud}{\ud t}\CF^{\Delta^k}_{Bi}(\Fp)+\sum^3_{i=1}\int_{\mathbb{R}^2}\big[(\partial_{\beta}\vv\cdot\nabla)\nn_i+(\vv\cdot\nabla)\partial_{\beta}\nn_i\big]\cdot\partial_{\beta}\Delta^{2k-1}\hh_i\ud\xx\nonumber\\
&\quad\leq\int_{\mathbb{R}^2}\Big(\frac{1}{2}\Delta^k\BOm\cdot\aaa_1+\frac{\eta_3}{\chi_3}\Delta^k\A\cdot\sss_3-\frac{1}{\chi_3}\CH^{\Delta^k}_3\Big)\CH^{\Delta^k}_3\ud\xx\nonumber\\
&\qquad+\int_{\mathbb{R}^2}\Big(\frac{1}{2}\Delta^k\BOm\cdot\aaa_2+\frac{\eta_2}{\chi_2}\Delta^k\A\cdot\sss_4-\frac{1}{\chi_2}\CH^{\Delta^k}_2\Big)\CH^{\Delta^k}_2\ud\xx\nonumber\\
&\qquad+\int_{\mathbb{R}^2}\Big(\frac{1}{2}\Delta^k\BOm\cdot\aaa_3+\frac{\eta_1}{\chi_1}\Delta^k\A\cdot\sss_5-\frac{1}{\chi_1}\CH^{\Delta^k}_1\Big)\CH^{\Delta^k}_1\ud\xx\nonumber\\
&\qquad+\delta\sum^3_{i=1}\int_{\mathbb{R}^2}|\nabla^{2k}\partial_t\nn_i|^2\ud\xx+\delta\int_{\mathbb{R}^2}|\nabla^{2k+2}\Fp|^2\ud\xx\nonumber\\
&\qquad +C\Big(\delta,k,\ve_1,E_0,\tau,T,\frac{T}{R^2}\Big)\Big(\int_{\mathbb{R}^2}|\nabla^{2k}\vv|^2\ud\xx+\int_{\mathbb{R}^2}|\nabla^{2k+1}\Fp|^2\ud\xx+1\Big).
\end{align}
By a direct calculation, we have
\begin{align}\label{diver-higher-2k}
&\bigg|\sum^3_{i=1}\int_{\mathbb{R}^2}\big[(\partial_{\beta}\vv\cdot\nabla)\nn_i+(\vv\cdot\nabla)\partial_{\beta}\nn_i\big]\cdot\partial_{\beta}\Delta^{2k-1}\hh_i\ud\xx\bigg|\nonumber\\
&\quad\leq\delta\int_{\mathbb{R}^2}|\nabla^{2k+2}\Fp|^2\ud\xx\nonumber\\
&\qquad +C\Big(\delta,k,\ve_1,E_0,\tau,T,\frac{T}{R^2}\Big)\Big(\int_{\mathbb{R}^2}|\nabla^{2k}\vv|^2\ud\xx+\int_{\mathbb{R}^2}|\nabla^{2k+1}\Fp|^2\ud\xx+1\Big).
\end{align}
Moreover, repeating similar arguments in (\ref{higher-deriv-estimate}) of Proposition \ref{Fp-v-4-prop} yields
\begin{align}\label{Delta-k-deriv-estimate-CHn}
-\sum^3_{i=1}\frac{1}{\chi_i}\int_{\mathbb{R}^2}|\CH^{\Delta^k}_i|^2\ud\xx\leq&-\frac{2\gamma^2}{\chi}\int_{\mathbb{R}^2}|\nabla^{2k+2}\Fp|^2\ud\xx\nonumber\\
&+C\Big(k,\ve_1,E_0,\tau,T,\frac{T}{R^2}\Big)\Big(\int_{\mathbb{R}^2}|\nabla^{2k+1}\Fp|^2\ud\xx+1\Big).
\end{align}
Thus, combining (\ref{v-energ-estimate-Delta-k}) with (\ref{Fp-energ-estimate-Delta-k})-(\ref{Delta-k-deriv-estimate-CHn}), and choosing $\delta>0$ small enough, we finally get
\begin{align*}
&\frac{\ud}{\ud t}\int_{\mathbb{R}^2}\Big(\frac{1}{2}|\nabla^{2k}\vv|^2+\CF^{\Delta^k}_{Bi}(\Fp)\Big)\ud\xx+\int_{\mathbb{R}^2}\Big(\frac{\eta}{4}|\nabla^{2k+1}\vv|^2+\frac{\gamma^2}{\chi}|\nabla^{2k+2}\Fp|^2\Big)\ud\xx\\
&\quad\leq C\Big(k,\ve_1,E_0,\tau,T,\frac{T}{R^2}\Big)\Big(\int_{\mathbb{R}^2}|\nabla^{2k}\vv|^2\ud\xx+\int_{\mathbb{R}^2}|\nabla^{2k+1}\Fp|^2\ud\xx+1\Big),
\end{align*}
which implies that (\ref{higher-regularity-ineq}) holds for $l=2k$ by Gronwall's inequality.

Lastly, we claim that the solution $(\Fp,\vv)$ is regular for all $t\in(0,T)$. In fact, for any $l>1$ and $(\xx,t)\in\mathbb{R}^2\times(0,T)$, from the estimates in Step 1 and Step 2, and the arbitrariness of $\tau\in(0,T)$, it follows that
\begin{align*}
(|\nabla^l\Fp|+|\nabla^l\vv|)(\xx,t)<\infty,\quad \Fp=(\nn_1,\nn_2,\nn_3)\in SO(3),
\end{align*}
which implies that the solution $(\Fp,\vv)$ is spatially regular. Using the frame system (\ref{new-frame-equation-n1})--(\ref{imcompressible-v}), we know that $(\partial_t\Fp,\partial_t\vv)$ is also spatially regular, and differentiating to the time $t$, thus obtain that the solution $(\Fp,\vv)$ is regular in $\mathbb{R}^2\times(0,T)$.

\end{proof}

\subsection{Global existence and proof of Theorem \ref{global-posedness-theorem}}

We are now in a position to prove the existence of global weak solutions to the frame system (\ref{new-frame-equation-n1})--(\ref{imcompressible-v}) in dimension two and to complete the proof of Theorem \ref{global-posedness-theorem}. The main strategy of the arguments will follow
several previous works \cite{Struwe,WW,Lin2}. We only present the sketch for the proof here.

For simplicity, we denote
\begin{align*}
E(t)\eqdefa E(\Fp,\vv)(t)=\int_{\mathbb{R}^2}e(\Fp,\vv)(\cdot,t)\ud\xx.
\end{align*}
According to the density properties of Sobolev spaces, for any data $\Fp^{(0)}\in H^1_{\Fp^*}\big(\mathbb{R}^2,SO(3)\big)$, there exists an approximate sequence $\{\Fp^{(0)}_m\}_{m\geq1}\in C^{\infty}\big(\mathbb{R}^2,SO(3)\big)$ with $\Fp^{(0)}_m=\Fp^{*}$ at infinity, such that
\begin{align*}
\lim_{m\rightarrow\infty}\|\Fp^{(0)}_m-\Fp^{(0)}\|_{H^1_{\Fp^*}}=0,
\end{align*}
where $\Fp^*=(\nn^*_1,\nn^*_2,\nn^*_3)\in SO(3)$ is a constant orthonormal frame. Here, we will assume  $\Fp^{(0)}_m\in H^4_{\Fp^*}\big(\mathbb{R}^2,SO(3)\big)$ (see \cite{SU,WW} for details).
Likewise, for any data $\vv^{(0)}\in L^2(\mathbb{R}^2,\mathbb{R}^2)$, there exists an approximate sequence $\{\vv^{(0)}_m\}_{m\geq1}\in C^{\infty}_c(\mathbb{R}^2,\mathbb{R}^2)$ such that
\begin{align*}
\lim_{m\rightarrow\infty}\|\vv^{(0)}_m-\vv^{(0)}\|_{L^2}=0.
\end{align*}

The absolute continuity of the integral $\int_{\mathbb{R}^2}(|\nabla\Fp^{(0)}|^2+|\vv^{(0)}|^2)\ud\xx$ implies that for any $\ve_1>0$, there exists $R_0>0$ such that
\begin{align*}
    \sup_{\xx\in\mathbb{R}^2}\int_{B_{R_0}(\xx)}(|\nabla\Fp^{(0)}|^2+|\vv^{(0)}|^2)\ud\xx\leq \ve_1.
\end{align*}
Since $(\Fp^{(0)}_m,\vv^{(0)}_m)$ strongly converges to $(\Fp^{(0)},\vv^{(0)})$ in $H^1_{\Fp^*}\times L^2$, there holds
\begin{align}\label{strong-con-resut}
\sup_{\xx\in\mathbb{R}^2}\int_{B_{R_0}(\xx)}(|\nabla\Fp^{(0)}_m|^2+|\vv^{(0)}_m|^2)\ud\xx\leq 2\ve_1,\quad \forall m\gg1.
\end{align}
We here assume (\ref{strong-con-resut}) holds for all $m\geq1$.

Theorem \ref{local-posedness-theorem} tells us that if  $(\nabla\Fp^{(0)}_m,\vv^{(0)}_m)\in H^4(\mathbb{R}^2)\times H^4(\mathbb{R}^2)
$ is the given initial data satisfying $\nabla\cdot\vv^{(0)}_m=0$, then there exists a time $T^m>0$ and a smooth solution $(\nabla\Fp_m,\vv_m)$
with the pressure $p_m$
such that
\begin{align*}
\nabla\Fp_m\in C\big([0,T^m];H^{4}(\mathbb{R}^2)\big),\quad \vv_m\in C\big([0,T^m];H^{4}(\mathbb{R}^2)\big)\cap L^2\big(0,T^m;H^{5}(\mathbb{R}^2)\big).
\end{align*}
Thus, for $R\leq \frac{R_0}{2}$, there exists $T^m_0\leq T^m$ such that
\begin{align*}
\sup_{\xx\in\mathbb{R}^2,0<t<T^m_0}\int_{B_R(\xx)}(|\nabla\Fp_m|^2+|\vv_m|^2)(\yy,t)\ud\yy\leq 4\ve_1.
\end{align*}
In addition, it can be seen from Proposition \ref{local-monotonic-ineq-prop} that $T^m_0\geq\frac{\ve^2_1R^2_0}{4C^2_2E^2_0}=T_0>0$ uniformly. We know from Proposition \ref{higher-regularity-l-prop} that for any $\tau\in(0,T_0)$ and $l\geq 1$, there holds
\begin{align}\label{higher-regularity-ineq-sup}
&\sup_{t\in(\tau,T_0)}\int_{\mathbb{R}^2}\big(|\nabla^{l+1}\Fp_m|^2+|\nabla^l\vv_m|^2\big)(\cdot,t)\ud\xx+\int^{T_0}_{\tau}\int_{\mathbb{R}^2}\big(|\nabla^{l+2}\Fp_m|^2+|\nabla^{l+1}\vv_m|^2\big)(\cdot,s)\ud\xx\ud s\nonumber\\
&\quad\leq C\Big(l,\ve_1,E_0,\tau,T,\frac{T}{R^2}\Big).
\end{align}
Furthermore, using Proposition \ref{energ-diss-prop}, Proposition \ref{Fp-v-4-prop}, and the frame equations (\ref{new-frame-equation-n1})--(\ref{new-frame-equation-n3}), we can deduce that $E\big(\Fp_m,\vv_m\big)(t)\leq E_0$, and for any $t\in[0,T^m_0]$,
\begin{align}\label{higher-derivative-total-C}
\int_{\mathbb{R}^2\times[0,T^m_0]}\Big(|\nabla^2\Fp_m|^2+|\nabla\vv_m|^2+|\partial_t\Fp_m|^2+|\nabla\Fp_m|^4+|\vv_m|^4\Big)\ud\xx\ud t\leq C(\ve_1,C_2,E_0).
\end{align}

Using the equation of $\vv$  in (\ref{frame-equation-v}), the pressure term $p$ in the distributional sense can be expressed by
\begin{align*}
\Delta p=\partial_i\partial_j\big(\sigma_{ij}-v_iv_j\big)+\partial_i\mathfrak{F}_i\quad \text{on}~\mathbb{R}^2\times[0,T^m_0],
\end{align*}
which implies from Calder$\acute{\rm o}$n--Zygmund estimates that
\begin{align}\label{pm-L2-estimate}
&\int_{\mathbb{R}^2\times[0,T^m_0]}|p_m|^2\ud\xx\ud t\nonumber\\
&\quad\leq C\int_{\mathbb{R}^2\times[0,T^m_0]}\big(|\vv_m|^4+|\nabla\Fp_m|^4+|\nabla\vv_m|^2+|\partial_t\Fp_m|^2+|\nabla^2\Fp_m|^2\big)\ud\xx\ud t\nonumber\\
&\quad\leq C(\ve_1,C_2,E_0).
\end{align}
It remains to estimate the term $\partial_t\vv_m$. For any $\phi\in C^{\infty}_c\big(\mathbb{R}^2\times(0,T^m_0),\mathbb{R}^2\big)$, applying the equation (\ref{frame-equation-v}) and integrating by parts, we obtain
\begin{align*}
&\int_{\mathbb{R}^2\times[0,T^m_0]}\partial_t\vv_m\cdot\phi\ud\xx\ud t\\
&\quad=\int_{\mathbb{R}^2\times[0,T^m_0]}\Big(\big(\vv_m\otimes\vv_m-\eta\nabla\vv_m-\sigma_m\big)\cdot\nabla\phi+\mathfrak{F}_m\cdot\phi+p_m{\rm div}\phi\Big)\ud\xx\ud t\\
&\quad\leq C\big\||\vv_m|^2+|\nabla\Fp_m|^2+|\nabla\vv_m|+|p_m|+|\partial_t\Fp_m|+|\nabla^2\Fp_m|\big\|^2_{L^2(\mathbb{R}^2\times[0,T^m_0])}\|\phi\|_{L^2_tH^1_x}\\
&\quad \leq C(\ve_1,C_2,E_0)\|\phi\|_{L^2_tH^1_x},
\end{align*}
which further gives
\begin{align}\label{partial-t-vm-L2}
\|\partial_t\vv_m\|_{L^2(0,T^m_0),H^{-1}(\mathbb{R}^2)}\leq C(\ve_1,C_2,E_0).
\end{align}

Therefore, (\ref{higher-regularity-ineq-sup})--(\ref{partial-t-vm-L2}) together with the Aubin--Lions Lemma implies that after passing to possible subsequences, there exists a solution $(\Fp-\Fp^*,\vv)\in W^{2,1}_2(\mathbb{R}^2\times[0,T_0],SO(3))\times W^{1,0}_2(\mathbb{R}^2\times[0,T_0],\mathbb{R}^2)$ with the pressure $p$ such that
\begin{align*}
&\Fp_m-\Fp^*\rightarrow \Fp-\Fp^*,~~\text{weakly~in}~W^{2,1}_2(\mathbb{R}^2\times[0,T_0],\mathbb{R}^2);\\
&\vv_m\rightarrow\vv,~~\text{weakly~in}~W^{1,0}_2(\mathbb{R}^2\times[0,T_0],\mathbb{R}^2);\\
&p_m\rightarrow p,~~\text{weakly~in}~L^2(\mathbb{R}^2\times[0,T_0],\mathbb{R}).
\end{align*}
It follows from (\ref{higher-derivative-total-C}) and (\ref{partial-t-vm-L2}) that we can assume
\begin{align*}
(\nabla\Fp,\vv)(\cdot,t)\rightarrow (\nabla\Fp^{(0)},\vv^{(0)})~\text{weakly~in}~L^2(\mathbb{R}^2),
\end{align*}
as $t\rightarrow 0$. In particular, we have
\begin{align*}
E(0)\leq\lim\inf_{t\rightarrow 0}E(t).
\end{align*}
On the other hand, the energy estimates of $(\Fp_m,\vv_m)$ imply
\begin{align*}
    E(0)\geq\lim\sup_{t\rightarrow0} E(t).
\end{align*}
Hence, we conclude that $(\nabla\Fp,\vv)(\cdot,t)$ converges to $(\nabla\Fp^{(0)},\vv^{(0)})$ strongly in $L^2(\mathbb{R}^2)$ and $(\Fp,\vv)$ is the solution to the frame system (\ref{new-frame-equation-n1})--(\ref{imcompressible-v}) with the given data $(\Fp^{(0)},\vv^{(0)})$.

By taking the weak limit of the regular estimates (\ref{higher-regularity-ineq-sup}), it is clear that $(\Fp,\vv)\in C^{\infty}(\mathbb{R}^2\times(0,T_0])$ and $(\nabla^{l+1}\Fp,\nabla^l\vv)(\cdot,T_0)\in L^2(\mathbb{R}^2)$ for any $l\geq1$. However, we know from Theorem \ref{local-posedness-theorem} that there exists a unique smooth solution to the frame system (\ref{new-frame-equation-n1})--(\ref{imcompressible-v}) with the initial data $(\Fp,\vv)(\cdot,T_0)$, still denoted as $(\Fp,\vv)$. Moreover, the blow-up criterion implies that if $(\Fp,\vv)$ blows up at the time $T^*$, then
\begin{align*}
    \|\nabla\times\vv(t)\|_{L^{\infty}(\mathbb{R}^2)}+\|\nabla\Fp(t)\|^2_{L^{\infty}(\mathbb{R}^2)}\rightarrow \infty,~~\text{as}~t\rightarrow T^*,
\end{align*}
which further implies
\begin{align}\label{nabla-4-fp+v-blow-up}
(|\nabla^4\Fp|+|\nabla^3\vv|)(\xx,t)\notin L^{\infty}_tL^2_{\xx}(\mathbb{R}^2\times(T_0,T^*)).
\end{align}

Let $T_1\in(T_0,+\infty)$ be the first singular time of $(\Fp,\vv)$, that is,
\begin{align*}
(\Fp,\vv)\in C^{\infty}\big(\mathbb{R}^2\times(0,T_1),SO(3)\times\mathbb{R}^2\big),
~~\text{but}~~
(\Fp,\vv)\notin C^{\infty}\big(\mathbb{R}^2\times(0,T_1],SO(3)\times\mathbb{R}^2\big).
\end{align*}
Then, from Proposition \ref{local-monotonic-ineq-prop} and (\ref{nabla-4-fp+v-blow-up}), we have
\begin{align*}
\lim\sup_{t\uparrow T_1}\max_{\xx\in\mathbb{R}^2}\int_{B_R(\xx)}(|\nabla\Fp|^2+|\vv|^2)(\cdot,t)\ud\xx\geq\ve_1,~~\forall R>0.
\end{align*}

Now we look for an eternal extension of this weak solution in time. By an argument similar to the proof (\ref{partial-t-vm-L2})(also see \cite{Lin2} for the similar derivation), using the system (\ref{new-frame-equation-n1})--(\ref{imcompressible-v}) together with the Aubin--Lions Lemma and the Sobolev embedding theorem, we deduce that
$
(\Fp-\Fp^*,\vv)\in C^0([0,T_1], L^2(\mathbb{R}^2)).
$
Then, we define
\begin{align*}
\big(\Fp(T_1)-\Fp^*,\vv(T_1)\big)\eqdefa\lim_{t\uparrow T_1}\big(\Fp(t)-\Fp^*,\vv(t)\big),~~\text{in}~L^2(\mathbb{R}^2).
\end{align*}
By the energy inequality, we have $\nabla\Fp\in L^{\infty}(0,T_1;L^2(\mathbb{R}^2))$, and thus $\nabla\Fp(t)\rightarrow\nabla\Fp(T_1)$ weakly in $L^2(\mathbb{R}^2)$. In the same way, we can extend the time $T_1$ to $T_2$ and so on. It can be proven that at any singular time $T_i(i\geq1)$, there exists a loss of energy of at least $\ve_1$, and then the number of finite singular times must be bounded by $L$, and for $1\leq i\leq L$ there holds
\begin{align*}
\lim\sup_{t\uparrow T_i}\max_{\xx\in\mathbb{R}^2}\int_{B_R(\xx)}(|\nabla\Fp|^2+|\vv|^2)(\cdot,t)\ud\xx\geq \ve_1,\quad\forall R>0.
\end{align*}
Therefore, if we use $(\Fp(T_L),\vv(T_L))$  with the given initial data $(\Fp^{(0)},\vv^{(0)})$ to construct a weak solution $(\Fp,\vv)$ to the system (\ref{new-frame-equation-n1})--(\ref{imcompressible-v}) as before, then $(\Fp,\vv)$ is just an eternal weak solution that we seek.
In conclusion, Theorem \ref{global-posedness-theorem}
is established.

\appendix
\section{The proof of (\ref{higher-regularity-ineq}) for $l=2$}

We here give the argument of (\ref{higher-regularity-ineq}) for the case $l=2$. First note that
\begin{align*}
&|\nabla\nn_i|^2=|\nn_i\cdot\Delta\nn_i|\leq |\nabla^2\nn_i|,~~i=1,2,3,\\
&|\nabla\mathfrak{F}|\leq C(|\nabla^3\Fp||\nabla\Fp|+|\nabla^2\Fp|^2+|\nabla^2\Fp||\nabla\Fp|^2+|\nabla\Fp|^4).
\end{align*}
Multiplying the equation (\ref{frame-equation-v}) by $\Delta^2\vv$, we obtain
\begin{align}\label{grad-2-L2-estimate}
&\frac{1}{2}\frac{\ud}{\ud t}\int_{\mathbb{R}^2}|\nabla^2\vv|^2\ud\xx+\eta\int_{\mathbb{R}^2}|\nabla\Delta\vv|^2\ud\xx\nonumber\\
&\quad=-\langle(\vv\cdot\nabla)\vv\cdot\Delta^2\vv\rangle+\langle\nabla\cdot\sigma,\Delta^2\vv\rangle+\langle\mathfrak{F},\Delta^2\vv\rangle\nonumber\\
&\quad\leq\frac{\eta}{4}\int_{\mathbb{R}^2}|\nabla\Delta\vv|^2\ud\xx+C\int_{\mathbb{R}^2}|\nabla(\vv\cdot\nabla\vv)|^2\ud\xx+C\int_{\mathbb{R}^3}(|\nabla^3\Fp|^2|\nabla\Fp|^2+|\nabla^2\Fp|^4)\ud\xx\nonumber\\
&\qquad+\langle\nabla\cdot\sigma,\Delta^2\vv\rangle.
\end{align}
By a direct calculation, we get
\begin{align*}
&\langle\nabla\cdot\sigma,\Delta^2\vv\rangle=-\langle\sigma,\Delta^2(\A+\BOm)\rangle\\
&\quad\leq-\Big(\frac{\eta_3}{\chi_3}\int_{\mathbb{R}^2}(\Delta\A\cdot\sss_3)\CH^{\Delta}_3\ud\xx+\frac{\eta_2}{\chi_2}\int_{\mathbb{R}^2}(\Delta\A\cdot\sss_4)\CH^{\Delta}_2\ud\xx+\frac{\eta_1}{\chi_1}\int_{\mathbb{R}^2}(\Delta\A\cdot\sss_5)\CH^{\Delta}_1\ud\xx\Big)\\
&\qquad-\frac{1}{2}\int_{\mathbb{R}^2}\Big((\Delta\BOm\cdot\aaa_1)\CH^{\Delta}_3+(\Delta\BOm\cdot\aaa_2)\CH^{\Delta}_2+(\Delta\BOm\cdot\aaa_3)\CH^{\Delta}_1\Big)\ud\xx\\
&\qquad+C\int_{\mathbb{R}^2}|\nabla^3\vv|\big(|\nabla^2\Fp||\nabla\vv|+|\nabla^2\Fp|^2+|\nabla^2\vv||\nabla\Fp|+|\nabla^3\Fp||\nabla\Fp|\big)\ud\xx,
\end{align*}
where
\begin{align*}
\CH^{\Delta}_1=\nn_2\cdot\Delta\hh_3-\nn_3\cdot\Delta\hh_2,\quad
\CH^{\Delta}_2=\nn_3\cdot\Delta\hh_1-\nn_1\cdot\Delta\hh_3,\quad
\CH^{\Delta}_3=\nn_1\cdot\Delta\hh_2-\nn_2\cdot\Delta\hh_1.
\end{align*}
Then, (\ref{grad-2-L2-estimate}) together with the above two estimates yields
\begin{align}\label{grad-2-L2-estimate-final}
&\frac{1}{2}\frac{\ud}{\ud t}\int_{\mathbb{R}^2}|\nabla^2\vv|^2\ud\xx+\frac{\eta}{2}\int_{\mathbb{R}^2}|\nabla\Delta\vv|^2\ud\xx\nonumber\\
&\quad\leq-\Big(\frac{\eta_3}{\chi_3}\int_{\mathbb{R}^2}(\Delta\A\cdot\sss_3)\CH^{\Delta}_3\ud\xx+\frac{\eta_2}{\chi_2}\int_{\mathbb{R}^2}(\Delta\A\cdot\sss_4)\CH^{\Delta}_2\ud\xx+\frac{\eta_1}{\chi_1}\int_{\mathbb{R}^2}(\Delta\A\cdot\sss_5)\CH^{\Delta}_1\ud\xx\Big)\nonumber\\
&\qquad-\frac{1}{2}\int_{\mathbb{R}^2}\Big((\Delta\BOm\cdot\aaa_1)\CH^{\Delta}_3+(\Delta\BOm\cdot\aaa_2)\CH^{\Delta}_2+(\Delta\BOm\cdot\aaa_3)\CH^{\Delta}_1\Big)\ud\xx\nonumber\\
&\qquad+C\int_{\mathbb{R}^2}(|\nabla\vv|^4+|\nabla^2\Fp|^4)\ud\xx+C\int_{\mathbb{R}^2}(|\nabla^3\Fp|^2+|\nabla^2\vv|^2)(|\nabla\Fp|^2+|\vv|^2)\ud\xx.
\end{align}

On the other hand, we define
\begin{align*}
\CF^{\Delta}_{Bi}(\Fp)\eqdefa&\frac{1}{2}\sum^3_{i=1}\Big(\gamma_i\|\nabla\Delta\nn_i\|^2_{L^2}+k_i\|\Delta{\rm div}\nn_i\|^2_{L^2}+\sum^3_{j=1}k_{ji}\|\Delta(\nn_j\cdot\nabla\times\nn_i)\|^2_{L^2}\Big).
\end{align*}
Note that for $i=1,2,3$, the following estimates hold,
\begin{align*}
|\Delta\hh_i|\leq& C(|\nabla^4\Fp|+|\nabla^3\Fp||\nabla\Fp|+|\nabla^2\Fp|^2),\\
|\nn_i\cdot\Delta^2\nn_i|\leq& C(|\nabla\Fp||\nabla^3\Fp|+|\nabla^2\Fp|^2),\\
|\nabla^2\partial_t\Fp|^2\leq& \max\{1,\chi^{-1}_i\}(|\nabla^3\vv|^2+|\nabla^4\Fp|^2)+C|\nabla^2\vv|^2|\nabla\Fp|^2\\
&+C|\nabla^2\Fp|^2(|\nabla\vv|^2+|\nabla^2\Fp|^2).
\end{align*}
Acting the differential operator $\partial_{\beta}$ on the equations (\ref{new-frame-equation-n1})--(\ref{new-frame-equation-n3}), multiplying them by $\partial_{\beta}\Delta\hh_i(i=1,2,3)$, respectively, and integrating by parts, we derive that
\begin{align}\label{energy-CF-Delta}
&\frac{\ud}{\ud t}\CF^{\Delta}_{Bi}(\Fp)+\sum^3_{i=1}\int_{\mathbb{R}^2}\big[(\partial_{\beta}\vv\cdot\nabla)\nn_i+(\vv\cdot\nabla)\partial_{\beta}\nn_i\big]\cdot\partial_{\beta}\Delta\hh_i\ud\xx\nonumber\\
&\quad\leq\int_{\mathbb{R}^2}\Big(\frac{1}{2}\Delta\BOm\cdot\aaa_1+\frac{\eta_3}{\chi_3}\Delta\A\cdot\sss_3-\frac{1}{\chi_3}\CH^{\Delta}_3\Big)\CH^{\Delta}_3\ud\xx\nonumber\\
&\qquad+\int_{\mathbb{R}^2}\Big(\frac{1}{2}\Delta\BOm\cdot\aaa_2+\frac{\eta_2}{\chi_2}\Delta\A\cdot\sss_4-\frac{1}{\chi_2}\CH^{\Delta}_2\Big)\CH^{\Delta}_2\ud\xx\nonumber\\
&\qquad+\int_{\mathbb{R}^2}\Big(\frac{1}{2}\Delta\BOm\cdot\aaa_3+\frac{\eta_1}{\chi_1}\Delta\A\cdot\sss_5-\frac{1}{\chi_1}\CH^{\Delta}_1\Big)\CH^{\Delta}_1\ud\xx\nonumber\\
&\qquad+\delta\sum^3_{i=1}\int_{\mathbb{R}^2}|\nabla^2\partial_t\nn_i|^2\ud\xx+C_{\delta}\int_{\mathbb{R}^2}(|\nabla^2\Fp|^4+|\nabla^3\Fp|^2|\nabla\Fp|^2+|\nabla^2\Fp|^2|\nabla\Fp|^4)\ud\xx\nonumber\\
&\qquad+\frac{\gamma^2}{2\chi}\int_{\mathbb{R}^2}|\nabla^4\Fp|^2\ud\xx+C\int_{\mathbb{R}^2}(|\nabla\vv|^4+|\nabla^2\Fp|^4)\ud\xx\nonumber\\
&\qquad+C\int_{\mathbb{R}^2}(|\nabla^3\Fp|^2+|\nabla^2\vv|^2)(|\nabla\Fp|^2+|\vv|^2)\ud\xx,
\end{align}
where $\delta>0$ will be determined later and $\chi^{-1}=\min\{\chi^{-1}_1,\chi^{-1}_2,\chi^{-1}_3\}$.

A direct calculation shows that
\begin{align*}
&\Big|\sum^3_{i=1}\int_{\mathbb{R}^2}\big[(\partial_{\beta}\vv\cdot\nabla)\nn_i+(\vv\cdot\nabla)\partial_{\beta}\nn_i\big]\cdot\partial_{\beta}\Delta\hh_i\ud\xx\Big|\\
&\quad+\frac{\gamma^2}{2\chi}\int_{\mathbb{R}^2}|\nabla^4\Fp|^2\ud\xx+C\int_{\mathbb{R}^2}(|\nabla\vv|^4+|\nabla^2\Fp|^4)\ud\xx\nonumber\\
&\quad+C\int_{\mathbb{R}^2}(|\nabla^3\Fp|^2+|\nabla^2\vv|^2)(|\nabla\Fp|^2+|\vv|^2)\ud\xx.
\end{align*}
By an estimate similar to (\ref{higher-deriv-estimate}), we have
\begin{align*}
-\sum^3_{i=1}\frac{1}{\chi_i}\int_{\mathbb{R}^2}|\CH^{\Delta}_i|^2\ud\xx\leq-\frac{2\gamma^2}{\chi}\int_{\mathbb{R}^2}|\nabla^4\Fp|^2\ud\xx+C\int_{\mathbb{R}^2}(|\nabla\Fp|^2|\nabla^3\Fp|^2+|\nabla^2\Fp|^4)\ud\xx.
\end{align*}
Consequently, combining (\ref{grad-2-L2-estimate-final}) and (\ref{energy-CF-Delta}) together with the above estimates, and choosing $\delta>0$ small enough, and using the Gagliardo--Nirenberg--Sobolev inequality, we arrive at
\begin{align*}
&\frac{\ud}{\ud t}\Big(\frac{1}{2}|\nabla^2\vv|^2+\CF^{\Delta}_{Bi}(\Fp)\Big)+\int_{\mathbb{R}^2}\Big(\frac{\eta}{4}|\nabla^3\vv|^2+\frac{\gamma^2}{\chi}|\nabla^4\Fp|^2\Big)\ud\xx\\
&\quad\leq C\int_{\mathbb{R}^2}(|\nabla\vv|^4+|\nabla^2\Fp|^4)\ud\xx+C\int_{\mathbb{R}^2}(|\nabla^3\Fp|^2+|\nabla^2\vv|^2)(|\nabla\Fp|^2+|\vv|^2)\ud\xx\\
&\quad\leq C\Big(\int_{\mathbb{R}^2}(|\vv|^4+|\nabla\Fp|^4)\ud\xx\Big)^{\frac{1}{2}}\Big(\int_{\mathbb{R}^2}(|\nabla^2\vv|^4+|\nabla^3\Fp|^4)\ud\xx\Big)^{\frac{1}{2}}+C\int_{\mathbb{R}^2}(|\nabla\vv|^4+|\nabla^2\Fp|^4)\ud\xx\\
&\quad\leq\min\Big\{\frac{\eta}{8},\frac{\gamma^2}{2\chi}\Big\}\int_{\mathbb{R}^2}(|\nabla^3\vv|^2+|\nabla^4\Fp|^2)\ud\xx\\
&\qquad+C\Big(\int_{\mathbb{R}^2}(|\vv|^4+|\nabla\Fp|^4+|\nabla\vv|^2+|\nabla^2\Fp|^2)\ud\xx\Big)\Big(\int_{\mathbb{R}^2}(|\nabla^2\vv|^2+|\nabla^3\Fp|^2)\ud\xx\Big).
\end{align*}
Thus, by Gronwall's inequality and Proposition \ref{Fp-v-4-prop}, from (\ref{higher-regular-l=1}), it follows that
\begin{align}\label{higher-regular-l=2}
&\int_{\mathbb{R}^2}(|\nabla^3\Fp|^2+|\nabla^2\vv|^2)(\cdot,t)\ud\xx\nonumber\\
&\quad+\int^t_{\tau}\int_{\mathbb{R}^2}(|\nabla^4\Fp|^2+|\nabla^3\vv|^2)(\cdot,s)\ud\xx\ud s\leq C\Big(\ve_1,E_0,\tau,T,\frac{T}{R^2}\Big),
\end{align}
for all $t\in[\tau,T]$ with $\tau\in (0,T)$.

\bigskip
\noindent{\bf Acknowledgments.}
 Sirui Li is partially supported by the NSFC under grant No. 12061019 and by the Growth Foundation for Youth Science and Technology Talent of Educational Commission of Guizhou Province of China under grant No. [2021]087. Jie Xu is partially supported by the NSFC under grant No. 12001524.

\end{document}